\documentclass[numbers=noendperiod]{filbar-article}

\usepackage{filbar-all}

\usepackage[color={filbar theme color!20},textsize=scriptsize,tickmarkheight=2pt,colorinlistoftodos]{todonotes}

\title{Fibredness of 3-manifolds is in NP}
\author{Filippo Baroni}
\date{}

\FilBarDefineThemeColor{HTML}{009895}
\AtBeginDocument{
}

\zcRefTypeSetup{propertygroup}{
    Name-sg=Properties,
    Name-pl=Properties,
    name-sg=properties,
    name-pl=properties,
    cap
}

\usetikzlibrary{arrows.meta,backgrounds,calc,math,quotes}

\NewDocumentCommand{\myfiguresource}{O{}m}{%
    \includegraphics[#1]{figures/#2.pdf}
}

\NewDocumentCommand{\dist}{O{\CCC(F)}}{\operatorname{dist}_{#1}}

\zcRefTypeSetup{certificate}{
    Name-sg=Certificate,
    Name-pl=Certificates,
    name-sg=certificate,
    name-pl=certificates,
    cap
}
\NewDocumentCommand{\cert}{so}{\mathfrak{S}\IfBooleanT{#1}{^*}\IfValueT{#2}{_{\mathrm{#2}}}}

\RenewDocumentCommand{\v}{m}{\mathbf{v}_{#1}}
\NewDocumentCommand{\hboundary}{}{\partial_h}
\NewDocumentCommand{\vboundary}{}{\partial_v}
\NewDocumentCommand{\abstr}{m}{\operatorname{ab}(#1)}
\NewDocumentCommand{\embedd}{o}{\operatorname{emb}\IfValueT{#1}{_{#1}}}
\NewDocumentCommand{\thick}{m}{\operatorname{thick}(#1)}
\NewDocumentCommand{\thin}{m}{\operatorname{thin}(#1)}
\NewDocumentCommand{\twistnumber}{o}{\operatorname{twist}\IfValueT{#1}{_{#1}}}

\NewDocumentCommand{\length}{m}{\ell(#1)}
\NewDocumentCommand{\pospushoff}{m}{#1^+}
\NewDocumentCommand{\negpushoff}{m}{#1^-}
\NewDocumentCommand{\posnbhd}{m}{\NNN^+(#1)}
\NewDocumentCommand{\negnbhd}{m}{\NNN^-(#1)}

\NewDocumentCommand{\area}{m}{\operatorname{area}(#1)}
\NewDocumentCommand{\transfermap}{}{\bigtriangleup}

\begin{document}

\maketitle{}
\thispagestyle{empty}

\begin{abstract}
We show that the problem of deciding whether a compact orientable $3$\=/manifold fibres over the circle lies in the complexity class \NP{}.
\end{abstract}

\tableofcontents

\section{Introduction}
\label{sec:introduction}

A compact connected orientable $3$\=/manifold $M$ is \emph{fibred} if it fibres over the circle; the fibres of such a fibration are compact orientable surfaces.
When $M$ is irreducible, a classical theorem of \citeauthor{stallings:fibering-certain-3} (see \cite[Theorem~2]{stallings:fibering-certain-3}) shows that this is equivalent to the existence of a surjective homomorphism from $\pi_1(M)$ onto $\ZZ$ with finitely generated kernel.
More recently, the resolution \cite{agol:virtual-haken-conjecture} of Thurston's virtual fibring conjecture \cite[Question~18]{thurston:three-dimensional-manifolds} has shown that every closed hyperbolic $3$\=/manifold admits a finite\=/degree cover that fibres over the circle.
This striking result raises the fundamental question: how can we effectively detect if a $3$\=/manifold is fibred?
The aim of the present article is to answer this question as follows (see \zcref{thm:fibredness detection is in NP}).

\begin{theoremintro}\label{thm:intro:fibredness is in np}
The problem of deciding whether a given compact connected orientable $3$\=/manifold is fibred lies in \NP{}.
\end{theoremintro}

Recall that a decision problem lies in the complexity class \NP{} if, whenever the answer to the problem is ``yes'', there is a certificate of this fact that can be verified in time polynomial in the size of the input; for a formal definition of the complexity class \NP{}, we refer the reader to \cite[Chapter~2]{arora-barak:computational-complexity}.
In the setting of \zcref{thm:intro:fibredness is in np} -- as is customary in algorithmic $3$\=/dimensional topology -- the input $3$\=/manifold is given as a triangulation $\TTT$, and the size of the input is simply the number of tetrahedra in $\TTT$.

Several precursors of \zcref{thm:intro:fibredness is in np} exist in the literature:
\begin{enumarabic}
\item the problem of recognising $S^2\times S^1$ and $D^2\times S^1$ lies in \NP{} by \cite[Theorem~3]{ivanov:computational-complexity-basic};
\item the problem of recognising $S^1\times[0,1]\times S^1$ lies in \NP{} by \cite[Theorem~12.1]{lackenby:efficient-certification-knottedness};
\item the problem of deciding whether a given closed orientable irreducible atoroidal $3$\=/manifold fibres over the circle lies in \NP{} by \cite[Corollary~1.4]{schleimer:sphere-recognition-lies}.
\end{enumarabic}

The last result cited -- a theorem of \citeauthor{schleimer:sphere-recognition-lies} -- is quite close to our \zcref{thm:intro:fibredness is in np}, the only difference being that it requires the input $3$\=/manifold $M$ to be closed, irreducible, and atoroidal.
Note that all orientable fibred $3$\=/manifolds (with the exception of $S^2\times S^1$) are irreducible, and moreover irreducibility can be certified in polynomial time thanks to a later result of \citeauthor{lackenby:efficient-certification-knottedness} \cite[Theorem~1.6]{lackenby:efficient-certification-knottedness}.
Therefore, the requirement that $M$ be irreducible can now be omitted from the statement of \citeauthor{schleimer:sphere-recognition-lies}'s theorem.
On the other hand, closedness and atoroidality impose genuine restrictions on the class of $3$\=/manifolds to which \cite[Corollary~1.4]{schleimer:sphere-recognition-lies} applies.
We also remark that, while it is trivial to check whether a triangulated $3$\=/manifold is closed, it is not known whether atoroidality can be certified in polynomial time; in our upcoming work \cite{baroni:certifying-hyperbolicity-fibred}, we will prove some results in this direction -- specifically for fibred $3$\=/manifolds -- building on the content of this article.
In fact, in addition to being suitable for $3$\=/manifolds with non\=/empty boundary or with essential tori, the main advantage of the certificate we employ to prove \zcref{thm:intro:fibredness is in np} is that it can be used to recover the \emph{monodromy} of the fibration.
This result, however, is beyond the scope of this article; we defer precise statements and proofs to \cite{baroni:certifying-hyperbolicity-fibred}.

A closed orientable $3$\=/manifold $M$ is fibred if and only if it contains a closed embedded surface $F$ such that cutting $M$ along $F$ yields a $3$\=/manifold $M\cut F$ homeomorphic to $F\times[0,1]$; a similar statement holds when $M$ has non\=/empty boundary, if we require that the homeomorphism $\umap{M\cut F}{F\times[0,1]}$ sends the original boundary of $M$ onto $\boundary F\times[0,1]$.
One could then imagine proving \zcref{thm:intro:fibredness is in np} by providing the surface $F\subs M$ as a certificate of fibredness; the verifier would simply have to check that the $3$\=/manifold $M\cut F$ is a product interval bundle over $F$.
There is, however, a crucial obstacle to this approach, as illustrated by the following example.
Denote by $M_n$ the exterior of the $(f_n,f_{n+1})$-torus knot in the $3$\=/sphere, where $f_n$ is the $n$\=/th Fibonacci number.
The $3$\=/manifold $M_n$ admits a unique fibration, where the fibre $F_n$ has Euler characteristic $1-(f_n-1)\cdot(f_{n+1}-1)$, whose absolute value is exponential in $n$.
However, one can triangulate $M_n$ with a number of tetrahedra that is linear in $n$ (see \cite[Theorem~1.3]{fominykh-wiest:upper-bounds-complexity}).
Therefore, in this setting, even constructing the $3$\=/manifold $M_n\cut F_n$ would break the polynomial\=/time requirement.

Instead, our proof of \zcref{thm:intro:fibredness is in np} involves finding a certificate for fibredness whose verification does not require cutting along a fibre.
This is achieved by exploiting a general property of \emph{normal surfaces} (see \zcref{sec:normal surfaces}).
If $F$ is a normal surface in a triangulated $3$\=/manifold $M$, then the $3$\=/manifold $M\cut F$ is naturally decomposed into two codimension\=/$0$ submanifolds: the ``parallelity bundle'', which is guaranteed to be an interval bundle, and the ``guts'', which can be triangulated with few tetrahedra.
We then prove that, if $F$ is a \emph{least\=/weight} fibre of $M$, then checking whether $M\cut F$ is a product interval bundle becomes substantially easier.
In particular, it suffices to focus on the guts and on ``small'' components of the parallelity bundle, completely ignoring the components of high genus.
The technical statement that makes this idea precise is given in \zcref{thm:parallelity bundle of a least-weight fibre}; the main technical tools used therein are \emph{annular simplifications}, as first described in \cite{lackenby-purcell:triangulation-complexity-fibred}.
A more detailed, but still informal, overview of the certificate is given in \zcref{sec:overview of the certificate}, after all the necessary notation has been set up.

The material in this article is drawn from Chapters 2 and 3 of the author's PhD thesis \cite{baroni:certifying-hyperbolicity-fibreda}.
The content of Chapters 4 and 5, concerning the certification of hyperbolicity for fibred $3$\=/manifolds, will appear in the upcoming article \cite{baroni:certifying-hyperbolicity-fibred}.
The results therein will rely heavily on specific properties of the certificate that we construct here.
For this reason, while we have made every effort to keep this article self\=/contained, we will include a handful of statements that are not strictly necessary for our proof of \zcref{thm:intro:fibredness is in np}, making sure to flag them appropriately.

\paragraph{Conventions and notation.}
\begin{itemize}[beginpenalty=10000]
\item Logarithms are taken with base $2$.
\item  For a finite set $X$, we denote by $\abs{X}$ its cardinality.
\item When $X$ is a topological space, the notation $\abs{X}$ refers instead to the number of connected components of $X$.
\item We always work in the PL category.
\item If $Y$ is a subspace of a topological space $X$, then we denote by $\closure{Y}$ the closure of $Y$ in $X$; the ambient space $X$ will always be clear from context.
\item By $\interior{Y}$ we always mean $Y\setminus\boundary Y$ (which is not necessarily equal to the interior of $Y$ in $X$).
\end{itemize}

\section{Surfaces in triangulated 3-manifolds}

\subsection{Normal surfaces}
\label{sec:normal surfaces}

The most convenient way to describe a compact $3$\=/manifold for algorithmic purposes is by means of a \emph{triangulation}.
A triangulation $\TTT$ is a purely combinatorial object, consisting of a finite list of vertices, edges, triangles, and tetrahedra.
Each $i$\=/simplex is equipped with information about its attaching map to the $(i-1)$\=/skeleton.
The only requirement we ask of $\TTT$ is that its topological realisation is a compact $3$\=/manifold $M$.
We remark that, with this combinatorial description, each tetrahedron is endowed with an orientation (up to arbitrarily and universally fixing an orientation of the ``standard'' tetrahedron).
We say that the triangulation is \emph{oriented} if the orientations of its tetrahedra patch together consistently to define a global orientation of the $3$\=/manifold $M$.
By ``triangulation of an oriented $3$\=/manifold $M$'', we mean an oriented triangulation of $M$ whose orientation agrees with that of $M$.
A natural measure of complexity of a triangulation $\TTT$ is the number $\card{\TTT}$ of tetrahedra it contains, which we also call the \emph{size} of $\TTT$.
The complexity of an algorithm that takes as input a triangulated $3$\=/manifold will always be measured against this quantity.

Let $\TTT$ be a triangulation of a compact $3$\=/manifold $M$.
We say that a surface $F$ embedded in $M$ is \emph{in general position} if it is disjoint from the vertices of $\TTT$ and transverse to the edges and triangles of $\TTT$.
For such a surface $F$, we define the \emph{weight} of $F$ to be the number
\[
w(F)=\lvert F\cap\TTT^{(1)}\rvert
\]
of points in the intersection of $F$ with the $1$\=/skeleton $\TTT^{(1)}$ of $\TTT$.

Most ``interesting'' surfaces in $M$ can be described combinatorially in the framework of normal surfaces.
We refer the reader to \cite[Section~3.3]{matveev:algorithmic-topology-classification} for a comprehensive exposition on normal surfaces, of which we offer a summary below.
A \emph{normal disc} is a disc properly embedded in a tetrahedron $T$ of $\TTT$ and disjoint from the vertices of $T$ such that its boundary crosses at least one edge of $T$, and no edge of $T$ more than once.
It is not hard to see that, up to \emph{normal isotopy} -- that is, up to isotopy that fixes the vertices and preserves the $1$\=/skeleton and the $2$\=/skeleton -- there are exactly seven types of normal disc in a tetrahedron: four ``triangles'' and three ``quadrilaterals''.
A general position surface $F$ properly embedded in $M$ is \emph{normal} if it intersects each tetrahedron of $\TTT$ in a union of normal discs; the typical intersection of a normal surface with a tetrahedron is shown in \zcref{fig:normal surfaces:intersection with tetrahedron}.

Denote by $t$ the size of the triangulation $\TTT$.
A normal surface $F$ can be encoded by a vector $\v{F}\in\ZZ_{\ge 0}^{7t}$ of $7t$ non\=/negative integers, recording for each tetrahedron $T$ of $\TTT$ and each type of normal disc in $T$ the number of normal discs in $F\cap T$ of that type.
The vector $\v{F}$ determines the normal surface $F$ uniquely up to normal isotopy in $M$.
There are two constraints a vector $\mathbf{w}\in\ZZ_{\ge 0}^{7t}$ needs to satisfy in order to represent a normal surface.
The first is given by the \emph{matching equations}, which we now describe.
An arc properly embedded in a triangle $R$ is called \emph{normal} if its endpoints lie in the interiors of two distinct edges of $R$; two normal arcs are of the same type if their endpoints lie on the same two edges of $R$.
Suppose now that $R$ is a triangle in the $2$\=/skeleton of $M$ that does not lie in $\boundary M$.
For each type $q$ of normal arc in $R$, the matching equation for $q$ reads
\[
\mathbf{w}_i+\mathbf{w}_j=\mathbf{w}_k+\mathbf{w}_l,
\]
where $i$, $j$, $k$, and $l$ are the four types of normal disc (two triangles and two quadrilaterals) that intersect $R$ in normal arcs of type $q$, with $i$ and $j$ lying in one of the two tetrahedra adjacent to $R$, and $k$ and $l$ lying in the other; a graphical representation of the matching equation for $q$ is shown in \zcref{fig:normal surfaces:matching equation}.

\begin{myfigure}
\centering
\begin{subfigure}{0.45\textwidth}
\centering
\myfiguresource[scale=0.3]{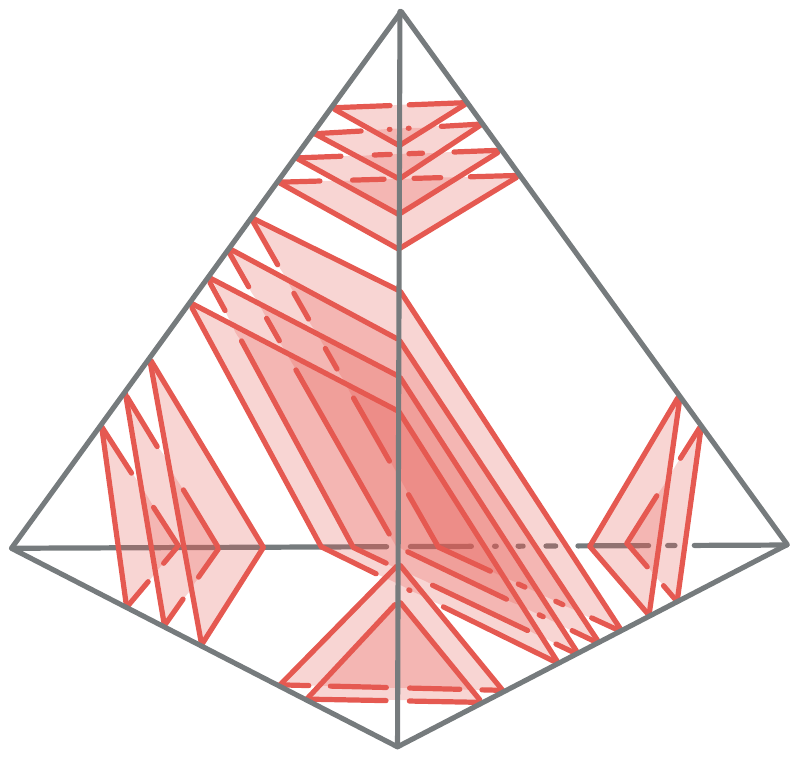}
\caption{}
\label{fig:normal surfaces:intersection with tetrahedron}
\end{subfigure}
\hfill
\begin{subfigure}{0.45\textwidth}
\centering
\myfiguresource[scale=0.3]{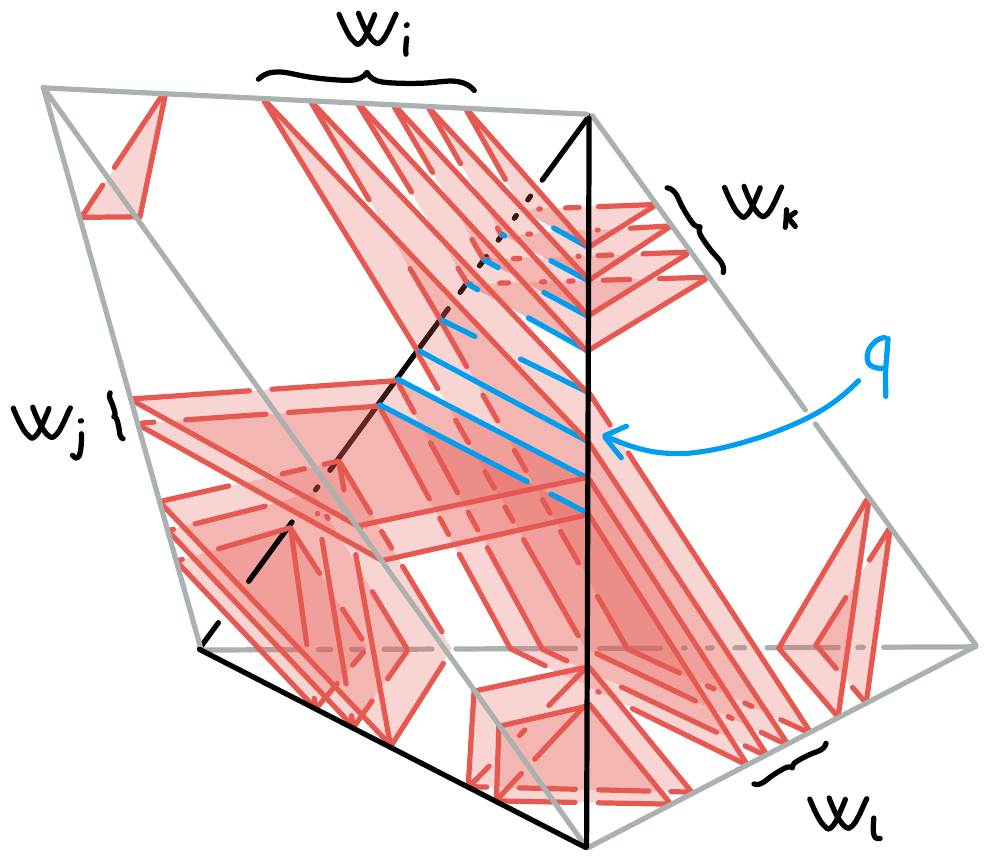}
\caption{}
\label{fig:normal surfaces:matching equation}
\end{subfigure}
\caption{\subref{fig:normal surfaces:intersection with tetrahedron} The intersection between a normal surface and a tetrahedron of the triangulation.
\subref{fig:normal surfaces:matching equation} The matching equation for the type $q$ of normal arc reads $\mathbf{w}_i+\mathbf{w}_j=\mathbf{w}_k+\mathbf{w}_l$.}
\end{myfigure}

The second constraint is given by the \emph{consistency equations}.
These equations enforce that each tetrahedron must contain at most one type of normal quadrilateral; in fact, it is easy to see that two normal quadrilateral of different types in the same tetrahedron must intersect.
More precisely, the consistency equation for a tetrahedron $T$ of $\TTT$ reads
\[
\text{at most one of $\mathbf{w}_i$, $\mathbf{w}_j$, and $\mathbf{w}_k$ is non\=/zero},
\]
where $i$, $j$, and $k$ are the three types of normal quadrilateral in $T$.
One can show that there is a one\=/to\=/one correspondence between normal isotopy classes of normal surfaces in $M$ and vectors in $\ZZ_{\ge 0}^{7t}$ satisfying the matching and consistency equations.
It will be implicitly understood that, whenever an algorithm manipulates a normal surface $F$ -- in particular, if it is given as input to the algorithm -- it is the associated vector that is being manipulated.
Note that the vector $\v{F}$ can be encoded with a number of binary digits that is linear in $t\cdot \log w(F)$.

Given two normal surfaces $G$ and $H$ in $M$, suppose that $\v{G}+\v{H}$ satisfies the consistency equations.
We can then define the \emph{normal sum} $G+H$ of $G$ and $H$ as the normal surface with normal vector $\v{G}+\v{H}$.
A normal surface is called \emph{fundamental} if it cannot be written as the normal sum of two non\=/empty normal surfaces.
A key property of fundamental surfaces is that their weights are uniformly bounded by a constant that depends only on the size of the triangulation $\TTT$.

\begin{proposition}[Weight of fundamental surfaces]
\label{thm:bound on weight of fundamental surfaces}
Let $\TTT$ be a triangulation of a compact $3$\=/manifold $M$ with $t$ tetrahedra, and let $F$ be a fundamental normal surface in $M$.
Then
\[
w(F)\le t^2\cdot2^{7t+7}.
\]
\end{proposition}
\begin{proof}
This result is an immediate consequence of \cite[Lemma~6.1]{hass-lagarias-pippenger:computational-complexity-knot}, which states that each entry of $\v{F}$ is bounded above by $t\cdot 2^{7t+2}$.
Since the vector $\v{F}$ has $7t$ entries, and each normal disc contributes to at most $4$ points in $F\cap\TTT^{(1)}$, the desired inequality follows.
\end{proof}

The following result shows how incompressible boundary\=/incompressible surfaces in irreducible boundary\=/irreducible $3$\=/manifolds can be isotoped to be normal.

\begin{proposition}[Normalising surfaces]
\label{thm:normalising surfaces}
Let $F$ be an incompressible boundary\=/incompressible general position surface properly embedded in a triangulated compact irreducible boundary\=/irreducible $3$\=/manifold $M$.
Suppose that no component of $F$ is a sphere or a disc.
Then $F$ is isotopic to a normal surface $F'$ with $w(F')\le w(F)$.
\end{proposition}
\begin{proof}
This is essentially conclusion 4 of \cite[Proposition~3.3.24]{matveev:algorithmic-topology-classification}, although that statement is slightly less general than what we are claiming.
One can, however, upgrade the proof therein, by noting that the eight normalisation moves described in \cite[\S3.3.3]{matveev:algorithmic-topology-classification}, when applicable, preserve the property of a surface to contain a subsurface that is isotopic to $F$; note that irreducibility and boundary\=/irreducibility of $M$, together with incompressibility and boundary\=/incompressibility of $F$, are crucial to this claim.
Moreover, these moves do not increase the weight of the surface, as noted in \cite[Remark~3.3.22]{matveev:algorithmic-topology-classification}.
\end{proof}

We say that a general position surface $F$ properly embedded in $M$ is \emph{least\=/weight} if it has minimum weight amongst all general position surfaces isotopic to it.
As an immediate consequence of \zcref{thm:normalising surfaces}, we see that every incompressible boundary\=/incompressible general position surface properly embedded in a triangulated compact irreducible boundary\=/irreducible $3$\=/manifold is isotopic to a least\=/weight normal surface, provided it does not have any sphere or disc components.

\subsection{Sub-complexes and triangulations of normal surfaces}
\label{sec:triangulations of normal surfaces}

Let $F$ be a triangulated compact surface.
In this article, the surface $F$ will always be a normal surface in some triangulated $3$\=/manifold, but we do not enforce this for now.
For the applications in the upcoming article \cite{baroni:certifying-hyperbolicity-fibred}, we will sometimes want to avoid situations in which a triangle has two or more edges lying on the boundary of the surface.
Hence, we give the following definition.

\begin{definition}[Flapless triangulation]
A triangulation of a compact surface $F$ is \emph{flapless} if each triangle has at most one edge lying on $\boundary F$.
\end{definition}

We will often need to deal with subsets of $F$ that are simplicial -- that is, unions of simplices -- but are not necessarily submanifolds.
A \emph{sub\=/$1$\=/complex} of $F$ is a union $X$ of edges of $F$; we denote by $\length{X}$ the number of edges in $X$.

Similarly, a \emph{sub\=/$2$\=/complex} of $F$ is a union $X$ of triangles of $F$; we denote by $\area{X}$ the number of triangles in $X$.
See \zcref{fig:triangulations of surfaces:sub-2-complex} for an example of a sub\=/$2$\=/complex of a triangulated surface.
Even though $X$ may have singular points that prevent it from being a subsurface -- namely, when it contains two triangles that intersect in a vertex and not in one of the adjacent edges -- there is an abstract triangulated surface $\abstr{X}$, together with an immersion $\map{\embedd[X]}{\abstr{X}}{F}$, such that $\embedd[X](\abstr{X})=X$ and $\embedd[X]$ is an embedding except at the vertices of $\abstr X$.
The surface $\abstr X$ is obtained by taking the triangles of $X$ and gluing them along shared edges, but not along isolated shared vertices.
We will often blur the distinction between $X$ and $\abstr{X}$; for instance, when we talk about the components of $X$, we will actually mean the images under $\embedd[X]$ of the components of $\abstr{X}$.
It will always be clear from the context whether we are talking about $X$ as a topological subspace of $F$ or as an abstract triangulated surface.

Sometimes, for convenience, we will want to transform a sub\=/$2$\=/complex $X$ of $F$ into an actual subsurface; this can be done in two ways.
The first way is to ``thicken'' $X$ by adding a small regular neighbourhood of $X$ in $F$; the result is a subsurface of $F$, that we denote by $\thick{X}$, as shown in \zcref{fig:triangulations of surfaces:thick thin}.
Note that $\thick{X}$ is defined up to isotopy and, as a topological space, it is homotopy equivalent to $X$.
We will always pick $\thick{X}$ in its isotopy class so that it does not intersect any other relevant objects in $F$ that are disjoint from $X$.
The second way to obtain a subsurface from $X$ is to ``carve out'' a regular neighbourhood of the boundary of $X$ from $X$.
More precisely, we let
\[
\thin{X}=\closure{X\setminus\thick{\closure{F\setminus X}}},
\]
as shown in \zcref{fig:triangulations of surfaces:thick thin}.
This is a subsurface of $X$, and it is homeomorphic to $\abstr{X}$.
Like $\thick{X}$, the subsurface $\thin{X}$ is defined up to isotopy; whenever possible, we will pick $\thin{X}$ so that it contains every relevant object in $F$ that is contained in $X$.

\begin{myfigure}
\centering
\begin{subfigure}{0.45\textwidth}
\centering
\myfiguresource[scale=0.3]{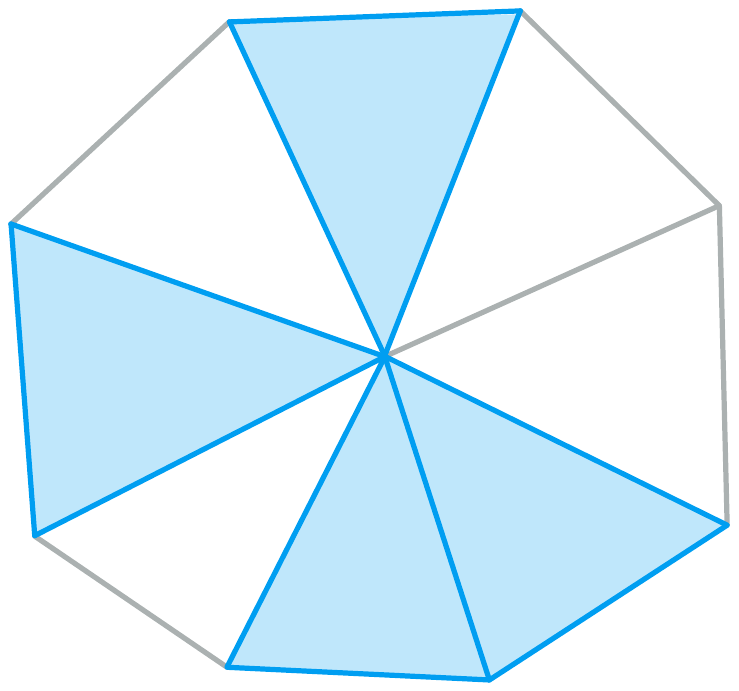}
\caption{}
\label{fig:triangulations of surfaces:sub-2-complex}
\end{subfigure}
\hfill
\begin{subfigure}{0.45\textwidth}
\centering
\myfiguresource[scale=0.3]{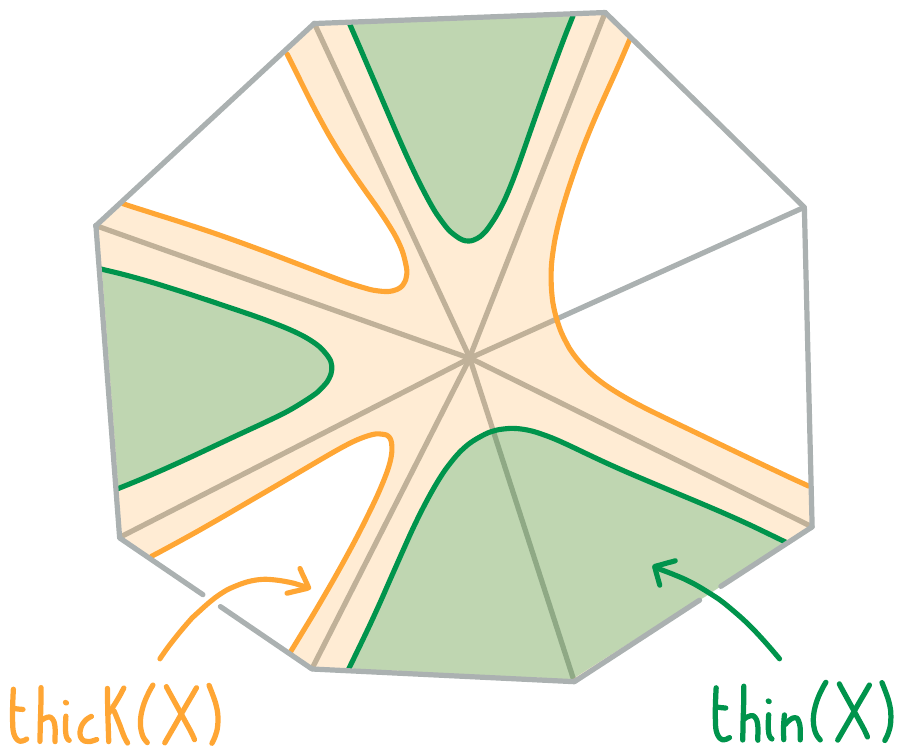}
\caption{}
\label{fig:triangulations of surfaces:thick thin}
\end{subfigure}
\caption{\subref{fig:triangulations of surfaces:sub-2-complex} A sub\=/$2$\=/complex $X$ of a triangulated surface.
\subref{fig:triangulations of surfaces:thick thin} The subsurfaces $\thick{X}$ and $\thin{X}$; note that $\thin{X}$ is a subset of $\thick{X}$.}
\end{myfigure}

Suppose now that $\TTT$ is a triangulation of a compact $3$\=/manifold $M$, and that $F$ is a normal surface in $M$.
The triangulation of $M$ induces a canonical triangulation of $F$, obtained as follows.
Normal triangles of $F$ are triangulated with $3$ triangles, and normal quadrilaterals are triangulated with $4$ triangles, as shown in \zcref{fig:triangulations of normal surfaces:triangulation}.
We say that two triangles in this triangulation of $F$ are of the same \emph{type} if they lie in the same tetrahedron $T$ of $M$, they lie on normal discs of the same type, and they intersect the same face of $\boundary T$ (here, we think of $T$ as an \emph{abstract tetrahedron}, that is, ignoring the face identifications).
Therefore, there are $3$ types of triangles for each type of normal triangle, and $4$ types of triangles for each type of normal quadrilateral.
From now on, we will always implicitly assume that normal surfaces are endowed with this triangulation.

\begin{remark}[The triangulation of a normal surface is flapless]
\label{rmk:the triangulation of a normal surface is flapless}
Even though the triangulation of a normal surface we have defined might seem ``suboptimal'', in the sense that it uses more triangles than necessary, it has the property (convenient for our results in \cite{baroni:certifying-hyperbolicity-fibred}) of being flapless.
\end{remark}

\begin{myfigure}
\centering
\myfiguresource[scale=0.3]{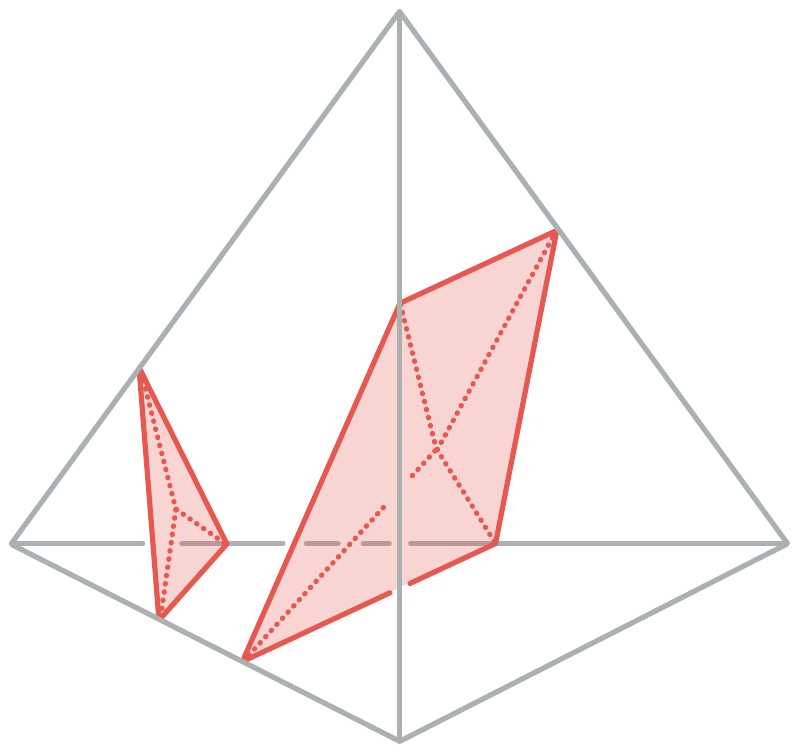}
\caption{Normal triangles of a normal surface are triangulated with $3$ triangles, and normal quadrilaterals with $4$ triangles.}
\label{fig:triangulations of normal surfaces:triangulation}
\end{myfigure}

We now describe how to compactly encode sub\=/$2$\=/complexes of $F$, with respect to the triangulation of $F$ inherited from $\TTT$.
Let us introduce the following notation: if $F$ is any compact surface, and $G$ is any subset of $F$, we define
\[
\boundary_F G=\closure{\boundary G\setminus\boundary F}.
\]
If $F$ is triangulated and $G$ is a sub\=/$2$\=/complex of $F$, then we can encode $G$ as follows: we list the edges in $\boundary_F G$, and we include an additional list of triangles of $F$, one for each component of $G$; note that this information is enough to fully reconstruct $G$.
This encoding uses a number of binary digits that is linear in
\[
\log \card{\TTT}\cdot\log w(F)\cdot\length{\boundary_F G}\cdot\card{G}.
\]
We will see in \zcref{sec:applications of agol-hass-thurston} how to extract topological information about $G$ from this compressed encoding.

\subsection{Agol-Hass-Thurston orbit-counting algorithm}

In \cite{agol-hass-thurston:computational-complexity-knot}, \citeauthor{agol-hass-thurston:computational-complexity-knot} introduce an algorithm to count the number of orbits of a given collection of pairings on an interval.
More precisely, the setting is as follows.
Let $N\ge 1$ be an integer.
A \emph{pairing} on $[1,N]\cap\ZZ$ is an order\=/preserving or order\=/reversing bijection
\[
\map{g}{[a,b]\cap\ZZ}{[c,d]\cap\ZZ},
\]
where $1\le a\le b\le N$ and $1\le c\le d\le N$ are integers.
Let
\[
\{\map{g_i}{[a_i,b_i]\cap\ZZ}{[c_i,d_i]\cap\ZZ}\}_{1\le i\le k}
\]
be a collection of pairings on $[1,N]\cap\ZZ$.
These pairings induce an equivalence relation on $[1,N]\cap\ZZ$, generated by
\[
\{x\sim g_i(x):1\le i\le k,a_i\le x\le b_i\}.
\]
We call the equivalence classes of this relation \emph{orbits}.
Note that a pairing $\umap{[a,b]\cap\ZZ}{[c,d]\cap\ZZ}$ can simply be described by the four integers $a$, $b$, $c$, and $d$, together with a sign indicating whether the pairing is order\=/preserving or order\=/reversing.
Therefore, the overall amount of information needed to describe the collection $g_1,\ldots, g_k$ is linear in $k\cdot\log N$.
The algorithm of \citeauthor{agol-hass-thurston:computational-complexity-knot}, in its most basic form, computes the number of orbits of a given collection of pairings.

\begin{theorem}[{\cite[Theorem~12]{agol-hass-thurston:computational-complexity-knot}}]
\label{thm:agol-hass-thurston orbit counting algorithm}
There is an algorithm that takes as input an integer $N\ge 1$ and a collection $g_1,\ldots, g_k$ of pairings on $[1,N]\cap\ZZ$, and outputs the number of orbits.
The running time of the algorithm is polynomial in $\log N$ and $k$.
\end{theorem}

Even though this algorithm works in the full generality described above, its primary application -- and the reason for its development -- is computing properties of normal surfaces in triangulated $3$\=/manifolds.
For instance, the basic version of the algorithm of \citeauthor{agol-hass-thurston:computational-complexity-knot} stated above can be used to compute the number of components of a normal surface $F$ in a triangulated $3$\=/manifold $M$.
We will now give an account of this classical result, already described in \cite[Corollary~14]{agol-hass-thurston:computational-complexity-knot}, in order to familiarise the reader with the algorithm, and set up the stage for more complex applications.

Let $N$ be the total number of normal discs of $F$; in other words, the integer $N$ is the sum of the coordinates of the vector $\v{F}$.
We think of integers in $[1,N]\cap\ZZ$ as the normal discs of $F$, in such a way that two consecutive normal discs of the same type in the same tetrahedron are represented by consecutive integers.
Therefore, the interval $[1,N]\cap\ZZ$ is split into $7t$ (possibly empty) subintervals representing the different types of normal disc, where $t$ is the number of tetrahedra of the triangulation of $M$.
Let $R$ be a triangle in the $2$\=/skeleton of $M$, and suppose that $R$ is adjacent to two (not necessarily distinct) tetrahedra $T_1$ and $T_2$.
Each of the three types of normal arc in $R$ determines up to three pairings, pairing normal discs of $T_1$ with normal discs of $T_2$.
These pairings are essentially the ones that determine the matching equations for normal surfaces; hence, we refer the reader to \zcref{fig:normal surfaces:matching equation} for a visual representation of the situation.
If we repeat this procedure for each triangle $R$ not contained in $\boundary M$, we obtain a collection of $\bigO(t)$ pairings on $[1,N]\cap\ZZ$.
By construction, the orbits of this collection are in natural bijection with the components of $F$; in fact, more precisely, an orbit consists exactly of the normal discs making up a component of $F$.
Since $N=\bigO(w(F))$, we can then apply the algorithm of \zcref{thm:agol-hass-thurston orbit counting algorithm} to compute the number of components of $F$; the running time of the algorithm is polynomial in $t$ and $\log w(F)$.
Therefore, we have just proved the following.

\begin{proposition}[Counting components of a normal surface]
\label{thm:counting components of a normal surface}
There is an algorithm that takes as input a triangulation of a compact $3$\=/manifold $M$ with $t$ tetrahedra and a normal surface $F$ in $M$, and outputs the number of components of $F$.
The running time of the algorithm is polynomial in $t$ and $\log w(F)$.
\end{proposition}

The most general version of the algorithm of \citeauthor{agol-hass-thurston:computational-complexity-knot} allows the interval $[1,N]\cap\ZZ$ to be equipped with a weight function with values in $\ZZ^d$ for some $d\ge 1$.
More precisely, to ensure efficiency, the weight function is given as a partition of $[1,N]\cap\ZZ$ into intervals $[p_j,q_j]\cap\ZZ$ for $1\le j\le m$, together with vectors $\mathbf{z}_j\in\ZZ^d$ for $1\le j\le m$.
This data defines a weight function $\map{z}{[1,N]\cap\ZZ}{\ZZ^d}$, by setting $z(x)=\mathbf{z}_j$ if $p_j\le x\le q_j$.
For a subset $X\subs[1,N]\cap\ZZ$, we define the weight of $X$ to be the sum of the weights of the elements of $X$.
The general version of the algorithm of \citeauthor{agol-hass-thurston:computational-complexity-knot} operates in this setting, and computes the list of orbits of a given collection of pairings, together with the weight of each orbit.
More precisely, the statement is as follows.

\begin{theorem}[{\cite[Theorem~16]{agol-hass-thurston:computational-complexity-knot}}]
\label{thm:agol-hass-thurston weighted orbit counting algorithm}
There is an algorithm that takes as input
\begin{itemize}
\item integers $N\ge1$ and $d\ge0$,
\item a collection $g_1,\ldots, g_k$ of pairings on $[1,N]\cap\ZZ$,
\item a partition of $[1,N]\cap\ZZ$ into intervals $[p_j,q_j]\cap\ZZ$ for $1\le j\le m$, and
\item vectors $\mathbf{z}_j\in\ZZ^d$ for $1\le j\le m$,
\end{itemize}
and outputs the list of orbits with their weights.
The running time of the algorithm is polynomial in $k$, $m$, $d$, $\log D$, and $\log N$, where $D$ is the maximum $\ell^1$-norm of the vectors $\mathbf{z}_1,\ldots,\mathbf{z}_m$.
\end{theorem}

To be even more precise, the algorithm as described in \cite{agol-hass-thurston:computational-complexity-knot} outputs a list of triples of the form $(r,s,\mathbf{w})$, signifying that each point in $[r,s]\cap\ZZ$ is a representative of a different orbit, all having the same weight $\mathbf{w}\in\ZZ^d$.
However, for our purposes, we can assume that the output of the algorithm is simply the list of weights attained by the orbits, together with their multiplicities.
We also remark that running the algorithm of \zcref{thm:agol-hass-thurston weighted orbit counting algorithm} with $d=0$ and the constant weight function $z$ is equivalent to running the basic orbit\=/counting algorithm of \zcref{thm:agol-hass-thurston orbit counting algorithm}.

\subsection{Applications of orbit-counting to normal surfaces}
\label{sec:applications of agol-hass-thurston}

Let $\TTT$ be a triangulation of a compact $3$\=/manifold $M$, and let $F$ be a normal surface in $M$.
A point $x$ of $F\cap\TTT^{(1)}$ can be fully determined by specifying an edge $e$ of $\TTT$, an orientation of $e$, and an integer between $1$ and $\card{e\cap F}$.
A transverse orientation of $F$ at $x$ is then simply determined by an orientation of $e$.
This is the input format we use in the proposition below.

\begin{proposition}[Computing orientability and transverse orientation of a normal surface]
\label{thm:transverse orientation of a normal surface}
There is an algorithm that takes as input
\begin{itemize}
\item a triangulation $\TTT$ of a compact orientable $3$\=/manifold $M$ with $t$ tetrahedra,
\item a connected normal surface $F$ in $M$,
\item two points $x,y\in F\cap\TTT^{(1)}$, and
\item a transverse orientation at $x$,
\end{itemize}
and outputs:
\begin{itemize}
\item whether $F$ is orientable or not;
\item assuming $F$ is orientable, the transverse orientation at $y$ that is compatible with the given one at $x$.
\end{itemize}
The running time of the algorithm is polynomial in $t$ and $\log w(F)$.
\end{proposition}
\begin{proof}
The key fact is that the vector $2\v{F}$ is the normal vector of a surface $G$ in $M$ which is the horizontal boundary of a regular neighbourhood of $F$ in $M$.
In particular, we see that $G$ is connected if and only if $F$ is not orientable.
Therefore, we can apply the algorithm of \zcref{thm:counting components of a normal surface} to the normal surface $G$ and deduce whether $F$ is orientable or not.
Assuming it is, we now need to compute the transverse orientation at $y$ that is compatible with the given one at $x$.
Thinking of $G$ as the unit normal bundle of $F$ in $M$, a transverse orientation at $x$ corresponds to a point $x'$ of $G\cap\TTT^{(1)}$ -- namely, one of the two points of $G$ that lie above $x$ -- which can be readily computed.
Finding the correct transverse orientation at $y$ essentially amounts to deciding which of the two points of $G$ that lie above $y$ belongs to the same connected component of $G$ as $x'$.
To this aim, we describe an ``orbit\=/tracking'' trick that will also be useful later.

Pick one of the two points of $G$ that lie above $y$ and call it $y'$.
In order to decide whether $x'$ and $y'$ lie in the same connected component of $G$, we run a slightly modified version of the algorithm of \zcref{thm:counting components of a normal surface}.
We equip the collection of pairings used therein with a weight function $\map{z}{[1,N]\cap\ZZ}{\ZZ^2}$; call the first coordinate of this $\ZZ^2$ the $x'$ coordinate, and the second one the $y'$ coordinate.
The $x'$ coordinate is set to $1$ on the normal discs of $G$ that are adjacent to $x'$, and $0$ elsewhere.
Similarly, the $y'$ coordinate is set to $1$ on the normal discs of $G$ that are adjacent to $y'$, and $0$ elsewhere.
We then run the algorithm of \zcref{thm:agol-hass-thurston weighted orbit counting algorithm} with this weight function.
If the algorithm outputs an orbit where both the $x'$ coordinate and the $y'$ coordinate are positive, then $x'$ and $y'$ lie in the same connected component of $G$, and hence $y'$ defines the correct transverse orientation at $y$.
Otherwise, the points $x'$ and $y'$ lie in different connected components of $G$, and hence the transverse orientation at $y$ is the opposite one.
\end{proof}

Recall that our combinatorial representation of a sub\=/$2$\=/complex $G$ of a normal surface $F$ in a triangulated $3$\=/manifold $M$ consists of two pieces of data: the edges of $\boundary_F G$ and a list containing one triangle of $F$ for each component of $G$.
If we want to make use of this representation, we are faced with two tasks.
Firstly, we need to be able to decide if a given representation is \emph{valid}, in the sense that it comes from an actual sub\=/$2$\=/complex of $F$.
Secondly, we need to be able to extract the relevant topological information about a sub\=/$2$\=/complex $G$ from its combinatorial representation.
This includes the Euler characteristic and number of boundary components of each component of $G$, but this is not enough: we also need to know how exactly the boundary components of $\abstr{G}$ are glued to the rest of $F$.
For this reason, we introduce the following definition.
If $F$ is oriented, then $\abstr{G}$ inherits an orientation, and hence each boundary component $b$ of $\abstr{G}$ is canonically oriented.
An \emph{$F$\=/boundary sequence} for $b$ is a sequence $e_1,\ldots,e_k$, where each $e_i$ is either an edge of $\boundary_FG$ or the symbol $\boundary F$, satisfying the following property: there exists a continuous orientation\=/preserving surjection $\map{f}{[0,k]}{b}$ that is an embedding on $(0,k)$ and such that, for each $1\le i\le k$, we have that
\begin{align*}
\embedd[G](f([i-1,i]))&=e_i&&\text{if $e_i$ is not $\boundary F$},\\
\embedd[G](f([i-1,i]))&\subs\boundary F&&\text{if $e_i$ is $\boundary F$}.
\end{align*}
Moreover, we ask that if $e_i=\boundary F$ for some $1\le i\le k$, then $e_{i+1}\neq\boundary F$ (where indices are taken modulo $k$).
Loosely speaking, one obtains an $F$\=/boundary sequence for $b$ by walking along $b$ (with the orientation induced by that of $\abstr{G}$) and recording the sequence of edges of $\boundary_FG$ that one encounters, interrupted by segments of $\boundary F$.
We remark that the $F$\=/boundary sequence for $b$ is unique up to cyclic permutations.

\begin{proposition}[Finding the components of a sub\=/$2$\=/complex of a normal surface]
\label{thm:finding the components of a sub-2-complex of a normal surface}
There is an algorithm that takes as input
\begin{itemize}
\item a triangulation of a compact oriented $3$\=/manifold $M$ with $t$ tetrahedra,
\item a connected transversely oriented normal surface $F$ in $M$,
\item a collection $\EEE$ of edges of $F$, and
\item a collection $\SSS$ of triangles of $F$,
\end{itemize}
and decides whether there exists a sub\=/$2$\=/complex $G$ of $F$ such that $\boundary_F G$ consists precisely of the edges of $\EEE$, and $\SSS$ contains at least one triangle of each component of $G$.
If this is the case, the algorithm also outputs, for each component $C$ of $G$:
\begin{itemize}
    \item the list of triangles in $\SSS$ that are contained in $C$;
    \item the Euler characteristic of $\abstr{C}$;
    \item the number of boundary components of $\abstr{C}$;
    \item for each component $b$ of $\boundary\abstr{C}$, an $F$\=/boundary sequence for $b$.
\end{itemize}
The running time of the algorithm is polynomial in $t$, $\log w(F)$, and the cardinalities of $\EEE$ and $\SSS$.
\end{proposition}
\begin{proof}
Denote by $\TTT$ the triangulation of $M$, and by $\RRR$ the induced triangulation of $F$.
Recall that the triangles of $\RRR$ are grouped into types, as detailed in \zcref{sec:triangulations of normal surfaces}.
In particular, each tetrahedron of $\TTT$ contains at most $16$ types of triangles of $\RRR$: at most $4$ for the unique type of normal quadrilateral, and at most $3$ per type of normal triangle.
We can then arrange the triangles of $\RRR$ in an interval $[1,N]\cap\ZZ$, where triangles of the same type form contiguous blocks; within each block, the triangles are ordered so that triangles lying on consecutive normal discs are adjacent.
It is easy to define a system of pairings on this interval, such that two triangles of $\RRR$ are paired if and only if they share an edge of $\RRR$ that is not in $\EEE$.
We also define a weight function
\[
\map{z}{[1,N]\cap\ZZ}{\ZZ^{d}},
\]
where the $d$ coordinates encode several pieces of information for each triangle $R$ of $\RRR$.
\begin{enumarabic}
\item The first block of $16t$ coordinates keeps track of the type of $R$ (here, we think of types as being indexed by the integers $1,\ldots,16t$); namely, the $i$\=/th coordinate is $1$ if $R$ is a triangle of type $i$, and $0$ otherwise.
\item The second block of $\card{\SSS}$ coordinates keeps track of which triangle in $\SSS$ -- if any -- is equal to $R$.
\item The third block of $2\card{\EEE}$ coordinates keeps track of which edges in $\EEE$ are adjacent to $R$, and from which side.
\item There is also a fourth block of coordinates that keeps track, for each triangle $R'$ of the same type as $R$ that is adjacent to some edge in $\EEE$, whether $R'$ appears before or after $R$, according to one of the two natural orders of the normal discs they lie on; we consider this order to be fixed once and for all.
These coordinates will be used later to compute information about the boundary components of $\abstr{G}$.
\end{enumarabic}

Run the algorithm of \zcref{thm:agol-hass-thurston weighted orbit counting algorithm} on the interval $[1,N]\cap\ZZ$, the collection of pairings, and the weight function $z$ we described above.
We can now assess whether a sub\=/$2$\=/complex $G$ of $F$ as required in the statement exists.
If any edge in $\EEE$ lies in $\boundary F$, then the answer is ``no''.
If the union of the edges in $\EEE$, seen as a graph, has a vertex of degree $1$ that does not lie on $\boundary F$, then the answer is ``no''.
Each orbit in the output of the algorithm of \zcref{thm:agol-hass-thurston weighted orbit counting algorithm} either contains a triangle in $\SSS$, or it does not; this can be assessed by looking at the second block of coordinates of the weight of the given orbit.
If some edge in $\EEE$ is contained in $0$ or $2$ orbits containing a triangle in $\SSS$ (which can be assessed by looking at the third block of coordinates), then the answer is ``no''.
Otherwise, the answer is ``yes'', and each orbit containing a triangle in $\SSS$ corresponds to a component of $\abstr{G}$; in other words, the sub\=/$2$\=/complex $G$ is the union of the orbits that contain a triangle in $\SSS$.

Fix now a component $C$ of $\abstr{G}$.
The weight $z(C)$ already contains, in the second block of coordinates, the list of triangles in $\SSS$ that lie inside $C$.
It is also not hard to see that the Euler characteristic of $C$ is an explicitly computable linear function of the weight $z(C)$ -- namely, of the first block of coordinates.
Computing the number of boundary components of $C$ and their $F$\=/boundary sequences is a bit more involved.
First, we can use the data contained in $z(C)$ to construct an interval $[1,N']\cap\ZZ$ that represents the edges of $\boundary C$, together with a collection of pairings encoding the adjacency of edges of $\boundary C$.
Note that we need to use the fourth block of coordinates here, to identify which edges of $\EEE$ lie on $\boundary C$.
Running the algorithm of \citeauthor{agol-hass-thurston:computational-complexity-knot} will then output the number of components of $\boundary C$.

In order to compute the $F$\=/boundary sequences, start by picking an edge $e_1\in\EEE$ that lies in $\boundary C$.
Note that we can recover which side of $e_1$ the component $C$ lies on from $z(C)$.
An application of \zcref{thm:transverse orientation of a normal surface} will allow us to deduce the correct orientation of the component $b$ of $\boundary C$ containing $e_1$.
We can then start walking along $b$ according to this orientation, explicitly recording the sequence $e_1,\ldots,e_i$ of edges in $\EEE$ that we encounter.
This process will terminate either when we get back to $e_1$ -- in which case we have successfully computed the $F$\=/boundary sequence for $b$ -- or when we hit $\boundary F$.
If the latter happens, we set $e_{i+1}=\boundary F$.
In order to understand what the next edge $e_{i+2}\in\EEE$ is, we construct an interval $[1,N'']\cap\ZZ$ representing the edges of $\boundary C\cap\boundary F$, together with a collection of pairings encoding the adjacency of said edges.
Using the ``orbit-tracking'' trick described in the proof of \zcref{thm:transverse orientation of a normal surface}, we can find the component of $\boundary C\cap\boundary F$ that comes after $e_i$ in the $F$\=/boundary sequence of $b$, together with the following edge $e_{i+2}\in\EEE$.
We can repeat this process until we get back to $e_1$, thus completing the $F$\=/boundary sequence for $b$.
Using the same technique, we can compute the $F$\=/boundary sequences for all the components of $\boundary C$ that are not contained in $\boundary F$.
Once we have done this, since we know how many boundary components $C$ has, we can deduce the number of components of $\boundary C$ that are contained in $\boundary F$; the $F$\=/boundary sequence for these components is simply $\boundary F$.
\end{proof}

\begin{proposition}[Deciding containments of sub\=/$2$\=/complexes of normal surfaces]
\label{thm:deciding containment of sub-2-complexes of normal surfaces}
For a compact oriented $3$\=/manifold $M$ triangulated with $t$ tetrahedra, a connected transversely oriented normal surface $F$ in $M$, and a sub\=/$2$\=/complex $G$ of $F$, the following questions can be answered by an algorithm with the specified running time.
\begin{enumarabic}
\item Given a triangle $T$ of $F$, decide whether $T$ is contained in $G$, in polynomial time in $t$, $\log w(F)$, $\card{G}$, and $\length{\boundary_F G}$.
\item Given two triangles $T_1$ and $T_2$ in $G$, decide whether they belong to the same component of $G$, in polynomial time in $t$, $\log w(F)$, $\card{G}$, and $\length{\boundary_F G}$.
\item Given a sub\=/$2$\=/complex $G'$ of $F$, decide whether $G'$ is contained in $G$, in polynomial time in $t$, $\log w(F)$, $\card{G}$, $\length{\boundary_F G}$, $\card{G'}$, and $\length{\boundary_F G'}$.
\item Given a sub\=/$2$\=/complex $G'$ of $F$, decide whether $G'$ intersects $G$, and whether $\interior{G'}$ intersects $\interior{G}$, in polynomial time in $t$, $\log w(F)$, $\card{G}$, $\length{\boundary_F G}$, $\card{G'}$, and $\length{\boundary_F G'}$.
\end{enumarabic}
\end{proposition}
\begin{proof}
The first two questions can be easily answered with the ``orbit\=/tracking'' trick described in the proof of \zcref{thm:transverse orientation of a normal surface}.
The third question reduces to the fourth one, since $G'$ is contained in $G$ if and only if $\interior{G'}$ is disjoint from $F\setminus G$; note that the sub\=/2\=/complex $\closure{F\setminus G}$ can be computed in polynomial time, since $\boundary_F\closure{F\setminus G}=\boundary_F G$, and for each edge $e$ of $\boundary_F G$, we can decide which of its two adjacent triangles belongs to $\closure{F\setminus G}$.
Therefore, we only need to show how to answer the fourth question.

We have seen in the proof of \zcref{thm:finding the components of a sub-2-complex of a normal surface} how to compute the list of components of $\abstr{G}$ and $\abstr{G'}$; by checking all possible pairs of components, we can suppose, without loss of generality, that $\abstr{G}$ and $\abstr{G'}$ are connected.
If $\boundary_F G$ and $\boundary_F G'$ are both empty, then $G$ and $G'$ intersect if and only if they are contained in the same component of $F$, which can be decided with an ``orbit\=/tracking'' trick; the same holds for $\interior{G}$ and $\interior{G'}$.
Otherwise, it is not hard to see that $\interior{G}$ and $\interior{G'}$ intersect if and only if there is a triangle of $F$ that is contained in $G\cap G'$ and intersects $\boundary_F G$ or $\boundary_F G'$.
Since the number of such triangles is bounded above by a polynomial in $t$, $\length{\boundary_F G}$, and $\length{\boundary_F G'}$, we can check all possible triangles in polynomial time.
Finally, if $\interior{G}$ and $\interior{G'}$ are disjoint, then $G$ and $G'$ intersect if and only if $\boundary_F G$ and $\boundary_F G'$ intersect.
This can be checked by direct inspection, thus providing an answer to the fourth question in the statement.
\end{proof}

\section{Guts and parallelity bundles}

\subsection{Pre-sutured manifolds and sub-3-complexes}

Sutured manifolds are a well\=/established tool in the study of $3$\=/manifolds, introduced by \citeauthor{gabai:foliations-topology-3} in \cite{gabai:foliations-topology-3}.
In this article, we will need sutured manifolds to describe candidate interval bundle structures on $3$\=/manifolds.
In particular, some of the items in our certificate will be ``$3$\=/manifolds with sutures'' that, in a valid certificate, will always be suture\=/preservingly homeomorphic to product interval bundles; however, if the certificate is invalid, this may not be the case, and in fact the sutures may not even be annuli or tori.
For this reason, we introduce a very lax variant of sutured manifolds, admitting arbitrary sutures, to fit our needs.

\begin{definition}[Pre\=/sutured manifold]
\label{def:pre-sutured manifold}
A \emph{pre\=/sutured manifold} is a compact oriented $3$\=/manifold $M$ together with a decomposition of $\boundary M$ into three compact surfaces $\boundary_0 M$, $\boundary_1 M$, and $\boundary_v M$; we require that $\boundary_0 M$ and $\boundary_1 M$ are disjoint, and that $\boundary_v M$ and $\boundary_0 M\cup\boundary_1 M$ intersect precisely along their boundaries.
The subsurface $\vboundary M\subs \boundary M$ is called the \emph{vertical boundary} of $M$.
We also write $\hboundary M$ for $\boundary_0 M\cup\boundary_1 M$, and we call it the \emph{horizontal boundary} of $M$.
\end{definition}

The terminology we use is borrowed from interval bundles, rather than from standard sutured manifold theory.
For a compact connected surface $F$, there is a unique orientable interval bundle $M$ over $F$.
Its horizontal boundary $\hboundary M$ is a two\=/sheeted cover of $F$, while its vertical boundary $\vboundary M$ is a disjoint union of annuli.
Whenever we say that a $3$\=/manifold is an interval bundle over a surface, we will always implicitly endow it with this natural pre\=/sutured manifold structure.
We will mostly be dealing with \emph{product interval bundles}, that is, $3$\=/manifolds of the form $M=F\times[0,1]$, where $F$ is a compact orientable surface.
In this case, note that $\vboundary M=\boundary F\times[0,1]$ and $\boundary_i M=F\times\{i\}$ for $i\in\{0,1\}$.
When $M$ is a \emph{twisted interval bundle} over a connected non\=/orientable surface, there is no natural way to split its horizontal boundary into $\boundary_0 M$ and $\boundary_1 M$; in this case we will allow both possible assignments -- that is, $\boundary_0 M=\hboundary M$ and $\boundary_1 M=\emptyset$ or vice versa.

An arc properly embedded in the vertical boundary of a pre\=/sutured manifold $M$ is \emph{vertical} if it intersects $\boundary_0M$ and $\boundary_1M$ each in a single point.
A subset of the vertical boundary of $M$ is \emph{vertical} if it is a disjoint union of (possibly infinitely many) vertical arcs; in particular, a compact connected vertical subset of $\vboundary M$ is necessarily a disc or an annulus.

Two pre\=/sutured manifolds $M$ and $N$ are \emph{homeomorphic as pre\=/sutured manifolds} if there is an orientation\=/preserving homeomorphism $\umap{M}{N}$ sending $\boundary_i M$ to $\boundary_i N$ for $i\in\{0,1\}$ -- and, consequently, $\vboundary M$ to $\vboundary N$.

A \emph{pre\=/sutured triangulation} is a triangulation of a pre\=/sutured manifold $M$, such that $\hboundary M$ and $\vboundary M$ are simplicial subsurfaces of $M$.
Note that the additional data -- that is, the surfaces $\boundary_0 M$, $\boundary_1 M$, and $\boundary_v M$ -- can be encoded with a number of binary digits that is linear in the size of the triangulation.
We say that a pre\=/sutured triangulation of a compact $3$\=/manifold $M$ is \emph{suitable} if $M$ is homeomorphic as a pre\=/sutured manifold to $\boundary_0 M\times[0,1]$.

Similarly to the $2$\=/dimensional setting, we will sometimes need to deal with subsets of a $3$\=/manifold that are submanifolds away from a lower\=/dimensional singular locus.
Let $M$ be a compact orientable $3$\=/manifold with a polyhedral cell structure -- meaning that $M$ is obtained by gluing a collection of polyhedral pieces along their faces.
This polyhedral cell structure might be a triangulation, or it might arise from cutting a triangulated $3$\=/manifold along a normal surface, as we will see in \zcref{sec:cutting along a normal surface}.
A \emph{sub\=/$3$\=/complex} of $M$ is a union $X$ of cells of $M$ such that the \emph{singular locus} of $X$ -- that is, the set of points of $x\in X$ such that no open neighbourhood of $x$ in $M$ intersects $X$ in an open $3$\=/ball or half\=/ball -- is a compact $1$\=/manifold properly embedded in $M$.
Like in the $2$\=/dimensional case, the sub\=/$3$\=/complex $X$ can be thought of as an abstract compact $3$\=/manifold $\abstr{X}$, together with an immersion $\map{\embedd[X]}{\abstr{X}}{M}$ such that $\embedd[X](\abstr{X})=X$.
The $3$\=/manifold $\abstr{X}$ can be obtained by taking the cells of $M$ contained in $X$ and gluing them along their common faces (but not along common isolated edges).
The immersion $\embedd[X]$ is an embedding, except at the points of $\abstr{X}$ that get mapped to the singular locus of $X$.
Like in the $2$\=/dimensional case, we will often blur the distinction between $X$ and $\abstr{X}$; in particular, by ``components of $X$'', we will mean the components of $\abstr{X}$, or possibly their images under $\embedd[X]$.

Similarly to what we did for sub\=/$2$\=/complexes of surfaces, we can define $\thick{X}$ and $\thin{X}$ for a sub\=/$3$\=/complex $X$ of $M$.
The definition is the same: we let $\thick{X}$ be the union of $X$ with a small regular neighbourhood of $\boundary X$ in $M$, and we let
\[
\thin{X}=\closure{M\setminus(\thick{\closure{M\setminus X}})}
\]
(note that $\closure{M\setminus X}$ is also a sub\=/$3$\=/complex of $M$).

A \emph{pre\=/sutured subcomplex} structure on $X$ is defined analogously to \zcref{def:pre-sutured manifold}, with the additional requirements that the singular locus of $X$ should be properly embedded in $\vboundary X$.
Note that a pre\=/sutured subcomplex structure on $X$ induces a pre\=/sutured manifold structure on $\abstr{X}$, by pull\=/back via $\embedd[X]$.
Conversely, a pre\=/sutured manifold structure on $\abstr{X}$ induces a pre\=/sutured subcomplex structure on $X$, provided that the preimage of the singular locus of $X$ under $\embedd[X]$ is properly embedded in $\vboundary\abstr{X}$.
If $X$ is a pre\=/sutured subcomplex of $M$, by \emph{suitable pre\=/sutured triangulation} of $X$ we mean the image of a suitable pre\=/sutured triangulation $\TTT$ of $\abstr{X}$ under $\embedd[X]$, provided that the preimage of the singular locus of $X$ under $\embedd[X]$ is a union of edges of $\TTT$.

\subsection{Cutting along a normal surface}
\label{sec:cutting along a normal surface}

If $F$ is a compact surface properly embedded in a compact $3$\=/manifold $M$, we can \emph{cut $M$ along $F$} to obtain a new compact $3$\=/manifold.
More precisely, we define $M\cut F$ to be the closure in $M$ of $M\setminus\NNN(F)$, where here and in the following $\NNN$ denotes a closed regular neighbourhood.
If, additionally, the $3$\=/manifold $M$ is oriented and the surface $F$ is transversely oriented, then $M'=M\cut F$ has a natural pre\=/sutured manifold structure.
In particular, we set $\vboundary M'=M'\cap\boundary M$, so that $\hboundary M'$ is the union of two parallel copies of $F$ in $M$; by convention, we set $\boundary_0 M'$ to be one of these copies, and $\boundary_1 M'$ to be the other, such that the transverse orientation of $F$ points from $\boundary_1 M'$ to $\boundary_0 M'$ at each point of $F$.
When $M$ is a compact oriented $3$\=/manifold and $F$ is a transversely oriented surface properly embedded in $M$ -- which we will assume for the rest of this section -- we will always implicitly endow $M\cut F$ with this pre\=/sutured manifold structure, as depicted in \zcref{fig:cutting along a normal surface:pre-sutured}.
Moreover, in this setting, the surface $F$ inherits a natural orientation; we will not emphasise this, but we will always implicitly assume that $F$ is oriented in this way.

The fact that $\NNN(F)$ is an interval bundle over $F$ provides two homeomorphisms
\[
\map{\pospushoff{(-)}}{F}{\boundary_0 M'}\quad\text{and}\quad\map{\negpushoff{(-)}}{F}{\boundary_1 M'},
\]
that we call the \emph{positive push\=/off} and the \emph{negative push\=/off} respectively; they are only defined up to isotopy, and they are depicted in \zcref{fig:cutting along a normal surface:push-offs}.

\begin{myfigure}
\centering
\begin{subfigure}{0.45\textwidth}
\centering
\myfiguresource[scale=0.3]{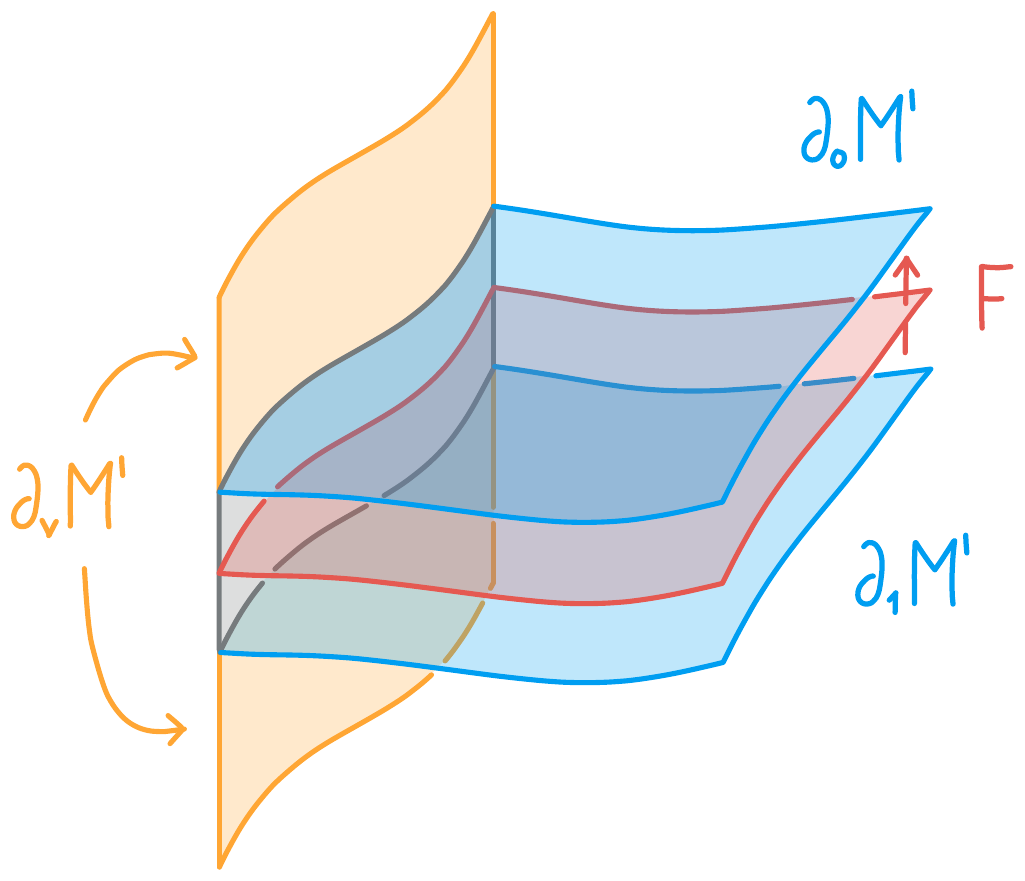}
\caption{}
\label{fig:cutting along a normal surface:pre-sutured}
\end{subfigure}
\hfill
\begin{subfigure}{0.45\textwidth}
\centering
\myfiguresource[scale=0.3]{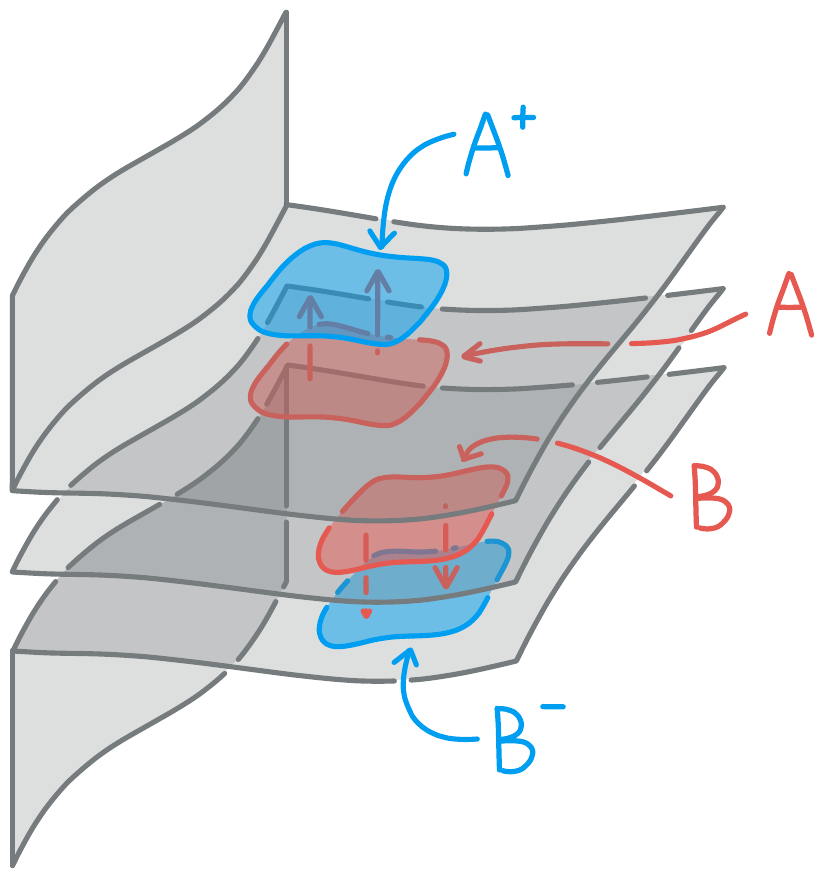}
\caption{}
\label{fig:cutting along a normal surface:push-offs}
\end{subfigure}
\caption{\subref{fig:cutting along a normal surface:pre-sutured} The subsurfaces $\boundary_0 M'$, $\boundary_1 M'$, and $\vboundary M'$ of $\boundary M'$ inducing the pre\=/sutured manifold structure on $M\cut F$; the transverse orientation of $F$ is represented by an arrow.
\subref{fig:cutting along a normal surface:push-offs} The positive and negative push\=/offs of the normal surface $F$; two arbitrary subsets $A$ and $B$ of $F$ are depicted, together with their push\=/offs $\pospushoff{A}\subs\boundary_0 M'$ and $\negpushoff{B}\subs\boundary_1 M'$.}
\end{myfigure}

When $M$ is described by a triangulation $\TTT$ and $F$ is normal, we can put more structure on $M\cut F$.
Let $T$ be a tetrahedron of $\TTT$.
A \emph{piece} of $M'$ is a component of $T\cap M'$ (here, we think of $T$ as an abstract tetrahedron).
A \emph{parallelity piece} of $M'$ is a piece of $M'$ that lies between two normal discs of the same type.
If a piece is not a parallelity piece, then we call it a \emph{gut piece}.
See \zcref{fig:cutting along a normal surface:parallelity and gut pieces} for an example of parallelity and gut pieces in a tetrahedron.
The union of all parallelity pieces in $M'$ is called the \emph{parallelity bundle}, and the union of all gut pieces is called the \emph{guts}.

\begin{myfigure}
\centering
\begin{subfigure}{0.45\textwidth}
\centering
\myfiguresource[scale=0.3]{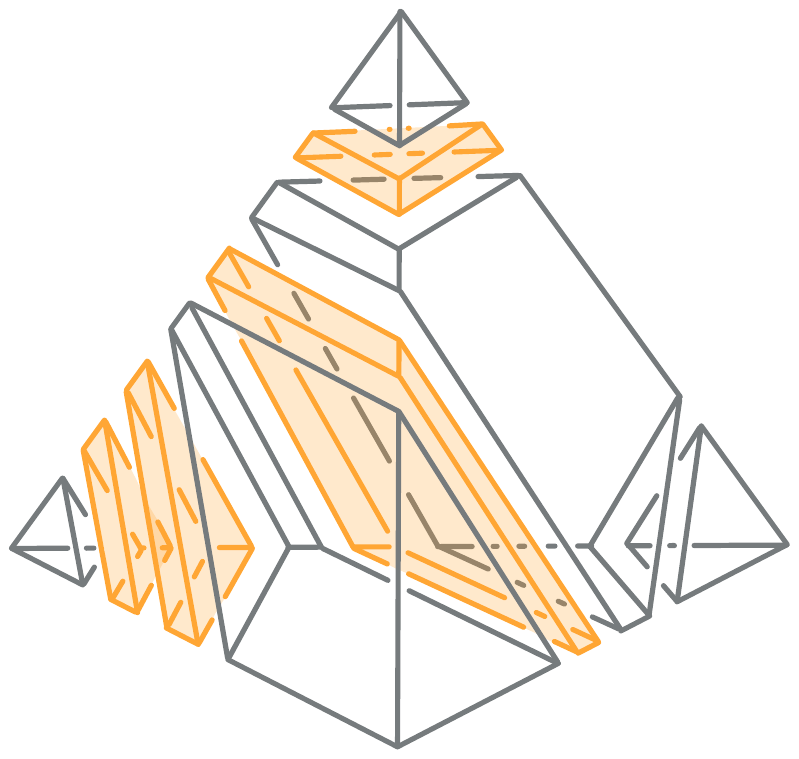}
\caption{}
\label{fig:cutting along a normal surface:parallelity pieces}
\end{subfigure}
\hfill
\begin{subfigure}{0.45\textwidth}
\centering
\myfiguresource[scale=0.3]{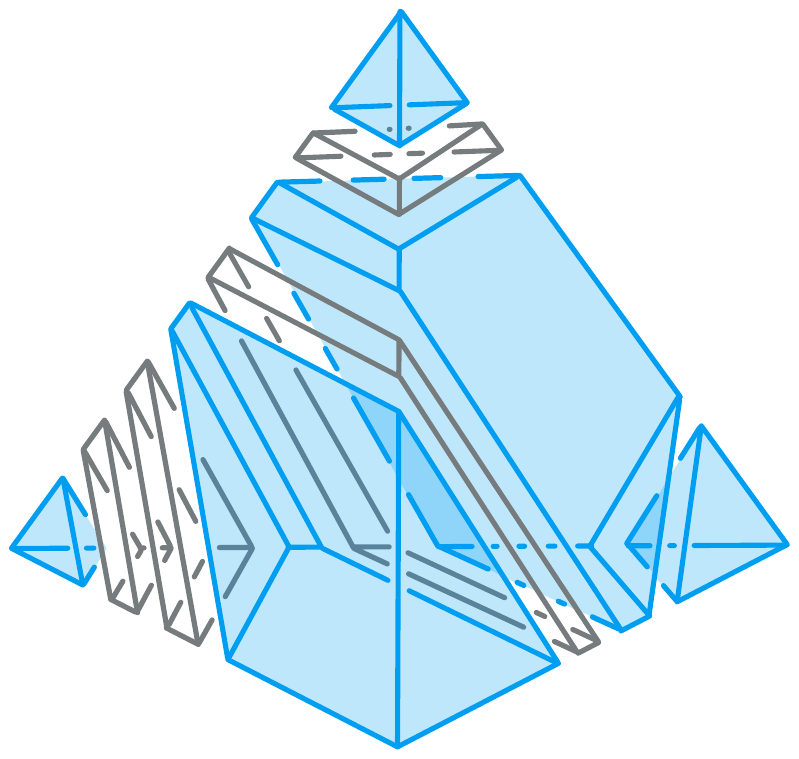}
\caption{}
\label{fig:cutting along a normal surface:gut pieces}
\end{subfigure}
\caption{\subref{fig:cutting along a normal surface:parallelity pieces} The parallelity pieces of $M'$ in a tetrahedron $T$ of $\TTT$ are the components of $T\cap M'$ that lie between two normal discs of the same type.
\subref{fig:cutting along a normal surface:gut pieces} The gut pieces of $M'$ in $T$ are the components of $T\cap M'$ that are not parallelity pieces.}
\label{fig:cutting along a normal surface:parallelity and gut pieces}
\end{myfigure}

Denote by $X$ the guts of $M'$, and by $Y$ the parallelity bundle.
Note that $X$ and $Y$ are sub\=/$3$\=/complexes of $M'$ (with the polyhedral cell structure whose cells are the pieces of $M'$), and the same holds for every component of $X$ and $Y$; we call a component of $X$ a \emph{gut component}, and a component of $Y$ a \emph{parallelity component}.
We remark that $Y$ is an interval bundle over some compact $2$\=/complex, since it is obtained by gluing together parallelity pieces, each of which is a product interval bundle over a normal disc.
Therefore, the submanifold $\thick{Y}$ of $M'$ is an interval bundle over some compact surface; we call this interval bundle structure the \emph{intrinsic interval bundle structure} of $\thick{Y}$.
Each component of $\thick{Y}$, endowed with its intrinsic interval bundle structure, is either a product interval bundle over an orientable surface, or a twisted interval bundle over a non\=/orientable surface; we will mostly be dealing with the former case, but we cannot a priori exclude the latter.
Note that, importantly, the intrinsic interval bundle structure of $\thick{Y}$ need not agree with the pre\=/sutured manifold structure of $M'$, in the sense that $\boundary_0\thick{Y}$ might intersect $\boundary_1 M'$ (or vice versa).
It is, however, true that
\[
\hboundary\thick{Y}=\thick{Y}\cap\hboundary M.
\]
There is some ambiguity in the choice of intrinsic interval bundle structure for $\thick{Y}$ -- namely, for each component $C$ of $\thick{Y}$, the surfaces $\boundary_0 C$ and $\boundary_1 C$ can be interchanged.
When $C$ is a product interval bundle intersecting both $\boundary_0 M$ and $\boundary_1 M$, we require that $\boundary_i C=\boundary i_M \cap C$ for $i\in\{0,1\}$.
Otherwise, we allow both choices.

Suppose -- as we will always implicitly assume -- that $\thin{X}$ is chosen so that
\[
\thin{X}\cap\thin{Y}=\boundary\thin{X}\cap\boundary\thin{Y}.
\]
Even though $\thin{X}$ is not necessarily an interval bundle, we can still endow it with a pre\=/sutured manifold structure, by setting
\begin{align*}
\boundary_0\thin{X}&=\thin{X}\cap\boundary_0 M',\\
\boundary_1\thin{X}&=\thin{X}\cap\boundary_1 M',\text{ and}\\
\vboundary\thin{X}&=\thin{X}\cap(\vboundary M'\cup \thick{Y}).
\end{align*}

The guts $X$ and parallelity bundle $Y$ of $M'$ admit pre\=/sutured subcomplex structures compatible with the pre\=/sutured manifold structures of $\thin{X}$ and $\thick{Y}$, defined by
\begin{align*}
\boundary_i X&=X\cap\boundary_i M'&\text{for $i\in\{0,1\}$,}\\
\boundary_i Y&=Y\cap\boundary_i\thick{Y}&\text{for $i\in\{0,1\}$.}
\end{align*}
We will always endow $X$ and $Y$ with these pre\=/sutured subcomplex structures, and $\abstr{X}$ and $\abstr{Y}$ with the induced pre\=/sutured manifold structures.
We call the pre\=/sutured structures on $Y$ and on $\abstr{Y}$ \emph{intrinsic}, again to emphasise that they need not be compatible with the pre\=/sutured manifold structure of $M'$.
We will also say that a subset $\Omega\subs\vboundary X$ is a vertical subset of $\vboundary X$ to mean that $\embedd[X]^{-1}(\Omega)$ is a vertical subset of $\vboundary\abstr{X}$, and similarly for $Y$.

\subsection{Triangulating the guts}
\label{sec:triangulating the guts}

Let us keep the notation of the previous paragraphs.
Recall that $F$ inherits a natural triangulation from $\TTT$.
We now describe how to triangulate the horizontal boundary and the guts of $M'$.
Firstly, note that $\hboundary M'$ is itself a normal surface, since it intersects each tetrahedron of $\TTT$ in a collection of normal discs; more precisely, there is a natural one\=/to\=/two correspondence between the normal discs of $F$ and the normal discs of $\hboundary M'$.
Therefore, we can triangulate $\hboundary M'$ as a normal surface, using $3$ triangles for each normal triangle and $4$ triangles for each normal quadrilateral.
With these canonical triangulations of $F$ and $\hboundary M'$, the positive and negative push\=/offs of $F$ can be naturally chosen to be simplicial isomorphisms.
More precisely, for each tetrahedron $T$ of $\TTT$, the positive push\=/off can be taken to map each normal disc $D\subs F\cap T$ to the normal disc $\boundary_0 M'\cap\NNN(D)$ simplicially; this is shown in \zcref{fig:triangulating the guts:simplicial push-off}.
A similar statement holds for the negative push\=/off.

In this setting, we can define the \emph{transfer map}.
This is a simplicial map, defined on the largest sub\=/$2$\=/complex of $F$ whose positive push\=/off is contained in $Y$.
To define it, let $D$ be a normal disc of $F$ contained in a tetrahedron $T$ of $\TTT$; suppose that $\pospushoff{D}$ is contained in a parallelity piece of $T$.
Then there is a unique normal disc $E$ in $T$, of the same type as $D$, such that either $\pospushoff{E}$ or $\negpushoff{E}$ cobounds a parallelity piece of $T$ with $\pospushoff{D}$.
Suppose, for concreteness, that $\pospushoff{E}$ cobounds a parallelity piece $P$ with $\pospushoff{D}$ (the case where $\pospushoff{E}$ is replaced by $\negpushoff{E}$ is analogous).
Let $\map{f}{\pospushoff{D}}{\pospushoff{E}}$ be the simplicial isomorphism induced by the interval bundle structure of $P$.
Then the transfer map $\transfermap$ sends a point $x\in D$ to the unique point $\transfermap(x)\in E$ such that
\[
\pospushoff{\transfermap(x)}=f(\pospushoff{x}).
\]

We now construct a triangulation of the guts of $M'$, building it piece by piece.
Let $P$ be a gut piece of $M'$; in other words, it is a component of $X\cap T$ for some (abstract) tetrahedron $T$ of $\TTT$.
We can think of $P$ as a polyhedron, with some faces coming from $P\cap\hboundary X$, and the other faces coming from $P\cap \boundary T$.
We triangulate the faces of $P\cap\hboundary X$ according to the triangulation of $\hboundary M'$ we have already defined, as shown in \zcref{fig:triangulating the guts:gut piece:1}.
Then, we note that $P\cap Y$ is a union of four\=/sided faces contained in $\boundary T$.
Some of these -- namely, those that intersect both $\boundary_0 M'$ and $\boundary_1 M'$ -- are vertical subsets of $\vboundary X$; denote the union of these faces by $V_P$.
We triangulate each face $R$ of $V_P$ by adding a diagonal edge.
In order to remove ambiguity, we describe how to pick one of the two possible diagonal edges consistently.
Let $e$ be the edge of $R$ lying on $\boundary_0X$.
Of the two edges of $R$ that are vertical in $\vboundary X$, let $e'$ be the one such that $e$ can be rotated counter\=/clockwise inside $R$ to reach $e'$; the counter\=/clockwise direction is taken by looking at $T$ from the outside (this is well\=/defined since $M$ is oriented).
The diagonal edge is then chosen so that one of its endpoints is $e\cap e'$, as depicted in \zcref{fig:triangulating the guts:gut piece:2}.

We then triangulate the remaining faces of $P$ -- that is, those in that are contained in $\boundary T$ but not in $V_P$ -- by coning over a vertex.
More precisely, let $t$ be the number of tetrahedra of $\TTT$.
We fix a total order on the vertices of the gut pieces of $M'$; here, we mean that if two vertices come from possibly different gut pieces but are identified in $M$, then they should only be counted once.
Since there are at most $6t$ gut pieces, independently of $F$, we can assume that this order is part of the data describing the triangulation $\TTT$.
Then, each relevant face of $P$ can be triangulated by coning over its smallest vertex, as shown in \zcref{fig:triangulating the guts:gut piece:3}.
Finally, we triangulate the interior of $P$ by coning the triangulation of $\boundary P$ over the smallest vertex of $P$.

\begin{myfigurepage}
\centering
\begin{subfigure}{0.45\textwidth}
\centering
\myfiguresource[scale=0.3]{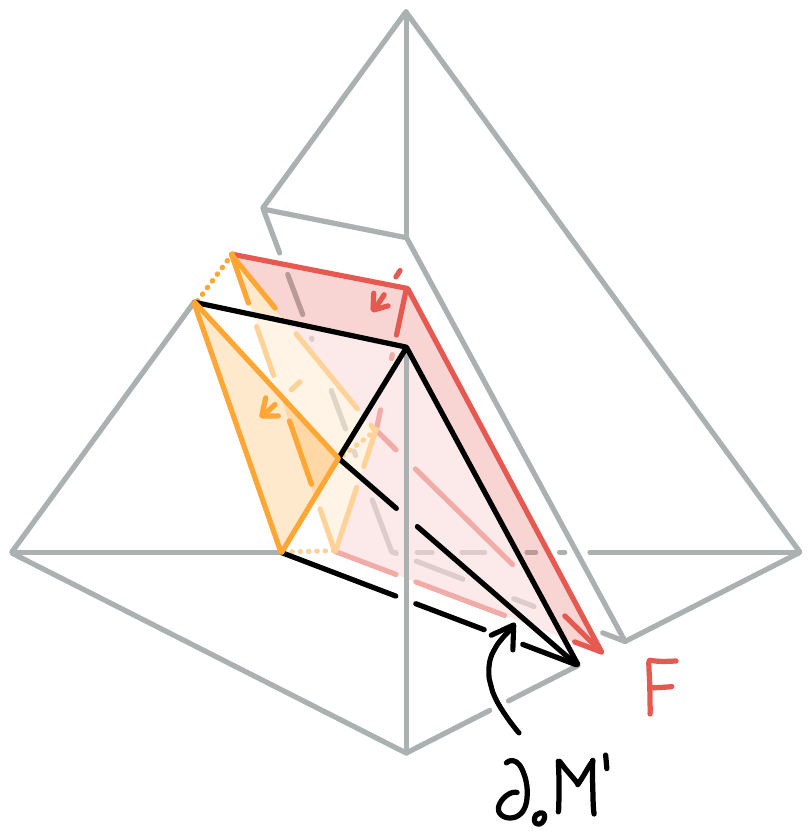}
\caption{}
\label{fig:triangulating the guts:simplicial push-off}
\end{subfigure}
\hfill
\begin{subfigure}{0.45\textwidth}
\centering
\myfiguresource[scale=0.3]{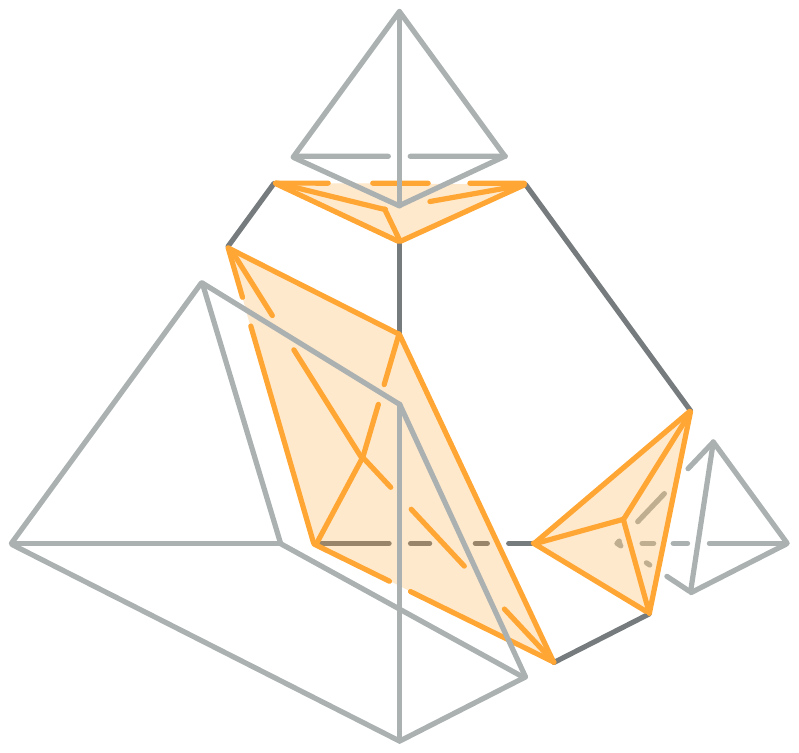}
\caption{}
\label{fig:triangulating the guts:gut piece:1}
\end{subfigure}\\[2em]
\begin{subfigure}{0.45\textwidth}
\centering
\myfiguresource[scale=0.3]{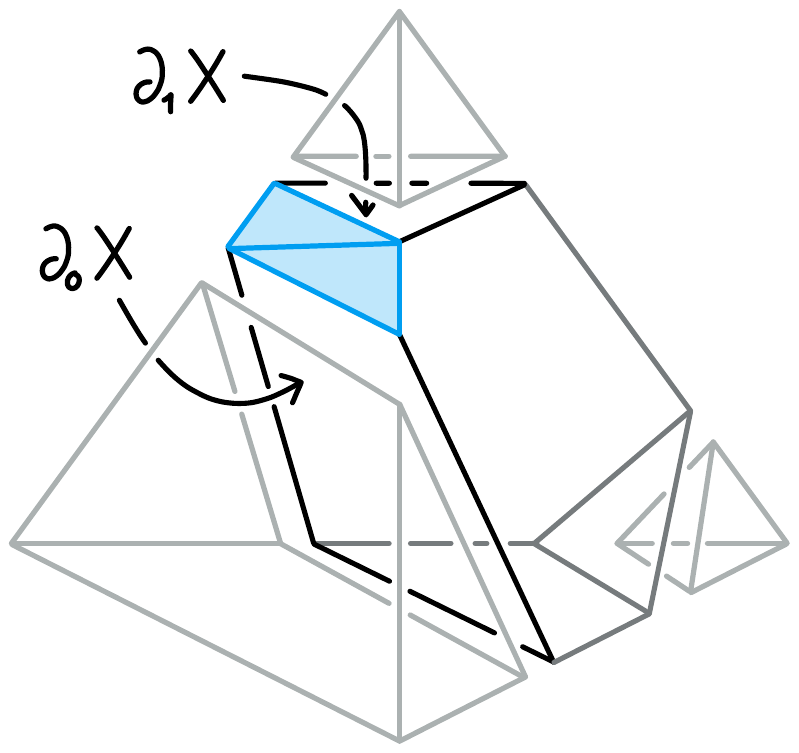}
\caption{}
\label{fig:triangulating the guts:gut piece:2}
\end{subfigure}
\hfill
\begin{subfigure}{0.45\textwidth}
\centering
\myfiguresource[scale=0.3]{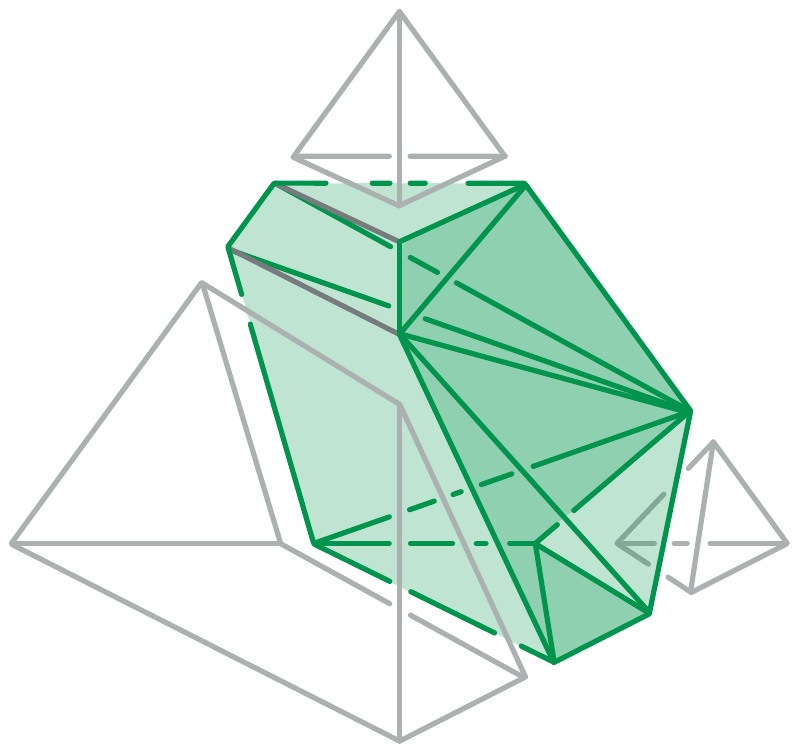}
\caption{}
\label{fig:triangulating the guts:gut piece:3}
\end{subfigure}
\caption{\subref{fig:triangulating the guts:simplicial push-off} The positive push\=/off of $F$ maps each normal disc of $F$ to a normal disc of $\boundary_0 M'$ simplicially. The arrow on $F$ represents the transverse orientation of $F$, while the arrow from a triangle of $F$ to a triangle of $\boundary_0 M'$ represents the positive push\=/off.
\subref{fig:triangulating the guts:gut piece:1} The faces of a gut piece that lie on $\hboundary X$ are triangulated according to the triangulation of $\hboundary M'$.
\subref{fig:triangulating the guts:gut piece:2} The faces of the gut piece that lie in $Y$ and are vertical in $\vboundary X$ are four\=/sided and triangulated with two triangles.
\subref{fig:triangulating the guts:gut piece:3} The remaining faces of the gut piece are triangulated by coning over vertices.}
\end{myfigurepage}

By repeating this construction for each gut piece of $M'$, we obtain a triangulation $\RRR$ of $\abstr{X}$.
Note that triangulations of different gut pieces are compatible along their common faces, since these faces are triangulated by coning over the same vertex.
By construction, the triangulation $\RRR$ is in fact a pre\=/sutured triangulation, compatible with the pre\=/sutured manifold structure of $\abstr{X}$.
This triangulation induces a pre\=/sutured triangulation of $X$, compatible with the pre\=/sutured sub\=/$3$\=/complex structure of $X$; we will blur the distinction between the two, thinking of $\RRR$ as both a triangulation of $\abstr{X}$ and a triangulation of $X$.
The restriction of $\RRR$ to $\hboundary\abstr{X}$ is flapless by \zcref{rmk:the triangulation of a normal surface is flapless}.
Moreover, for each gut piece $P$, the vertical subset $V_P$ of $\vboundary\abstr{X}$ is simplicial.

We have already remarked how the number of gut pieces of $M'$ is at most linear in $t$.
The following elementary result gives more granular bounds concerning the triangulation of the guts.

\begin{proposition}[Bounds on the triangulation of the guts]
\label{thm:triangulation of the guts is small}
Let $\TTT$ be a triangulation of a compact oriented $3$\=/manifold $M$ with $t$ tetrahedra.
Let $F$ be a transversely oriented normal surface in $M$, and let $M'=M\cut F$.
Then the pre\=/sutured triangulation $\RRR$ of the guts $X$ of $M'$ satisfies:
\begin{enumroman}
\item $\card{\RRR}\le 50t$;
\item $\area{\hboundary X}\le 32t$;
\item $\area{\vboundary X}\le 36t$;
\end{enumroman}
\end{proposition}
\begin{proof}
Fix an (abstract) tetrahedron $T$ of $\TTT$.
It is easy to see that $\hboundary X\cap T$ is a union of $n_t\le8$ normal triangles ($0$ or $2$ per vertex of $T$) and $n_q\in\{0,2\}$ normal quadrilaterals.
As a consequence, the portion of $\hboundary X$ contained in $T$ is triangulated with
\[
3n_t+4n_q\le32
\]
triangles; this proves the second claim.

Let us now focus on a face $R$ of $T$.
Denote by $r$ the intersection $\hboundary X\cap R$.
By our previous remark on $\hboundary X\cap T$, we see that $r$ is a union of $m\le 8$ normal arcs in $R$.
A simple counting argument reveals that $X\cap R$ has $1+m/2$ polygonal components, and the total number of sides of these components is $3+2m$.
Since each component of $X\cap R$ is triangulated by coning over a vertex, the total number of triangles in $X\cap R$ is
\begin{equation}
\label{eqn:triangulation of the guts is small:triangles in R}
\area{X\cap R}=3+2m-2(1+m/2)=1+m\le 9.
\end{equation}
Since each tetrahedron has $4$ faces, and $\vboundary X\subs\TTT^{(2)}$, this proves the third claim.

Finally, let us shift our attention back to the (arbitrary) tetrahedron $T$.
A counting argument similar to the above shows that $X\cap T$ has $1+n_t/2+n_q/2$ polyhedral components -- namely, the gut pieces of $T$.
We have already seen that $\hboundary X\cap T$ is triangulated with $3n_t+4n_q$ triangles.
The intersection $\hboundary X\cap\boundary T$ consists of $3n_t+4n_q$ normal arcs.
By applying \zcref{eqn:triangulation of the guts is small:triangles in R} to each face of $T$, we see that $X\cap\boundary T$ is triangulated with $4+3n_t+4n_q$ triangles.
Therefore, the total number of triangles in the triangulation of $\boundary(X\cap T)$ is
\[
(3n_t+4n_q)+(4+3n_t+4n_q)=4+6n_t+8n_q.
\]
Since each vertex of the triangulation of $X\cap T$ has degree at least $3$, and each component of $X\cap T$ is triangulated by coning the triangulation of its boundary over a vertex, we conclude that the number of tetrahedra of $X\cap T$ is at most
\begin{align*}
\area{\boundary(X\cap T)}-3\card{X\cap T}&=4+6n_t+8n_q-3\left(1+\frac{1}{2}n_t+\frac{1}{2}n_q\right)\\
&=1+\frac{9}{2}n_t+\frac{13}{2}n_q\le1+36+13=50.
\end{align*}
This concludes the proof of the first claim.
\end{proof}

We remark that the construction of the triangulation $\RRR$ of $\abstr{X}$ is very explicit.
Given the bounds of \zcref{thm:triangulation of the guts is small}, we can construct this triangulation algorithmically, in polynomial time in $t$ and $\log w(F)$.

Finally, we introduce some notation for neighbourhoods of sub\=/$2$\=/complexes of $F$ in $M'$.
Let $T$ be a triangle of $F$.
If $\pospushoff{T}$ is contained in a parallelity piece of $M'$, then we write $\posnbhd{T}$ for the unique parallelity piece containing $\pospushoff{T}$.
If, instead, the positive push\=/off of $T$ is contained in a gut piece of $M'$, then we write $\posnbhd{T}$ for the unique gut component containing $\pospushoff{T}$.
For a sub\=/$2$\=/complex $G$ of $F$, we let
\[
\posnbhd{G}=\bigcup_{T\subs G}\posnbhd{T},
\]
where the union is taken over all triangles $T$ of $G$; note that $\posnbhd{G}$ is a sub\=/$3$\=/complex of $M'$.
An analogous definition is given for $\negnbhd{G}$ in terms of the negative push\=/off.

\subsection{Parallelity bundle of least-weight fibres}

\begin{definition}[Vertical surface in interval bundle]
Let $F$ be a compact orientable surface.
A surface $A$ properly embedded in $F\times[0,1]$ is \emph{vertical} if it can be isotoped preserving $\boundary F\times[0,1]$ to a union of interval fibres of the form $\{x\}\times[0,1]$.
\end{definition}

In the context of the proof of \zcref{thm:parallelity bundle of a least-weight fibre}, we will need to work in the general setting of \emph{$3$\=/manifolds with boundary pattern}, of which we recall the fundamentals here; we refer the reader to \cite[Section~3.3.2]{matveev:algorithmic-topology-classification} for the full details.
Let $M$ be a compact $3$\=/manifold, and let $\Gamma\subs\boundary M$ be a union of disjoint curves, called the \emph{boundary pattern}.
In fact, the boundary pattern $\Gamma$ is allowed to be an arbitrary graph, but we will not need this level of generality here.
A subset of $M$ is \emph{clean} if it is disjoint from $\Gamma$.
A surface $F\subs M$ is said to be \emph{properly embedded} in $(M,\Gamma)$ if it is properly embedded in $M$ and transverse to $\Gamma$.
A boundary\=/compressing disc $D$ is \emph{trivial} if $\boundary D\cap F$ cuts off a clean disc from $F$.
If $F$ admits a non\=/trivial clean boundary\=/compressing disc in $M$, we say that it is \emph{boundary\=/compressible} in $(M,\Gamma)$; otherwise, it is \emph{boundary\=/incompressible} in $(M,\Gamma)$.

\begin{lemma}[Essential surfaces in interval bundles are vertical]
\label{thm:essential surfaces in interval bundles are vertical}
Let $M=F\times[0,1]$ for some compact orientable surface $F$, and consider the boundary pattern $\Gamma=\boundary F\times\{0,1\}$.
Let $A$ be a surface properly embedded in $(M,\Gamma)$ that is incompressible and boundary\=/incompressible.
Suppose that no component of $A$ is a surface without boundary or a disc intersecting $\Gamma$ at most twice, and that no component of $\boundary A$ is contained in $\vboundary M$.
Then $A$ is vertical in $M$.
\end{lemma}
\begin{proof}
We can assume without loss of generality that $A$ is connected, and moreover that $M$ is connected.
If any component of $A\cap\vboundary M$ is an arc with both endpoints on $\boundary_0 M$ or $\boundary_1 M$, then we immediately see that $A$ is boundary\=/compressible.
Therefore, every component of $A\cap\vboundary M$ is a vertical arc.

We now wish to invoke \cite[Proposition~5.6]{johannson:homotopy-equivalences-3} to conclude that $A$ must be vertical.
To this aim, we must first translate the definition of ``essential'' from \cite[Definition~3.1]{johannson:homotopy-equivalences-3} to our setting.
We say that a surface $S$ properly embedded in $(M,\Gamma)$ is \emph{essential in the sense of \cite{johannson:homotopy-equivalences-3}} if
\begin{itemize}
\item $S$ is incompressible in $M$, and
\item for any (not necessarily clean) boundary\=/compressing disc $D$ for $S$ in $M$ that intersects $\Gamma$ at most once, the arc $D\cap S\subs S$ cuts a disc off of $S$ that intersects $\Gamma$ at most once.
\end{itemize}
Then \cite[Proposition~5.6]{johannson:homotopy-equivalences-3} implies the following: if $F$ is not a sphere or a disc, then any surface $S$ properly embedded in $(M,\Gamma)$ that is essential in the sense of \cite{johannson:homotopy-equivalences-3} is horizontal or vertical in $M$, provided that no component of $S$ is a sphere or a disc intersecting $\Gamma$ at most three times; in fact, ``three'' can be replaced by ``two'' here, since $S$ will necessarily intersect $\Gamma$ an even number of times.
Note that, if $F$ were a sphere or a disc, then $A$ would necessarily be compressible or boundary\=/compressible.
Moreover, the surface $A$ cannot be horizontal, because $\boundary A$ intersects $\hboundary M$ by assumption.

Therefore, all we need to show is that $A$ is essential in the sense of \cite{johannson:homotopy-equivalences-3}.
We already know that $A$ is incompressible in $M$.
Let $D$ be a boundary\=/compressing disc for $A$ in $M$ that intersects $\Gamma$ at most once.
If $D$ is clean, then by boundary\=/incompressibility of $A$ we know that $D\cap A$ cuts a clean disc off of $A$.
Otherwise, the boundary of $D$ is naturally split into three non\=/empty arcs: let $a=D\cap A$, $b_v=D\cap\vboundary M$, and $b_0=D\cap\hboundary M$.
Without loss of generality, assume that $b_0\subs\boundary_0 M$.
Let $c$ be the subarc of $A\cap\vboundary M$ that connects an endpoint of $b_v$ to $\boundary_0 M$; the existence of such a subarc is guaranteed by our previous observation that every component of $A\cap\vboundary M$ is a vertical arc.
This situation is depicted in \zcref{fig:essential surfaces in interval bundles are vertical:a}.

The arc $b_v\cup c$ cuts a disc $D_v$ off of $\vboundary M$.
The disc $D\cup D_v$ can be isotoped slightly off of $\vboundary M$ to a clean boundary\=/compressing disc $D'$ for $A$.
Since $A$ is boundary\=/incompressible, we deduce that the arc $D'\cap A$ cuts a clean disc off of $A$.
It is then easy to conclude (see \zcref{fig:essential surfaces in interval bundles are vertical:b}) that $a$ cuts a disc off of $A$ that intersects $\Gamma$ exactly once -- in particular, the point of intersection is $\boundary c\setminus b_v$.
This concludes the proof, since $A$ is essential in the sense of \cite{johannson:homotopy-equivalences-3} and therefore vertical in $M$.\qedhere

\begin{myfigure}
\begin{subfigure}{0.45\textwidth}
\centering
\myfiguresource[scale=.3]{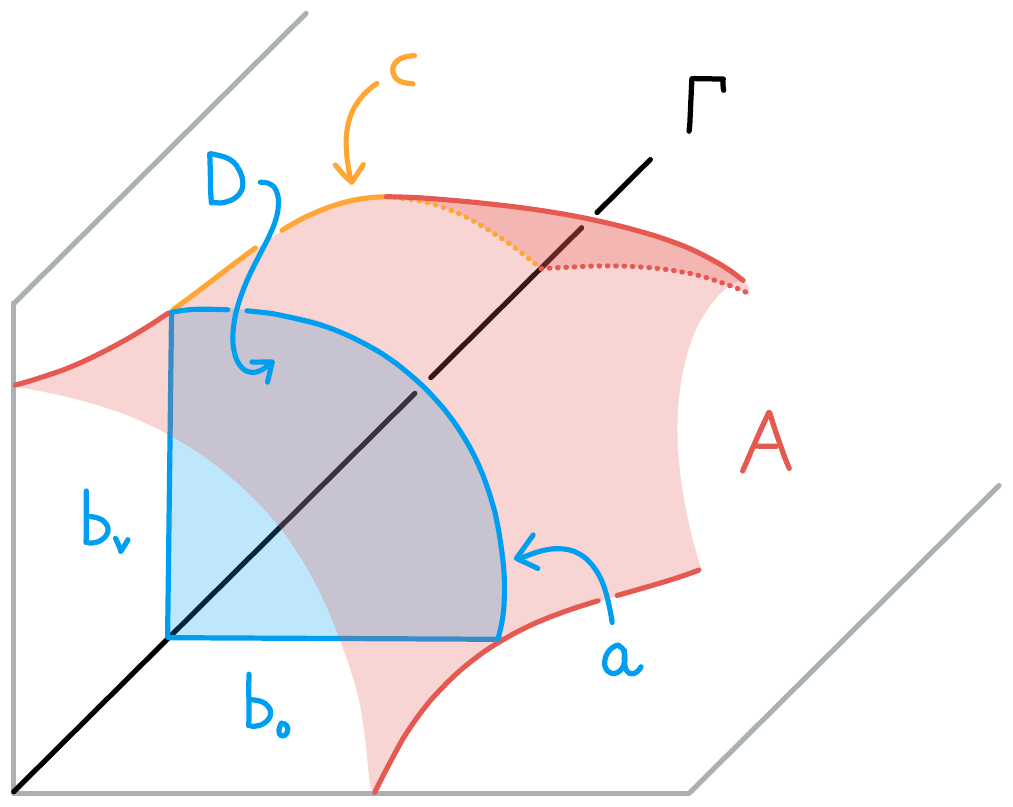}
\caption{}
\label{fig:essential surfaces in interval bundles are vertical:a}
\end{subfigure}\hfill
\begin{subfigure}{0.45\textwidth}
\centering
\myfiguresource[scale=.3]{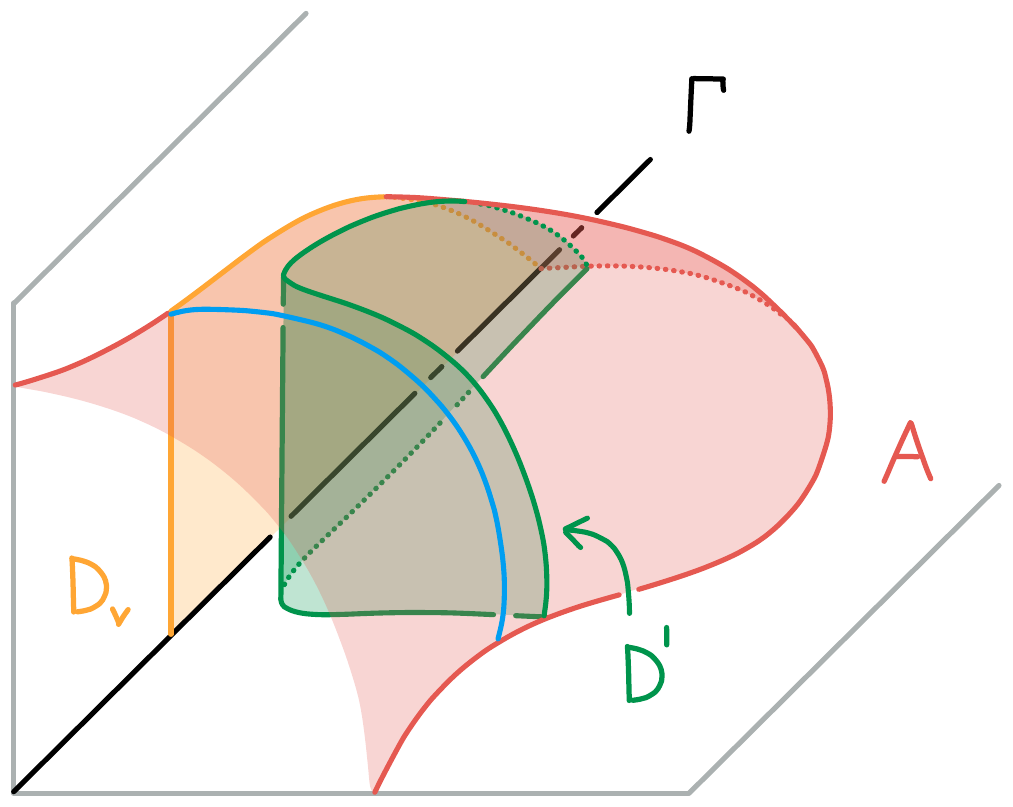}
\caption{}
\label{fig:essential surfaces in interval bundles are vertical:b}
\end{subfigure}
\caption{\subref{fig:essential surfaces in interval bundles are vertical:a} The boundary\=/compressing disc $D$ for $A$ intersects $\Gamma$ at most once.
\subref{fig:essential surfaces in interval bundles are vertical:b} The disc $D\cup D_v$ can be isotoped to a clean boundary\=/compressing disc $D'$ for $A$.}
\end{myfigure}
\end{proof}

\begin{definition}[Square in an interval bundle]
Let $M=F\times[0,1]$ for some compact orientable surface $F$.
A \emph{square} is a disc properly embedded in $M$ that intersects $\boundary F\times\{0,1\}$ transversely and exactly four times.
\end{definition}

\begin{definition}[{Cutting a $D^2\times[0,1]$ off of an interval bundle}]
Let $M=F\times[0,1]$ for some compact orientable surface $F$.
We say that an annulus or square $A$ properly embedded in $M$ \emph{cuts a $D^2\times[0,1]$ off of $M$} if there is an embedding $\map{f}{D^2\times[0,1]}{M}$ such that
\[
f(D^2\times\{0,1\})\subs\hboundary M\quad\text{and}\quad A\subs f(\boundary D^2\times[0,1])\subs A\cup\vboundary M,
\]
and moreover $f^{-1}(A)$ is vertical in $D^2\times[0,1]$.
We refer to the image of $f$ as a $D^2\times[0,1]$ cut off of $M$ by $A$.
\end{definition}

There are essentially four ways in which an annulus or square can cut a $D^2\times[0,1]$ off of $M$, as shown in \zcref{fig:cutting a disc cross I off}.

\begin{myfigure}
\foreach \i in {1,2,3,4}{%
\begin{subfigure}{0.24\textwidth}
\centering
\myfiguresource[scale=.25]{cutting-a-disc-cross-I-off-\i}
\caption{}
\label{fig:cutting a disc cross I off:\i}
\end{subfigure}
\ifnumcomp{\i}{<}{4}{\hfill}{}%
}
\caption{\subref{fig:cutting a disc cross I off:1} An annulus cutting a (possibly knotted) $D^2\times[0,1]$ off of $M$ that connects $\boundary_0 M$ and $\boundary_1 M$.
\subref{fig:cutting a disc cross I off:2} An annulus cutting a (possibly knotted) $D^2\times[0,1]$ off of $M$ that connects $\boundary_0 M$ (or $\boundary_1 M$) to itself.
\subref{fig:cutting a disc cross I off:3} A square cutting a $D^2\times[0,1]$ off of $M$ that connects $\boundary_0 M$ and $\boundary_1 M$.
\subref{fig:cutting a disc cross I off:4} A square cutting a $D^2\times[0,1]$ off of $M$ that connects $\boundary_0 M$ (or $\boundary_1 M$) to itself.}
\label{fig:cutting a disc cross I off}
\end{myfigure}

\begin{proposition}[Parallelity bundle of a least\=/weight fibre]
\label{thm:parallelity bundle of a least-weight fibre}
Let $\TTT$ be a triangulation of a compact connected orientable $3$\=/manifold $M$ that fibres over the circle, and let $F$ be a least\=/weight normal fibre of $M$.
Denote by $M'$ the $3$\=/manifold $M\cut F\homeo F\times[0,1]$, by $X$ the guts of $M'$, and by $Y$ the parallelity bundle of $M'$.
Let $A$ be a component of $\thin{X}\cap\thick{Y}$.
Then $A$ is vertical in $M'$ or cuts a $D^2\times[0,1]$ off of $M'$.
\end{proposition}
\begin{proof}
Let $\Gamma=\boundary\hboundary M'$.
We analyse two cases, depending on whether $A$ is an annulus or a square.

\paragraph{When $A$ is an annulus.}
If $A$ is not vertical, then by \zcref{thm:essential surfaces in interval bundles are vertical} it must be compressible or boundary\=/compressible in $(M',\Gamma)$.
Suppose that $A$ is compressible in $M'$, and let $D$ be a non\=/trivial compressing disc for $A$.
Compressing $A$ along $D$ yields two disjoint discs $A_1$ and $A_2$, with $\boundary A_1$ and $\boundary A_2$ contained in $\hboundary M'$.
Since $\hboundary M'$ is incompressible in $M'$, the boundary of $A_1$ bounds a disc $D_1$ in $\hboundary M'$; by irreducibility of $M'$, we deduce that the sphere $A_1\cup D_1$ bounds a $3$\=/ball $B_1$ in $M'$ (when $F$ is a sphere, the $3$\=/manifold $M'$ is not irreducible, but this statement is still true up to replacing $D_1$ with the other disc in $\hboundary M'$ having the same boundary).
Similarly, we find a $3$\=/ball $B_2$ in $M'$ whose boundary is the union of $A_2$ and a disc $D_2\subs\hboundary M'$.
If $B_1$ and $B_2$ are disjoint, then $A$ cuts a $D^2\times[0,1]$ off of $M'$, namely $B_1\cup B_2\cup\NNN(D)$; this situation is depicted in \zcref{fig:cutting a disc cross I off:1} or \ref{fig:cutting a disc cross I off:2}.
If, instead, one of the two $3$\=/balls is contained in the other -- say, $B_1\subs B_2$, then we consider the surface
\[
F'=(\boundary_iM'\setminus D_2)\cup A\cup D_1,
\]
where $\boundary_iM'$ is the component of $\hboundary M'$ containing $D_1$ and $D_2$; see \zcref{fig:trivial annular simplification}.

\begin{myfigure}
\begin{subfigure}{0.45\textwidth}
\centering
\myfiguresource[scale=.27]{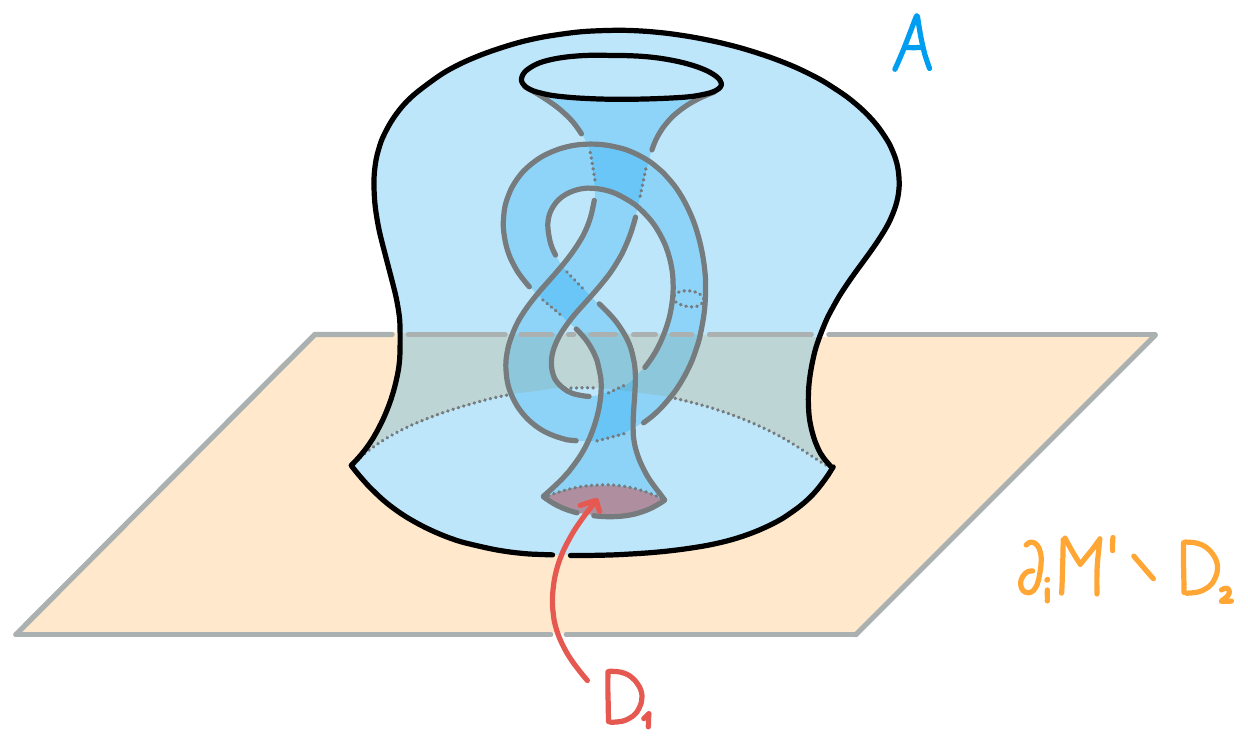}
\caption{}
\label{fig:trivial annular simplification}
\end{subfigure}\hfill
\begin{subfigure}{0.45\textwidth}
\centering
\myfiguresource[scale=.27]{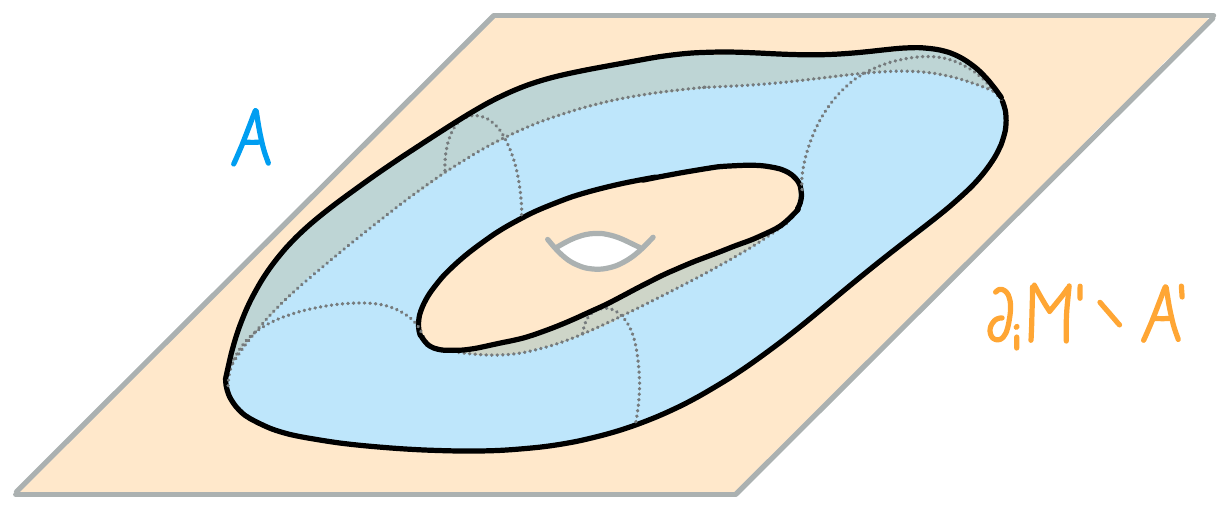}
\caption{}
\label{fig:essential annular simplification}
\end{subfigure}
\caption{\subref{fig:trivial annular simplification} In a trivial annular simplification, the surface $F'$ is the union of three pieces: the surface $\boundary_iM'\setminus D_2$, the annulus $A$, and the disc $D_1$.
\subref{fig:essential annular simplification} In an essential annular simplification, the surface $F'$ is the union of two pieces: the surface $\boundary_iM'\setminus A'$ and the annulus $A$.}
\end{myfigure}

This construction is called \emph{trivial annular simplification}, and is described in \cite[\S6.4]{lackenby-purcell:triangulation-complexity-fibred}.
Note that $F'$ is isotopic to $\boundary_iM'$ in $M$, and hence to $F$.
The annulus $A$ does not intersect the edges of $\TTT$; therefore, we see that
\[
w(\boundary_iM')-w(F')=\lvert(D_2\setminus D_1)\cap\TTT^{(1)}\rvert,
\]
where weights are taken with respect to $\TTT$.
However, since $w(\boundary_iM')=w(F)$, by minimality of $F$ we deduce that $D_2\setminus D_1$ must be disjoint from the $1$\=/skeleton of $\TTT$, and moreover that $w(F')=w(F)$.
Recall that $A$ is a component of $\thin{X}\cap\thick{Y}$.
If we let $\pi$ be the natural projection of $\vboundary\thick{Y}$ onto $\vboundary Y$, induced by the retraction $\umap{\thick{Y}}{Y}$, then $\pi(A)$ is a simplicial subset of $\vboundary X\cap\vboundary Y$; in particular, it will contain a component $e$ of the intersection $M'\cap\TTT^{(1)}$.
Note that, since $D_2\setminus D_1$ cannot intersect $e$, we necessarily have that $e$ has an endpoint on $D_1$ and the other on $\boundary_iM'\setminus D_2$.
Let $a$ be a vertical arc in $A$ such that $\pi(a)=e$.
There is an isotopy of $M'$ that takes $a$ to $e$ and then pushes it slightly past $e$; this isotopy takes $F'$ to a surface $F''$ such that
\[
F''\cap\TTT^{(1)}=(F'\cap\TTT^{(1)})\setminus e.
\]
See \zcref{fig:parallelity bundle of a least-weight fibre:isotopy} for a depiction of this isotopy.
The surface $F''$ is isotopic to $F$ in $M$, and has weight strictly less than that of $F$, contradicting the minimality of $F$.

\begin{myfigure}
\myfiguresource[scale=.3]{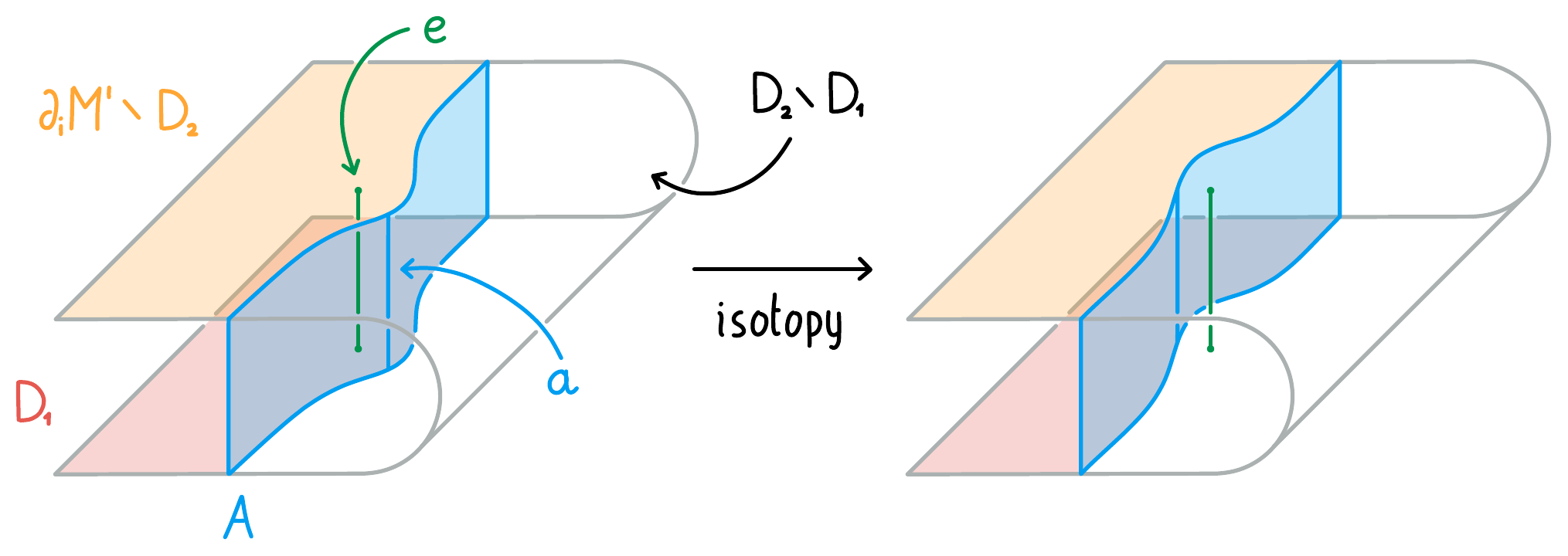}
\caption{Dragging $a$ to $e$ and then past it produces an isotopy of $M'$ that sends the surface $F'=(\boundary_iM'\setminus D_2)\cup A\cup D_1$ to a surface $F''$ that does not intersect $e$ and, hence, has weight strictly less than that of $F'$.}
\label{fig:parallelity bundle of a least-weight fibre:isotopy}
\end{myfigure}

Suppose now that $A$ is incompressible and boundary\=/compressible in $(M',\Gamma)$, and let $D$ be a non\=/trivial clean boundary\=/compressing disc for $A$.
Compressing $A$ along $D$ yields a disc $A_1$, with $\boundary A_1$ contained in $\hboundary M'$.
Like above, we find that $A_1$ cobounds a $3$\=/ball $B_1$ with $\hboundary M'$.
We see that $D$ cannot be contained in $B_1$, since $A$ is incompressible in $M'$.
Therefore, the annulus $A$ cobounds a solid torus (namely, $B_1\cup\NNN(D)$) with an annulus $A'$ in $\hboundary M'$.
Denote by $\boundary_iM'$ the component of $\hboundary M'$ containing $A'$, and consider the surface
\[
F'=(\boundary_iM'\setminus A')\cup A;
\]
see \zcref{fig:essential annular simplification}.
This is an instance of \emph{essential annular simplification} as described in \cite[\S6.4]{lackenby-purcell:triangulation-complexity-fibred}.
Note that $F'$ is isotopic to $\boundary_iM'$ in $M$, and hence to $F$.
Moreover, the same argument as above shows that $F'$ can be isotoped to a general position surface with weight strictly less than that of $F$, contradicting the minimality of $F$.

\paragraph{When $A$ is a square.}
If $A$ is not vertical, then by \zcref{thm:essential surfaces in interval bundles are vertical} it must be boundary\=/compressible in $(M',\Gamma)$.
Let $D$ be a non\=/trivial clean boundary\=/compressing disc for $A$.
Compressing $A$ along $D$ yields two disjoint discs $A_1$ and $A_2$, with $\boundary A_1$ and $\boundary A_2$ each intersecting $\Gamma$ twice.
It is easy to see that $A_1$ and $A_2$ cobound $3$\=/balls -- say, respectively, $B_1$ and $B_2$ -- with $\boundary M'$.

Suppose first that the boundary\=/compressing disc $D$ intersects $\hboundary M'$, and is therefore disjoint from $\vboundary M'$.
If one of the two $3$\=/balls is contained in the other, say $B_1\subs B_2$, then $A$ cuts a $D^2\times[0,1]$ off of $M'$, namely $B_2\setminus(B_1\cup\NNN(D))$; this situation is depicted in \zcref{fig:cutting a disc cross I off:4}.
Otherwise, the two $3$\=/balls are disjoint.
Let $B$ be the closure of the component of $M'\setminus A$ intersecting $B_1$ and $B_2$; then $B$ is itself a $3$\=/ball, and it is the union of $B_1$, $B_2$, and a neighbourhood of $D$.
Let $A'=B\cap\hboundary M'$, and let $\boundary_iM'$ be the component of $\boundary_h M'$ containing $A'$.
Consider the surface
\[
F'=(\boundary_iM'\setminus A')\cup A.
\]
Like above, we can argue that $F'$ is isotopic to $F$ in $M$, and it can be isotoped to a general position surface with weight strictly less than that of $F$, contradicting the minimality of $F$.

Suppose now that the boundary-compressing disc $D$ intersects $\vboundary M'$ -- and is therefore disjoint from $\hboundary M'$.
If $B_1\subs B_2$ or $B_2\subs B_1$, then in fact $A$ also admits a boundary\=/compressing disc that intersects $\hboundary M'$, and we reduce to the previous case.
Otherwise, the two $3$\=/balls are disjoint, and it is then easy to see that $A$ cuts a $D^2\times[0,1]$ off of $M'$, namely $B_1\cup B_2\cup\NNN(D)$; this situation is depicted in \zcref{fig:cutting a disc cross I off:3} or \ref{fig:cutting a disc cross I off:4}.
\end{proof}

\section{The certificate}

% Remark on the logical order of things

\subsection{Algorithms for interval bundles}

Consider the following decision problem.

\begin{algoproblem}[\problemname{Sutured interval bundle recognition}]
\label{prb:interval bundle recognition}
\Input{a pre\=/sutured triangulation $\TTT$ of a compact $3$\=/manifold.}
\Output{whether $\TTT$ is a suitable pre\=/sutured triangulation of $F\times[0,1]$ for some compact orientable surface $F$.}
The size of the input is measured by the number of tetrahedra in $\TTT$.
\end{algoproblem}

It turns out that a positive answer to this question can be certified in polynomial time.

\begin{theorem}[{\nameref{prb:interval bundle recognition} is in \NP}]
\label{thm:interval bundle recognition is in NP}
The problem \nameref{prb:interval bundle recognition} is in \NP{}.
\end{theorem}

This result appears as Theorem 12.1 in \cite{lackenby:efficient-certification-knottedness}, stated in greater generality for handle structures instead of pre\=/sutured triangulations.
In fact, to be precise, since we are given as input a pre\=/sutured triangulation instead of a \emph{sutured triangulation}, a pre\=/processing step is required, to verify that the given pre\=/sutured triangulation is indeed a sutured triangulation.
In practice, for a given pre\=/sutured triangulation of a $3$\=/manifold $M$, one simply needs to check that $\vboundary M$ is a union of annuli, each of which intersects both $\boundary_0 M$ and $\boundary_1 M$; moreover, that $\boundary_0 M$ and $\boundary_1 M$ both intersect each component of $M$.
These checks can be easily carried out in polynomial time in the size of the triangulation.
More crucially, the reader should be aware that the published proof of \cite[Theorem~12.1]{lackenby:efficient-certification-knottedness} contains a gap, in that it only works for sutured triangulations with non\=/empty sutures -- in other words, when the surface $F$ has non\=/empty boundary.
However, this issue can easily be fixed by cutting the $3$\=/manifold along a vertical normal annulus with bounded weight, which can be provided as part of the certificate; this is discussed at the beginning of \cite[Section~11]{lackenby-schleimer:recognising-elliptic-manifolds}.

It should come as no surprise that the minimum number of tetrahedra needed to triangulate an interval bundle over a planar surface $F$ is linear in the number of boundary components of $F$.
For our fibredness certificate, we will need a more refined version of this estimate, that bounds the number of tetrahedra needed to extend a given triangulation of a vertical subset of $\boundary F\times[0,1]$ to a triangulation of $F\times[0,1]$.
The bound is linear in the number of triangles of the prescribed triangulation, provided that the vertical subset touches every component of $\boundary F\times[0,1]$.

\begin{proposition}[Triangulating interval bundles over planar surfaces efficiently]
\label{thm:triangulating interval bundles efficiently}
Let $M=F\times[0,1]$ for some compact orientable planar surface $F$.
Let $\Omega\subs\boundary F$ be a union of finitely many circles or closed intervals with non\=/empty interiors; suppose that $\Omega$ intersects every component of $\boundary F$.
Let $\TTT_0$ be a triangulation of $\Omega\times[0,1]\subs\vboundary M$ with $t_0$ triangles, such that every point in $\boundary\Omega\times\{0,1\}$ is a vertex of $\TTT_0$.
Then there is a triangulation $\TTT$ of $M$ such that:
\begin{enumarabic}
\item the restriction of $\TTT$ to $\Omega\times[0,1]$ is $\TTT_0$;
\item the restriction of $\TTT$ to $\hboundary M$ is flapless;
\item the number of tetrahedra of $\TTT$ is at most $26t_0$.
\end{enumarabic}
\end{proposition}
\begin{proof}
It is not restrictive to assume that $M$ is connected.
We start by finding a triangulation $\TTT_1$ of $\vboundary M$ that agrees with $\TTT_0$ on $\Omega\times[0,1]$.
Denote by $b_1,\dots,b_m$ the components of $\boundary F$; for each $1\le i\le m$, we construct a triangulation of $b_i\times[0,1]$.
Let $\Omega_i=\Omega\cap b_i$, and let $s_i$ be the number of triangles of $\TTT_0$ contained in $\Omega_i\times[0,1]$.
If $b_i\subs\Omega$, then we simply triangulate $b_i\times[0,1]$ by restricting $\TTT_0$.
Otherwise, denote by $k$ the number of components of $\Omega_i$, and by $h$ the number of edges of $\TTT_0$ that are contained in $\boundary\Omega_i\times[0,1]$; a simple counting argument shows that $s_i\ge h$.
The closure of the complement of $\Omega_i\times[0,1]$ in $b_i\times[0,1]$ is a union of topological discs; each of these discs has a natural cell structure of its boundary, with one edge on $b_i\times\{0\}$, one on $b_i\times\{1\}$, and the remaining edges on $\boundary\Omega_i\times[0,1]$ coming from $\TTT_0$.
The number of these discs is $k$, and the combined simplicial length of their perimeters is exactly $2k+h$.
Therefore, their union can be triangulated with exactly $h$ triangles, as shown in \zcref{fig:triangulating interval bundles efficiently:triangulating vboundary}.

\begin{myfigure}

\myfiguresource[scale=.3]{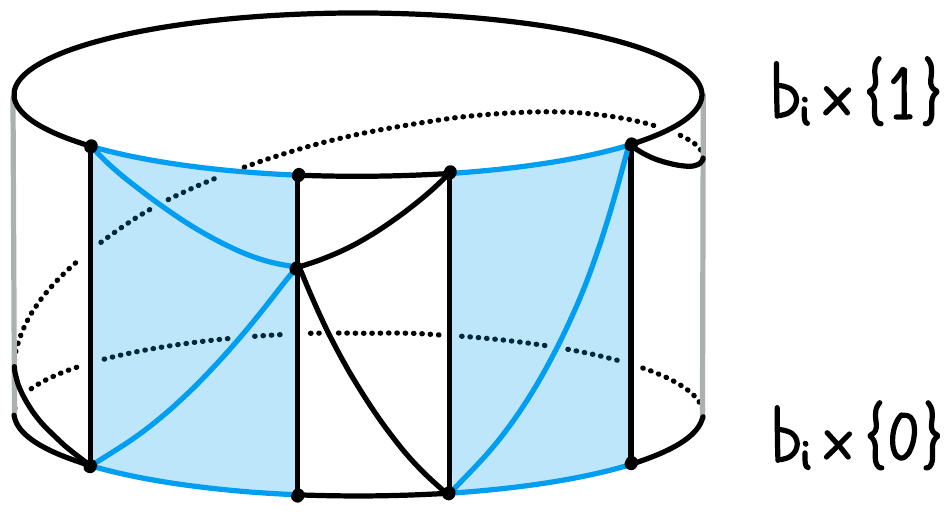}
\caption{The triangulation $\TTT_1$ restricted to $b_i\times[0,1]$.
The edges and triangles of $\TTT_1$ that are not in $\TTT_0$ are highlighted.}
\label{fig:triangulating interval bundles efficiently:triangulating vboundary}
\end{myfigure}

In conclusion, we have constructed a triangulation of $b_i\times[0,1]$ with at most $2s_i$ triangles (recall that $h\le s_i$) that agrees with $\TTT_0$ on $\Omega_i\times[0,1]$.
Repeating this construction for every boundary component $b_i$ yields a triangulation $\TTT_1$ of $\vboundary M$ that agrees with $\TTT_0$ on $\Omega\times[0,1]$ and has
\begin{equation}\label{eqn:triangulating interval bundles efficiently:t1}
t_1\le2s_1+\dots+2s_m=2t_0
\end{equation}
triangles.

The next step is to find a triangulation of a union $S$ of vertical squares in $M$ that agrees with $\TTT_1$ on $S\cap\vboundary M$, such that $M\cut S$ is a $3$\=/ball.
To this aim, for each $1\le i\le m$, fix a simplicial arc $a_i\subs b_i\times[0,1]$ that is vertical in $M$.
Denote by $\ell_i$ the simplicial length of $a_i$, and -- without loss of generality -- assume that $\ell_1\le\ell_i$ for every $2\le i\le m$.
Find squares $S_2,\dots,S_m$ in $M$ such that $S_i\cap\vboundary M=a_1\cup a_i$ for each $2\le i\le m$, and $S_i\cap S_j=a_1$ for all $1\le i<j\le m$.
The boundary of each $S_i$ has a natural cell structure with exactly $\ell_1+\ell_i+2$ edges; therefore, we can triangulate $S_i$ with $\ell_1+\ell_i$ triangles.
Let $\TTT_2$ be the triangulation of $S=S_2\cup\dots\cup S_m$ obtained from this construction, as shown in \zcref{fig:triangulating interval bundles efficiently:S}. The number $t_2$ of triangles of $\TTT_2$ is exactly
\begin{equation}\label{eqn:triangulating interval bundles efficiently:t2}
t_2=(m-1)\ell_1+\ell_2+\dots+\ell_m.
\end{equation}

Finally, we construct a triangulation of $M$, by exploiting the fact that $B=M\cut S$ is a topological $3$\=/ball.
The triangulations $\TTT_1$ and $\TTT_2$ induce a cell structure on the boundary of $B$, consisting of $t_1+2t_2$ triangles and two polygons $P_0$ and $P_1$, where $P_i=B\cap\boundary_i M$ for $i\in\{0,1\}$; this is depicted in \zcref{fig:triangulating interval bundles efficiently:B}.
Note that, for each $i\in\{0,1\}$, the number of edges of $P_i$ is exactly $r_i+2m-2$, where $r_i$ is the number of edges of $\TTT_1$ that are contained in $\boundary_i M$.
We can then triangulate each $P_i$ with $r_i+2m-2$ triangles, by coning over a point in the interior of $P_i$; this construction guarantees that the triangulation of each $P_i$ is flapless, and it produces a triangulation of $\boundary B$ with exactly
\[
t_1+2t_2+r_0+r_1+4m-4
\]
triangles.
By coning over a vertex of $B$, we obtain a triangulation $\TTT_3$ of $B$ whose number $t_3$ of tetrahedra is at most
\begin{equation}\label{eqn:triangulating interval bundles efficiently:bound on t3}
t_3\le t_1+2t_2+r_0+r_1+4m-5.
\end{equation}
We conclude by gluing this triangulation of $B$ along $S$ to obtain a triangulation $\TTT$ of $M$.

\begin{myfigure}
\begin{subfigure}[b]{.42\linewidth}
\centering
\myfiguresource[scale=.29]{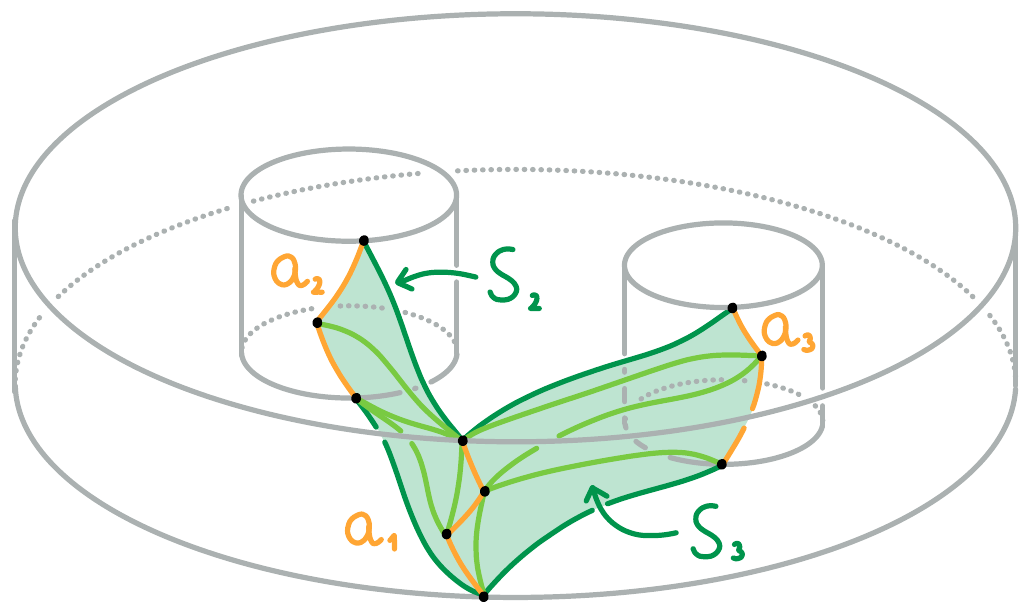}
\caption{}
\label{fig:triangulating interval bundles efficiently:S}
\end{subfigure}\hfill
\begin{subfigure}[b]{.48\linewidth}
\centering
\myfiguresource[scale=.29]{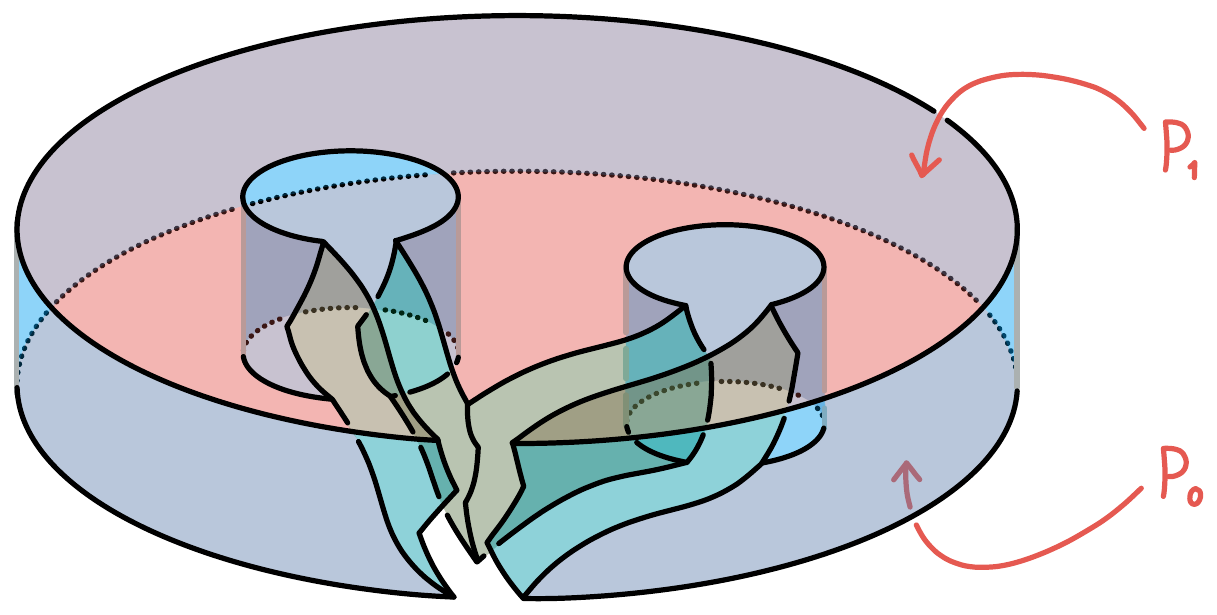}
\caption{}
\label{fig:triangulating interval bundles efficiently:B}
\end{subfigure}
\caption{\subref{fig:triangulating interval bundles efficiently:S} For each $2\le i\le m$, the square $S_i$ can be triangulated with $\ell_1+\ell_i$ triangles. \subref{fig:triangulating interval bundles efficiently:B} The topological $3$\=/ball $B=M\cut S$ has a natural cell structure, consisting of triangles coming from $\TTT_1$ and (two copies of) $\TTT_2$, and two polygons $P_0$ and $P_1$.}
\end{myfigure}

By construction, the restriction of $\TTT$ to $\Omega\times[0,1]$ is exactly $\TTT_0$, and the restriction of $\TTT$ to $\hboundary M$ is flapless, so all that is left is bounding the number of tetrahedra of $\TTT$ -- or, equivalently, of $\TTT_3$.
Let $s$ be the number of edges of $\TTT_1$; note that, trivially, we have the bound $s\le 3t_1$.
We also remark that, equally trivially, we have that $r_i\ge m$ for $i\in\{0,1\}$.
Elementary counting arguments show that
\begin{align}
\ell_1+\dots+\ell_m+r_0+r_1&\le s\le 3t_1,
\label{eqn:triangulating interval bundles efficiently:sum of l and r}\\
m\ell_1\le\ell_1+\dots+\ell_m&\le 3t_1-2m.
\label{eqn:triangulating interval bundles efficiently:m times l1}
\end{align}
We can finally bound the number of tetrahedra of $\TTT_3$ as follows:
\begin{alignaside*}
t_3&\le t_1+2t_2+r_0+r_1+4m-5&\aside{by \zcref{eqn:triangulating interval bundles efficiently:bound on t3}}\\
&=t_1+2(m-1)\ell_1+2\ell_2+\dots+2\ell_m+r_0+r_1+4m-5&\aside{by \zcref{eqn:triangulating interval bundles efficiently:t2}}\\
&\le 4t_1+(2m-3)\ell_1+\ell_2+\dots+\ell_m+4m-5&\aside{by \zcref{eqn:triangulating interval bundles efficiently:sum of l and r}}\\
&\le 7t_1+2(m-2)\ell_1+2m-5&\aside{by \zcref{eqn:triangulating interval bundles efficiently:m times l1}}\\
&\le13t_1-4\ell_1-2m-5&\aside{by \zcref{eqn:triangulating interval bundles efficiently:m times l1}}\\
&< 13t_1\\
&\le 26t_0&\aside{by \zcref{eqn:triangulating interval bundles efficiently:t1}.\qedhere}
\end{alignaside*}
\end{proof}

\subsection{Overview of the certificate}
\label{sec:overview of the certificate}

The remainder of this section is devoted to providing a certificate for fibredness of $3$\=/manifolds.
The precise construction of this certificate, carried out in \zcref{sec:certificate:construction}, is very technical and involved; we aim to make it less overwhelming, by providing a high level overview and some motivation in the following paragraphs.

Suppose we are given a triangulated compact orientable $3$\=/manifold $M$ with $t$ tetrahedra, that is guaranteed to be fibred.
The first piece of information that we include in the certificate is, quite naturally, a normal fibre $F$ of $M$.
To keep the certificate polynomial in size, we must ensure that the logarithm of the weight of $F$ is polynomial in $t$; the existence of such a fibre is proved in \zcref{thm:fibre of small weight,thm:toroidal fibre of small weight}.
One could imagine that simply providing the fibre would be sufficient to certify fibredness of $M$.
This is not completely unreasonable, although the burden is then shifted to the verifier, who must check that the given normal surface $F$ is indeed a fibre of $M$.
In fact, the verifier could apply \cite[Theorem~9.3]{lackenby:efficient-certification-knottedness} to cut $M$ along $F$, obtaining a pre\=/sutured manifold $M'=M\cut F$, and then check that $M'$ is a product interval bundle.
We do not know whether this last step be performed in polynomial time; however, it can be \emph{certified} in polynomial time, thanks to \cite[Theorem~12.1]{lackenby:efficient-certification-knottedness} (see the statement in \zcref{thm:interval bundle recognition is in NP}).
Therefore, a certificate consisting of a normal fibre $F$ and an auxiliary certificate (as provided by \cite[Theorem~12.1]{lackenby:efficient-certification-knottedness}) showing that $M\cut F$ is a product interval bundle would prove that $M$ is fibred, and could be verified in polynomial time.

This na\"ive approach, however, has a crucial flaw.
Namely, the size of this certificate would be polynomial in $t$ \emph{and} in the Euler characteristic of $F$, which itself is not in general polynomial in $t$, as the example we gave in \zcref{sec:introduction}.
This dependence cannot be avoided, since the minimum number of tetrahedra needed to triangulate $M\cut F$ is linear in $\abs{\chi(F)}$.
Therefore, we need a more refined strategy to certify that $F$ is a fibre of $M$ \emph{without} directly cutting $M$ along $F$.
We will see that this can be achieved, at the cost of a more complicated certificate.

For the sake of simplicity, suppose that the guts $X$ and the parallelity bundle $Y$ of $M'$ are both submanifolds of $M'$ (in the general case, we would be working with $\thin{X}$ and $\thick{Y}$ instead).
Suppose, moreover, that $M$ is closed; this assumption is not necessary, but it simplifies the exposition.
One could optimistically imagine a situation in which the annuli $X\cap Y$ are all vertical in $M'$.
This would imply that $Y$ is a product interval bundle, whose interval bundle structure agrees with that of $M'$.
The guts $X$ would also be homeomorphic as a pre\=/sutured manifold to a product interval bundle, whose interval bundle structure again agrees with that of $M'$.
In this case, certifying that $M$ is fibred -- or, equivalently, that the pre\=/sutured manifold $M'$ is a product interval bundle -- would be exceedingly easy: one simply needs to provide a proof that $X$ (with its pre\=/sutured manifold structure) is a product interval bundle, which can be done thanks to \cite[Theorem~12.1]{lackenby:efficient-certification-knottedness} (again, see the statement in \zcref{thm:interval bundle recognition is in NP}).
The key observation to keep in mind here is that, even though the parallelity bundle $Y$ could be huge, we do not need to deal with it, since it is guaranteed to be an interval bundle; conversely, we have no guarantee on the guts $X$, but the size of the triangulation of $X$ is polynomial in $t$ (see \zcref{thm:triangulation of the guts is small}).

Unfortunately, the annuli $X\cap Y$ will not in general be vertical in $M'$.
However, this optimistic situation forms the core of our certificate.
One can show, as we have done in \zcref{thm:parallelity bundle of a least-weight fibre}, that if the fibre $F$ is chosen to be least\=/weight, then each annulus in $X\cap Y$ is either vertical, or it bounds a ``tube'' in $M'$; the second case is depicted in \zcref{fig:cutting a disc cross I off:1,fig:cutting a disc cross I off:2}.
This means that we can decompose $M'$ into three submanifolds:
\begin{itemize}
\item the union $Y_v$ of the parallelity components that are vertical in $M'$;
\item the remaining parallelity components $Y_t$, that are contained in a union of tubes $T$;
\item the guts $X$.
\end{itemize}
The reader can find a graphical representation this decomposition in \zcref{fig:decomposition induced by a least-weight fibre:overview}, where $Y_v$ is denoted by $\posnbhd{F_v}$ and $Y_t$ is denoted by $\posnbhd{F_0}\cup\negnbhd{F_1}$.
The reason for the different notation is that working with subsurfaces of $F$ rather than submanifolds of $M'$ is more convenient for algorithmic purposes.
The relevant subsurfaces are $F_g$, $F_v$, $F_0,$ and $D_0$, representing the traces on $\boundary_0M'\homeo F$ of $X$, $Y_v$, $Y_t$, and $T$ respectively; the subsurfaces $F_g'$, $F_v'$, $F_1$, and $D_1'$ play the same role for $\boundary_1 M'\homeo F$.

In order to show that $M$ is fibred, we will then provide a certificate consisting of a least\=/weight fibre $F$ and a decomposition of $M'$ into the three submanifolds $X$, $Y_v$, and $Y_t$ (as explained above, this decomposition is actually given in terms of subsurfaces of $F$).
The submanifold $Y_t$ could be very large, but -- since it is a product interval bundle over planar surfaces -- it can be retriangulated with a number of tetrahedra that is linear in $t$, as shown in \zcref{thm:triangulating interval bundles efficiently}.
By attaching this triangulation to the standard triangulation of the guts $X$, we obtain a triangulated $3$\=/manifold $N$ that is homeomorphic to $X\cup Y_t$, such that the size of the triangulation of $N$ is linear in $t$ (recall that the triangulation of the guts is ``small'', as seen in \zcref{thm:triangulation of the guts is small}).
Like in the above simplified case, the parallelity components $Y_v$ could be huge, but we do not need to deal with them.
All the verifier needs to do is check that the $3$\=/manifold $N$ (with a suitable pre\=/sutured manifold structure) is a product interval bundle, which again can be done thanks to \cite[Theorem~12.1]{lackenby:efficient-certification-knottedness}.
This will guarantee that $N$ and $Y_v$ glue along their common boundary to give a product interval bundle structure to $M'$.

\begin{myfigurepage}
\myfiguresource[scale=.25]{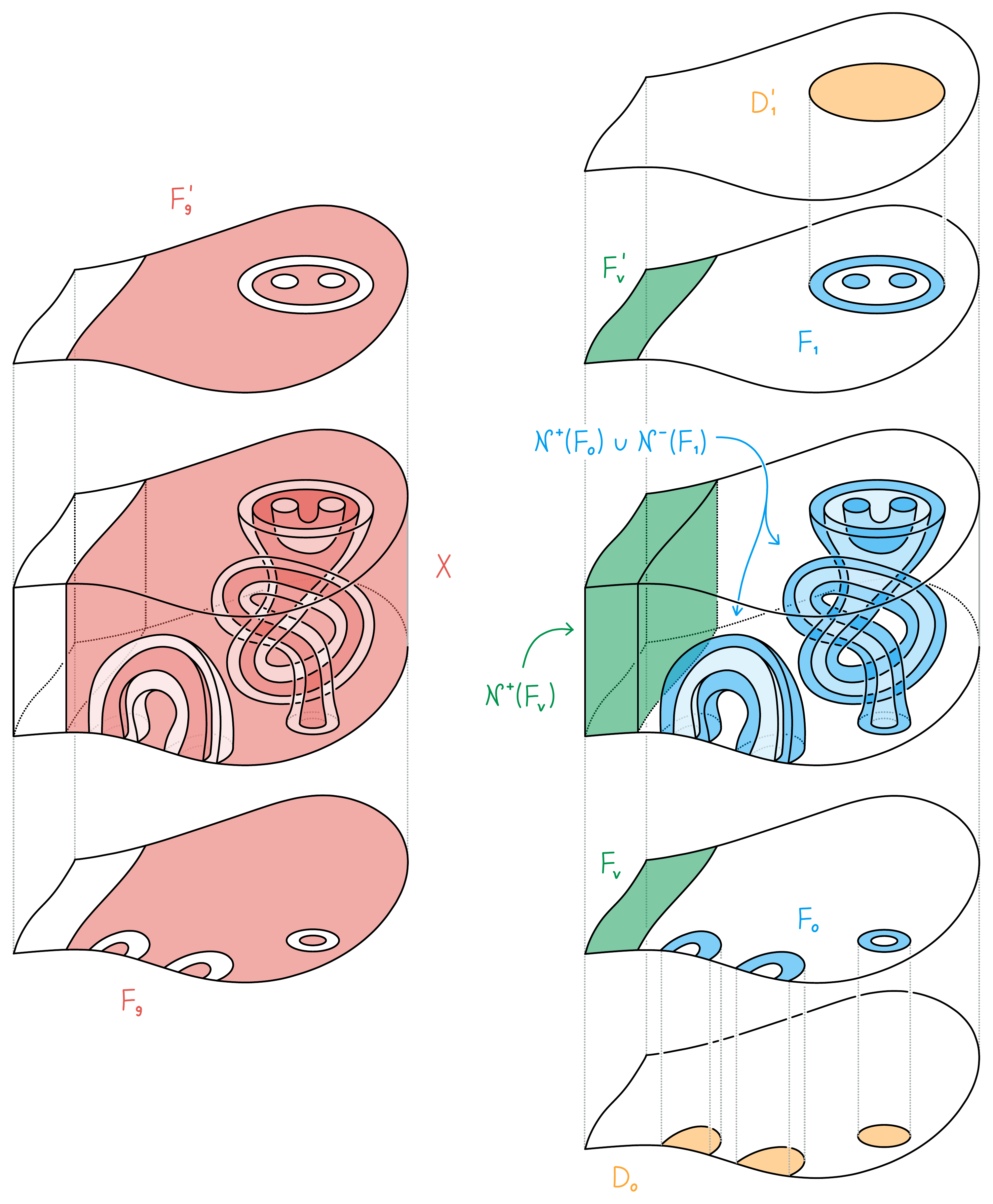}
\caption{Some objects that are part of the fibredness certificate.
The top and bottom portions of the figure represent sub\=/$2$\=/complexes of $F$, namely $F_v$, $F_0$, $F_1$, $F_g$, $F_v'$, $F_g'$, and $D_1'$.
Note that $D_1'$ is not technically part of the certificate, but it is used in its construction.
The middle row of the figure depicts a portion of $M'$.
Alluding to the fact that, by \zcref{thm:correctness of fibredness certificate}, if the certificate is valid then $F$ is a fibre of $M$, we have drawn $M'$ as an interval bundle, with $\boundary_0 M'=\pospushoff{F}$ at the bottom and $\boundary_1 M'=\negpushoff{F}$ at the top.
On the left, we show the guts $X$ of $M'$.
On the right, the parallelity bundle $Y$ of $M'$ is shown, partitioned into the sub\=/$3$\=/complexes $\posnbhd{F_v}$ and $\posnbhd{F_0}\cup\negnbhd{F_1}$.}
\label{fig:decomposition induced by a least-weight fibre:overview}
\end{myfigurepage}

To (loosely) summarise, not all parallelity components are guaranteed to be vertical in $M'$.
However, the ones that are not are contained in tubes.
We retriangulate these tubes with few tetrahedra and attach them to the guts to obtain a $3$\=/manifold $N\homeo X\cup Y_t$.
The triangulation of $N$ is given as part of the certificate, together with an auxiliary certificate (provided by \cite[Theorem~12.1]{lackenby:efficient-certification-knottedness}) showing that $N$ is a product interval bundle.
The verifier can check that $N$ is indeed a product interval bundle, and this guarantees that $M'=(X\cup Y_t)\cup Y_v$ is also a product interval bundle -- or, in other words, that $F$ is a fibre of $M$.
% This certificate can also be used to recover the monodromy of the fibration; however, we defer this result to \cite{baroni:certifying-hyperbolicity-fibred}.

\subsection{Construction of the certificate}
\label{sec:certificate:construction}

Let $\TTT$ be a triangulation of a compact connected oriented $3$\=/manifold $M$, and let $F$ be a transversely oriented least\=/weight normal fibre of $M$.
We now endeavour to construct a certificate of the fact that $F$ is a fibre of $M$; we encourage the reader to refer to \zcref{fig:decomposition induced by a least-weight fibre:overview} for a visual representation of this construction.
Let $M'=M\cut F$, and denote by $X$ and $Y$ the guts and the parallelity bundle of $M'$ respectively.
Let $F_g$ and $F_g'$ be the sub\=/$2$\=/complexes of $F$ such that $\pospushoff{F_g}=\boundary_0 X$ and $\negpushoff{(F_g')}=\boundary_1 X$.

\paragraph{The sub-2-complexes $F_v$, and $F_v'$, $F_0$, and $F_1$.}
Let $A$ be the union of the components of $\thin{X}\cap\thick{Y}$ that are vertical in $M'$, and let $A'$ be the union of the other components; by \zcref{thm:parallelity bundle of a least-weight fibre}, every component of $A'$ cuts a $D^2\times[0,1]$ off of $M'$.
Note that if a component of $\thick{Y}$ intersects $A$, then the interval bundle structure of that component agrees with the interval bundle structure of $M'$; in particular, said component does not intersect $A'$.
We let $F_v$ be the largest sub\=/$2$\=/complex of $F$ such that $\thick{\posnbhd{F_v}}$ is the union of the components of $\thick{Y}$ that intersect $A$; by the above remark, we see that $\thick{\posnbhd{F_v}}$ is disjoint from $A'$.
Similarly, denote by $F_v'$ the largest sub\=/$2$\=/complex of $F$ such that $\negnbhd{F_v'}=\posnbhd{F_v}$.
We then let $F_0$ and $F_1$ be the sub\=/$2$\=/complexes of $F$ such that the intersection of the remaining components of $\thick{Y}$ -- that is, those that intersect $A'$ -- with $\hboundary M'$ is $\thick{\pospushoff{F_0}\cup\negpushoff{F_1}}$.
We also have that $F_v$ is the closure in $F$ of $F\setminus(F_g\cup F_0)$, and $F_v'$ is the closure in $F$ of $F\setminus(F_g'\cup F_1)$.
In particular, we remark that $F_0$ and $F_v$ are disjoint, and so are $F_1$ and $F_v'$.
Moreover, we have that $F_0$ and $F_g$ have disjoint interiors, and so do $F_1$ and $F_g'$.

\paragraph{The sub-2-complexes $D_0$ and $D_1'$.}
Consider a component $C$ of $\thick{\posnbhd{F_0}}$, endowed with its intrinsic interval bundle structure; in particular, we have that $\vboundary C\subs A'\cup\vboundary M'$.
Recall that every component of $A'$ cuts a $D^2\times[0,1]$ off of $M'$.
There must be a component of $A'\cap C$ that cuts a $D^2\times[0,1]$ containing $C$ off of $M'$; otherwise, the interval bundle structure of $C$ would agree with that of $M'$, contradicting the fact that $C$ intersects $A'$.
Let $A''$ be such a component of $A'\cap C$, and denote by $B$ a $D^2\times[0,1]$ cut off by $A''$ that contains $C$.
It is clear from \zcref{fig:cutting a disc cross I off} that the intersection $\boundary_0 M'\cap B$ is $\thick{\pospushoff{E}}$ for a sub\=/$2$\=/complex $E$ of $F$ such that $\thick{E}$ is a union of $0$, $1$, or $2$ disjoint discs, each of which intersects $\boundary F$ in at most one arc.

Let $\DDD$ be the set of all the sub\=/$2$\=/complexes obtained this way; more precisely, set
\[
\DDD=\left\{E:
\begin{matrix*}[l]
\text{$C$ is a component of $\thick{\posnbhd{F_0}}$,}\\
\text{$A''$ is a component of $C\cap A'$,}\\
\text{$B$ is a $D^2\times[0,1]$ cut off of $M'$ by $A''$ such that $C\subs B$,}\\
\text{$E$ is a sub\=/$2$\=/complex of $F$ such that}\\
\text{$\thick{\pospushoff{E}}$ is a component of $\boundary_0M'\cap B$}
\end{matrix*}
\right\}.
\]
We claim that it is impossible for two sub\=/$2$\=/complexes $E_1$ and $E_2$ in $\DDD$ to intersect but not be contained one in the other.
Suppose for a contradiction that this is the case, and let $B_i=\thick{\posnbhd{E_i}}$ for $i\in\{1,2\}$, with the interval bundle structure induced by the fact that each of them is a $D^2\times[0,1]$ cut off by some non\=/vertical annulus or square, say $A_1$ and $A_2$ respectively.
Note that $\thick{E_1}$ and $\thick{E_2}$ are two discs in $F$ that intersect, are not contained one in the other, and such that $\boundary_F\thick{E_1}$ and $\boundary_F\thick{E_2}$ are disjoint.
Therefore, it follows that $F=\thick{E_1}\cup\thick{E_2}$ -- in particular, the surface $F$ is a sphere or a disc -- and hence that $M'=B_1\cup B_2$.
Moreover, the interval bundle structures of $B_1$ and $B_2$ agree on their intersection, and therefore extend to an interval bundle structure on $M'$, which has to agree with the original one on $M'$, since it has the same horizontal boundary.
However, this contradicts the fact that $A_1$ and $A_2$ are not vertical in $M'$.

We now let $D_0$ be the union of all the sub\=/$2$\=/complexes in $\DDD$.
By the claim above, any two maximal sub\=/$2$\=/complexes in $\DDD$ are disjoint, and therefore each component of $\thick{D_0}$ is a disc intersecting $\boundary F$ in a single (possibly empty) arc.
Moreover, it follows from the definition of $\DDD$ that $F_0$ is contained in $D_0$.
Finally, we show that $D_0$ is disjoint from $F_v$.
To this aim, let $E$ be an element of $\DDD$, and let $B=\posnbhd{E}$.
Note that every component of $\thick{Y}$ is either disjoint from or contained in $B$.
If, for the sake of contradiction, some component $C$ of $\thick{\posnbhd{F_v}}$ is contained in $B$, then $B$ contains a vertical arc.
It is then easy to see that this contradicts the assumption that $B$ is a $D^2\times[0,1]$ cut off by a non\=/vertical annulus or square; we refer the reader to \zcref{fig:cutting a disc cross I off} once again.
This shows that $D_0$ and $F_v$ are disjoint.

Repeating the same exact constructions for components of $\thick{\negnbhd{F_1}}$, we obtain a sub\=/$2$\=/complex $D_1'$ of $F$ such that $F_1\subs D_1'$ and $\thick{D_1'}$ is a union of disjoint discs, each of which intersects $\boundary F$ in a single arc; moreover, the sub\=/$2$\=/complex $\negpushoff{(D_1')}$ of $\boundary_1 M$ is disjoint from $\posnbhd{F_v}$.
We remark that, by constructing $D_0$ and $D_1'$, we have also indirectly proved that $\thick{F_0}$ and $\thick{F_1}$ are planar surfaces.

\paragraph{The 3-manifolds $N$, $N'$, $N''$ and the sub-2-complexes $G_0''$ and $G_1''$.}
Let $N_0'=X$, $N_0''=\posnbhd{F_0}\cup\negnbhd{F_1}$, $N_0=N_0'\cup N_0''$, $N=\abstr{N_0}$.
We also let $N'=\embedd[N_0]^{-1}(N_0')$ and $N''=\embedd[N_0]^{-1}(N_0'')$, noting that $N=N'\cup N''$, and that $N'$ and $N''$ have disjoint interiors.
Moreover, we see that $\thin{N_0}$ is the closure in $M'$ of a union of components of $M'\setminus A$.
Since $A$ is vertical in $M'$, we have that $\thin{N_0}$ is a product interval bundle over a compact orientable surface -- namely $\thin{F_g\cup F_0}$; this implies that $N$ is a product interval bundle over $G=\abstr{F_g\cup F_0}$.
This interval bundle structure has $\boundary_i N=\embedd[N_0]^{-1}(\boundary_i M')$ for $i\in\{0,1\}$.

Let $N_0''$ inherit the intrinsic pre\=/sutured subcomplex structure from the parallelity bundle of $M'$.
This induces a pre\=/sutured subcomplex structure on $N''$, with
\[
\hboundary N''=\embedd[N_0]^{-1}(\hboundary N_0'')=\embedd[N_0]^{-1}(\pospushoff{F_0}\cup\negpushoff{F_1}),\quad\vboundary N''=\embedd[N_0]^{-1}((X\cup\vboundary M')\cap N_0'').
\]
For $i\in\{0,1\}$, let $G_i''=\embedd[N_0]^{-1}(\boundary_i N_0'')$, so that $G_0''\cap G_1''=\emptyset$, and $G_0''\cup G_1''=\hboundary N\cap N''$.
Moreover, we remark that $\abstr{N''}\homeo \abstr{G_0''}\times[0,1]$.

\paragraph{The map $f_g$ and the sub-2-complex $D_1$.}
Let $\map{f_g}{N'}{X}$ be the restriction of $\embedd[N_0]$ to $N'$.
This is a simplicial map, that can be realised by gluing together the edges of $N'$ that are identified in $M'$ under $\embedd[N_0]$; more precisely, all the identified edges lie in the intersection of the guts $X$ with $\posnbhd{F_v}$.
It is then easy to establish the following facts about $f_g$:
\begin{itemize}
\item $f_g$ restricts to an orientation\=/preserving homeomorphism $\umap{\interior{N'}}{\interior{X}}$;
\item $f_g(\boundary_i N\cap N')=\boundary_iX$ for $i\in\{0,1\}$;
\item $f_g(N''\cap N')=(\posnbhd{F_0}\cup\negnbhd{F_1})\cap X$;
\item $f_g$ restricts to a simplicial isomorphism
\[
\umap{N'\cap\hboundary N''}{\pospushoff{(\boundary_F F_0)}\cup\negpushoff{(\boundary_F F_1)}},
\]
and this restriction extends to an orientation\=/preserving homeomorphism $\umap{\hboundary N''}{\pospushoff{F_0}\cup\negpushoff{F_1}}$.
\end{itemize}
The extension in the last bullet point is simply the restriction of $\embedd[N_0]$ to $\hboundary N''$.

Let $D_1=\embedd[N_0]^{-1}(\negpushoff{(D_1')})$.
Recall that each component of $\thick{D_1'}$ is a disc intersecting $\boundary F$ in a single (possibly empty) arc, and that $\negpushoff{(D_1')}$ is disjoint from $\posnbhd{F_v}$.
It follows that each component of $\thick{D_1}$ is a disc intersecting $\boundary\boundary_1 N$ in a single (possibly empty) arc, and moreover that $f_g(D_1\cap N')$ is disjoint from $\posnbhd{F_v}$.
Finally, we remark that
\[
\boundary_1 N\cap N''=\embedd[N_0]^{-1}(\negpushoff{F_1})\subs\embedd[N_0]^{-1}(\negpushoff{(D_1')})=D_1.
\]

\paragraph{The triangulation $\RRR$.}
Let $\RRR'$ be the pull\=/back of the triangulation of the guts $X$ under $\embedd[N_0]$; this defines a triangulation of the sub\=/$3$\=/complex $N'$ of $N$.
Note that every component of $\vboundary N''$ intersects $N'$, since if -- say -- $F_0$ contained a boundary component of $F$, then it would be impossible for $D_0$ to contain $F_0$.
Since $\abstr{N''}\homeo \abstr{G_0''}\times[0,1]$ is a product interval bundle over a planar surface, we can apply \zcref{thm:triangulating interval bundles efficiently} to obtain a suitable pre\=/sutured triangulation $\RRR''$ of $N''$ that agrees with $\RRR'$ on $N'\cap N''$, and such that the restriction of $\RRR''$ to $\hboundary N''$ is flapless.
The two triangulations $\RRR'$ and $\RRR''$ glue together to form a triangulation $\RRR$ of $N$; since $\hboundary N$ is simplicial in this triangulation, we can endow $\RRR$ with the structure of a suitable pre\=/sutured triangulation of $N$.
Since the restrictions of $\RRR'$ to $\hboundary N\cap N'$ and of $\RRR''$ to $\hboundary N''$ are both flapless, we deduce that the same is true of the restriction of $\RRR$ to $\hboundary N$.

\paragraph{Quantitative bounds.}
We now demonstrate that the objects we have constructed are ``small'' with respect to the size $t$ of the triangulation $\TTT$.
Start by recalling that $F_g$ is simplicially isomorphic to $\boundary_0 X$; \zcref{thm:triangulation of the guts is small} then implies that $\area{F_g}\le 32t$.
Moreover, since the triangulation of $F_g$ is flapless, we have that $\length{\boundary F_g}\le\area{F_g}\le 32t$.
Since $\boundary_F F_0$ is contained in $\boundary F_g$, it immediately follows that $\length{\boundary_F F_0}\le 32t$; the same argument used for $F_0$ also shows that $\length{\boundary_F F_1}\le 32t$.
By construction, we have that $\boundary_F D_0\subs\boundary_F F_0$, and hence $\length{\boundary_F D_0}\le 32t$ as well.
Finally, we recall that $D_1=\embedd[N_0]^{-1}(\negpushoff{(D_1')})$, and that $\boundary_F D_1'\subs\boundary_F F_1$, so we deduce that $\length{\boundary D_1}=\length{\boundary_F D_1'}\le 32t$.

We now turn to the size of the triangulation $\RRR$ of $N$.
The size of $\RRR'$ is addressed by \zcref{thm:triangulation of the guts is small}, giving that the number of tetrahedra of $\RRR'$ is at most $50t$.
A bound on the size of $\RRR''$ comes from \zcref{thm:triangulating interval bundles efficiently}, provided that we can estimate the number of triangles in $N'\cap N''$.
Note that $N'\cap N''$ is a union of triangles in $\embedd[N_0]^{-1}(\vboundary X)$, so \zcref{thm:triangulation of the guts is small} bounds its area by $36t$.
\zcref[S]{thm:triangulating interval bundles efficiently} then gives a bound of $936t$ for the size of $\RRR''$, and hence a bound of $986t$ for the size of $\RRR$.

\subsection{Definition of the certificate}

We now give a more ``axiomatic'' definition of the fibredness certificate constructed in \zcref{sec:certificate:construction}.
The purpose of this is to achieve a clear separation of the statements we need to prove:
\begin{itemize}
\item if $F$ is a fibre of $M$, then it admits a valid certificate (\zcref{thm:existence of fibredness certificate});
\item if $F$ admits a valid certificate, then $F$ is a fibre of $M$ (\zcref{thm:correctness of fibredness certificate});
\item the validity of a certificate can be checked in polynomial time (\zcref{thm:verification of fibredness certificate}).
\end{itemize}
Given the number and complexity of the objects and properties involved, we believe this -- admittedly cumbersome -- axiomatic definition is necessary.

\begin{definition}[Certificate for fibredness]
\label{def:certificate}
Let $\TTT$ be a triangulation of a compact connected oriented $3$\=/manifold $M$ with $t$ tetrahedra, and let $F$ be a transversely oriented connected normal surface in $M$.
Let $M'=M\cut F$, and denote by $X$ the guts of $M'$.
We say that a certificate $\Sigma$ lies in $\cert[fib](M,F)$ if it consists of:
\begin{itemize}
\item sub\=/$2$\=/complexes $F_0$, $F_1$, $D_0$ of $F$;
\item a pre\=/sutured triangulation $\RRR$ of a compact oriented $3$\=/manifold $N$;
\item a sub\=/$2$\=/complex $D_1$ of $\boundary_1 N$;
\item sub\=/$3$\=/complexes $N'$ and $N''$ of $N$;
\item sub\=/$2$\=/complexes $G_0''$ and $G_1''$ of $\hboundary N$;
\item a simplicial map $\map{f_g}{N'}{X}$,
\item two additional certificates $\Sigma'$ and $\Sigma''$ of the type recognised by the algorithm of \zcref{thm:interval bundle recognition is in NP}.
\end{itemize}

We define the \emph{size} of $\Sigma$ to be the number
\begin{align*}
\card{\Sigma}&=t+\log(w(F)+1)+\card{\RRR}+\card{\Sigma'}+\card{\Sigma''}\\
&{}+\length{\boundary_F F_0}+\length{\boundary_F F_1}+\length{\boundary_F D_0},
\end{align*}
where $\card{\Sigma'}$ and $\card{\Sigma''}$ are the sizes of the certificates $\Sigma'$ and $\Sigma''$ respectively.
We remark that the certificate $\Sigma$ can be described with a number of binary digits that is polynomial in $\card{\Sigma}$.

Let $F_g$ and $F_g'$ be the sub\=/$2$\=/complexes of $F$ such that $\pospushoff{F_g}=\boundary_0 X$ and $\negpushoff{(F_g')}=\boundary_1 X$.
Denote by $F_v$ the closure of $F\setminus(F_g\cup F_0)$ in $F$, and by $F_v'$ the closure of $F\setminus(F_g'\cup F_1)$ in $F$.
We say that $\Sigma$ is \emph{valid}, and write $\Sigma\in\cert*[fib](M,F)$, if the following conditions are satisfied:
\newcommand{\certitem}[1]{\item\label{itm:fibredness certificate:#1}}
\begin{enumgrouped}
\begin{itemgroup}[itm:fibredness certificate:Fs][reftype=propertygroup]
\zcsetup{reftype=property}
\certitem{Fg F0 disjoint} $F_g$ and $F_0$ have disjoint interiors;
\certitem{F0 Fv disjoint} $F_0$ and $F_v$ are disjoint;
\certitem{Fg' F1 disjoint} $F_g'$ and $F_1$ have disjoint interiors;
\certitem{F1 Fv' disjoint} $F_1$ and $F_v'$ are disjoint;
\certitem{N+(Fv) vertical} $\posnbhd{F_v}=\negnbhd{F_v'}$;
\end{itemgroup}
\begin{itemgroup}[itm:fibredness certificate:N][reftype=propertygroup]
\zcsetup{reftype=property}
\certitem{N is I-bundle} $N\homeo G\times[0,1]$ for some compact orientable surface $G$;
\certitem{R suitable triangulation} $\RRR$ is a suitable pre\=/sutured triangulation of $N$;
\certitem{R flapless triangulation} the restriction of $\RRR$ to $\hboundary N$ is flapless;
\certitem{decomposition of N} $N=N'\cup N''$, and $N'$ and $N''$ have disjoint interiors;
\certitem{Gi'' disjoint} $G_0''$ and $G_1''$ are disjoint;
\certitem{Gi'' union} $G_0''\cup G_1''=\hboundary N\cap N''$;
\certitem{N'' is I-bundle} $\abstr{N''}\homeo G''\times[0,1]$ for some compact orientable surface $G''$;
\certitem{R'' suitable triangulation}$\RRR''$ is a suitable pre\=/sutured triangulation of $N''$, where $\RRR''$ is the pre\=/sutured triangulation of $N''$ that, as a triangulation, is simply the restriction of $\RRR$ to $N''$, and such that $\boundary_i N''=G_i''$ for $i\in\{0,1\}$;
\end{itemgroup}
\begin{itemgroup}[itm:fibredness certificate:f][reftype=propertygroup]
\zcsetup{reftype=property}
\certitem{fg homeomorphism} $f_g$ restricts to an orientation\=/preserving homeomorphism $\umap{\interior{N'}}{\interior{X}}$;
\certitem{fg of horizontal boundary} $f_g(\boundary_i N\cap N')=\boundary_i X$ for $i\in\{0,1\}$;
\certitem{fg of vertical boundary} $f_g(N'\cap N'')=(\posnbhd{F_0}\cup\negnbhd{F_1})\cap X$;
\certitem{fg extends to horizontal boundary} $f_g$ restricts to a simplicial isomorphism
\[
\umap{N'\cap\hboundary N''}{\pospushoff{(\boundary_F F_0)}\cup\negpushoff{(\boundary_F F_1)}},
\]
and this restriction extends to an orientation\=/preserving homeomorphism $\umap{\hboundary N''}{\pospushoff{F_0}\cup\negpushoff{F_1}}$;
\end{itemgroup}
\begin{itemgroup}[itm:fibredness certificate:D][reftype=propertygroup]
\zcsetup{reftype=property}
\certitem{D0 union of discs} each component of $\thick{D_0}$ is a disc intersecting $\boundary F$ in a single (possibly empty) arc;
\certitem{F0 in D0} $F_0\subs D_0$;
\certitem{D0 disjoint from Fv} $D_0$ is disjoint from $F_v$;
\certitem{D1 union of discs} each component of $\thick{D_1}$ is a disc intersecting $\boundary\boundary_1 N$ in a single (possibly empty) arc;
\certitem{F1 in D1} $\boundary_1 N\cap N''\subs D_1$;
\certitem{D1 disjoint from boundary} $f_g(D_1\cap N')$ is disjoint from $\posnbhd{F_v}$;
\end{itemgroup}
\begin{itemgroup}[itm:fibredness certificate:quantitative][reftype=propertygroup]
\zcsetup{reftype=property}
\certitem{bounds on surfaces} $\max\{\length{\boundary_F F_0},\length{\boundary_F F_1},\length{\boundary_F D_0},\length{\boundary D_1}\}\le 32t$;
\certitem{size of R} $\RRR$ has at most $986t$ tetrahedra;
\end{itemgroup}
\begin{itemgroup}[itm:fibredness certificate:aux][reftype=propertygroup]
\zcsetup{reftype=property}
\certitem{Sigma' valid} $\Sigma'$ certifies that $\RRR$ is a suitable pre\=/sutured triangulation of $N$ (according to the algorithm of \zcref{thm:interval bundle recognition is in NP}), and its size is polynomial in $\card{\RRR}$;
\certitem{Sigma'' valid} $\Sigma''$ certifies that the pull\=/back of $\RRR''$ under $\embedd[N'']$ is a suitable pre\=/sutured triangulation of $\abstr{N''}$ (according to the algorithm of \zcref{thm:interval bundle recognition is in NP}), and its size is polynomial in $\card{\RRR''}$.\qedhere
\end{itemgroup}
\end{enumgrouped}
\end{definition}

In our exposition in \zcref{sec:certificate:construction}, we have taken particular care to highlight how the certificate we constructed satisfies all the properties of \zcref{def:certificate}; the only additional remark we should make is that the certificates $\Sigma'$ and $\Sigma''$ are provided by \zcref{thm:interval bundle recognition is in NP}.
Therefore, we can state the following.

\begin{proposition}[Existence of the fibredness certificate]
\label{thm:existence of fibredness certificate}
Let $\TTT$ be a triangulation of a compact connected oriented fibred $3$\=/manifold $M$, and let $F$ be a transversely oriented least\=/weight normal fibre of $M$.
Then $\cert*[fib](M,F)$ is non\=/empty.
\end{proposition}

\subsection{Correctness of the certificate}

\begin{lemma}[Basic properties of the fibredness certificate]
\label{thm:basic properties of fibredness certificate}
With the notation of \zcref{def:certificate}, if $\Sigma\in\cert*[fib](M,F)$, then the following hold:
\begin{enumarabic}
\zcsetup{reftype=property}
\item $F_g$, $F_g'$, $F_v$, $F_v'$, $F_0$, and $F_1$ are unions of normal discs;
\item\label{itm:basic properties of fibredness certificate:parallelity components} $\posnbhd{F_v}$, $\posnbhd{F_0}$, and $\negnbhd{F_1}$ are unions of parallelity components of $M'$;
\item the parallelity bundle of $M'$ is the union of $\posnbhd{F_v}$ and $\posnbhd{F_0}\cup\negnbhd{F_1}$;
\item $(\posnbhd{F_0}\cup\posnbhd{F_1})\cap\hboundary M'=\pospushoff{F_0}\cup\negpushoff{F_1}$;
\item $\posnbhd{F_v}$ and $\posnbhd{F_0}\cup\negnbhd{F_1}$ are disjoint;
\item $f_g$ restricts to a homeomorphism $\umap{f_g^{-1}(X')}{X'}$, where $X'$ is the complement in $X$ of the singular locus of $X$;
\item\label{itm:basic properties of fibredness certificate:Fi planar} $\thick{F_0}$ and $\thick{F_1}$ are planar surfaces;
\item\label{itm:basic properties of fibredness certificate:area of Fg} $\area{F_g}\le 32t$ and $\area{F_g'}\le 32t$;
\end{enumarabic}
\end{lemma}
\begin{proof}
We prove each statement separately.
\begin{enumarabic}
\item By definition, the surface $F_g$ is a union of normal discs.
\zcref[S]{itm:fibredness certificate:F0 Fv disjoint,itm:fibredness certificate:Fg F0 disjoint} imply that $F_0$ and $F_v$ are unions of components of $\closure{F\setminus F_g}$.
In particular, they must both be unions of normal discs.
The same argument applies to $F_g'$, $F_v'$, and $F_1$.
\item This follows immediately from the fact that $F_v$ and $F_0$ are unions of components of $\closure{F\setminus F_g}$, and similarly $F_1$ is a union of components of $\closure{F\setminus F_g'}$.
\item This is a consequence of the fact that $F=F_g\cup F_0\cup F_v$ and $F=F_g'\cup F_1\cup F_v'$.
\item We only prove the statement for $\boundary_0 M'$.
Let $Z$ be a component of $\posnbhd{F_0}\cup\negnbhd{F_1}$.
Each component $C$ of $Z\cap\boundary_0 M'$ is a component of $\pospushoff{F_0}$ or of $\pospushoff{F_v}$.
Suppose, for a contradiction, that $C$ is a component of $\pospushoff{F_v}$.
\zcref[S]{itm:fibredness certificate:N+(Fv) vertical} implies that $C=Z\cap\boundary_0 M'$.
But $C$ must be disjoint from $\pospushoff{F_0}$ by \zcref{itm:fibredness certificate:F0 Fv disjoint}, and $Z\cap\boundary_1 M'\subs\negpushoff{(F_v')}$ must be disjoint from $\negpushoff{F_1}$ by \zcref{itm:fibredness certificate:Fg' F1 disjoint}.
However, this contradicts the fact that $Z$ is a component of $\posnbhd{F_0}\cup\negnbhd{F_1}$.
\item Two parallelity components intersect if and only if their horizontal boundaries intersect.
However, \zcref{itm:fibredness certificate:N+(Fv) vertical} implies that $\posnbhd{F_v}\cap\hboundary M'=\pospushoff{F_v}\cup\negpushoff{(F_v')}$.
We readily see that the horizontal boundaries of $\posnbhd{F_v}$ and $\posnbhd{F_0}\cup\negnbhd{F_1}$ are disjoint.
\item If $x$ is a point in $X\setminus X'$, then it admits a neighbourhood $U$ in $X$ such that $U\setminus\boundary X$ is connected.
Since, by \zcref{itm:fibredness certificate:fg homeomorphism}, the restriction of $f_g$ to $\interior{N'}$ is a homeomorphism onto $\interior{X}$, we see that $f_g^{-1}(U\setminus\boundary X)$ is connected.
Therefore, there can be at most one point in $N'$ mapping to $x$ under $f_g$.
\item \zcref[S]{itm:fibredness certificate:F0 in D0,itm:fibredness certificate:D0 union of discs} immediately imply that $\thick{F_0}$ is a planar surface.
To see that $\thick{F_1}$ is planar, note first that \zcref{itm:fibredness certificate:fg homeomorphism,itm:fibredness certificate:fg extends to horizontal boundary} combine to give a homeomorphism $\umap{\boundary_1 N}{\negpushoff{(F_g'\cup F_1)}}$ sending $\boundary_1 N\cap N''$ to $\negpushoff{F_1}$.
From \zcref{itm:fibredness certificate:F1 in D1,itm:fibredness certificate:D1 union of discs}, we conclude that $\thick{F_1}$ is contained in a disjoint union of discs, and hence it is planar.
\item The sub\=/$2$\=/complex $F_g$ of $F$ is simplicially isomorphic to $\boundary_0 X$, hence by \zcref{thm:triangulation of the guts is small} we have that $\area{F_g}\le 32t$.
The argument for $F_g'$ is completely analogous.\qedhere
\end{enumarabic}
\end{proof}

\begin{lemma}[Extending homeomorphisms of the boundary of interval bundles]
\label{thm:extending homeomorphisms of the boundary of interval bundles}
Let $M$ and $N$ be pre\=/sutured manifolds that are both homeomorphic to $F\times[0,1]$ for some compact orientable surface $F$, and let $A$ be a closed vertical subset of $\vboundary M$.
Let $\map{f_0}{\hboundary M\cup A}{\boundary N}$ be an embedding such that $f_0(\boundary_i M)=\boundary_i N$ for $i\in\{0,1\}$.
Then there is a homeomorphism $\map{f}{M}{N}$ of pre\=/sutured manifolds that coincides with $f_0$ on $A$.
\end{lemma}
\begin{proof}
Firstly, note that $B=f_0(A)$ is a vertical subset of $\vboundary N$.
Let $A'$ be the union of all the components of $\vboundary M$ that intersect $A$.
Consider the restriction of $f_0$ to $A\cup(\hboundary M\cap A')$; this is a homeomorphism to $B\cup(\hboundary N\cap B')$ for some union $B'$ of components of $\vboundary N$, and it can be extended to a homeomorphism $\map{f'}{A'}{B'}$ (see \zcref{fig:extending homeomorphisms of the boundary of interval bundles:extending to A'}).
Let $A''=\boundary_0 M\cup A'$ and $B''=\boundary_0 N\cup B'$.
Since $\boundary_0 M$ and $\boundary_0 N$ are homeomorphic, we can extend $f'$ to a homeomorphism $\map{f''}{A''}{B''}$.
Finally, note that $(M,A'')$ and $(N,B'')$, as pairs of topological spaces, are both homeomorphic to $(F\times[0,1],F\times\{0\})$; this is shown in \zcref{fig:extending homeomorphisms of the boundary of interval bundles:flattening}.
Every homeomorphism $\umap{F\times\{0\}}{F\times\{0\}}$ extends to a homeomorphism $\umap{F\times[0,1]}{F\times[0,1]}$; hence, the homeomorphism $\map{f''}{A''}{B''}$ extends to a homeomorphism $\map{f}{M}{N}$.
This can be chosen so that $f(\boundary_i M)=\boundary_i N$ for $i\in\{0,1\}$, and by construction we have that $f$ agrees with $f_0$ on $A$, as required.
\end{proof}

\begin{myfigure}
\begin{subfigure}{0.48\textwidth}
\centering
\myfiguresource[scale=.25]{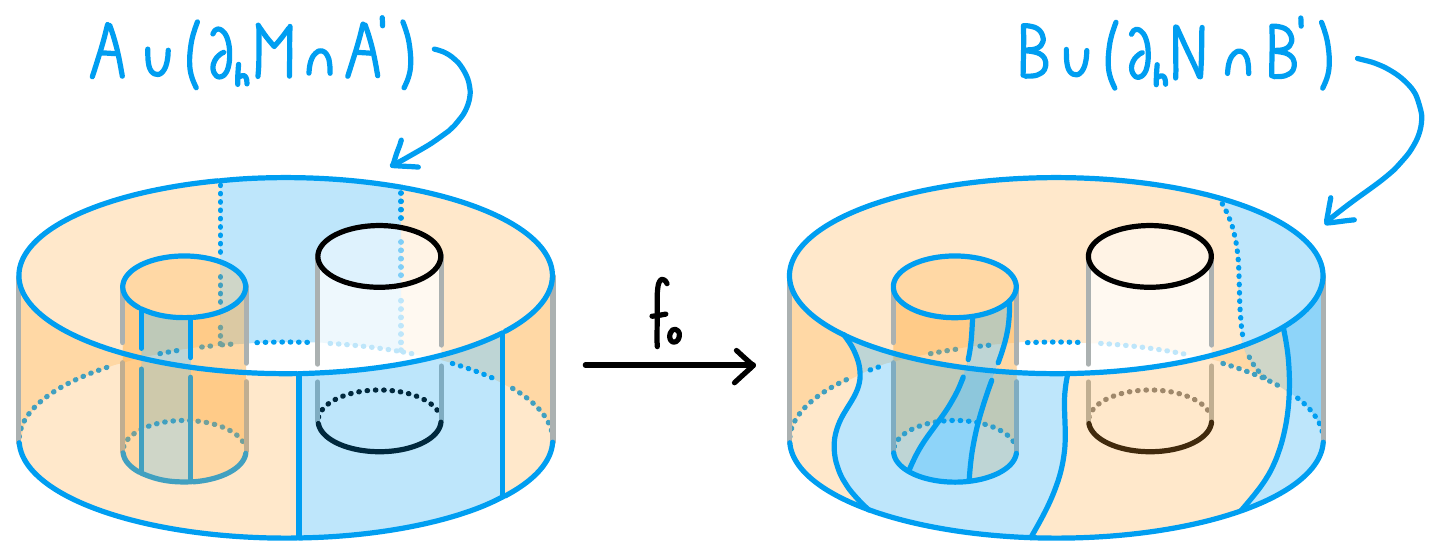}
\caption{}
\label{fig:extending homeomorphisms of the boundary of interval bundles:extending to A'}
\end{subfigure}\hfill
\begin{subfigure}{0.48\textwidth}
\centering
\myfiguresource[scale=.25]{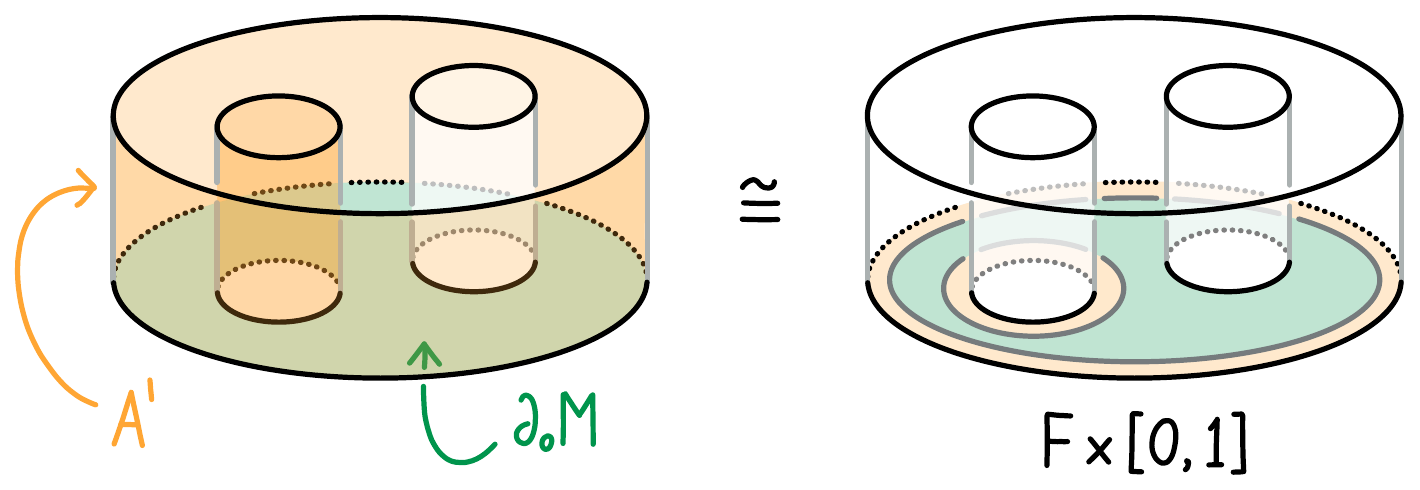}
\caption{}
\label{fig:extending homeomorphisms of the boundary of interval bundles:flattening}
\end{subfigure}
\caption{\subref{fig:extending homeomorphisms of the boundary of interval bundles:extending to A'} The homeomorphism $f_0$ restricts to a homeomorphism from $A\cup(\hboundary M\cap A')$ to $B\cup(\hboundary N\cap B')$, and this restriction can be extended to a homeomorphism $\umap{A'}{B'}$. The union of the shaded regions is, respectively, $A'$ on the left and $B'$ on the right.
\subref{fig:extending homeomorphisms of the boundary of interval bundles:flattening} The pair $(M,A'')$ is homeomorphic to $(F\times[0,1],F\times\{0\})$; this can be seen by ``flattening'' $A'$ on the horizontal boundary of $F\times[0,1]$.}
\end{myfigure}

\begin{proposition}[Correctness of the fibredness certificate]
\label{thm:correctness of fibredness certificate}
Suppose that $M$ is a triangulated compact connected oriented $3$\=/manifold, and $F$ is a transversely oriented connected normal surface in $M$.
Let $\Sigma\in\cert*[fib](M,F)$ be a valid certificate.
Then the following hold.
\begin{enumarabic}
\zcsetup{reftype=property}
\item Let $M_N=X\cup\posnbhd{F_0}\cup\negnbhd{F_1}$, and endow $M_N$ with the structure of a pre\=/sutured subcomplex given by
\begin{gather*}
\boundary_i M_N=\boundary_i M'\cap M_N\qquad\text{for $i\in\{0,1\}$,}\\
\vboundary M_N=(\vboundary M'\cup\posnbhd{F_v})\cap M_N.
\end{gather*}
Then there exists a map $\map{f}{N}{M_N}$ such that $f(\boundary_i N)=\boundary_i M_N$ for $i\in\{0,1\}$; moreover, the map $f$ restricts to $f_g$ on $N'$, and to a homeomorphism $\umap{N''}{\posnbhd{F_0}\cup\negnbhd{F_1}}$.
\item The pre\=/sutured manifold $\abstr{M_N}$ is a product interval bundle over a compact orientable surface; moreover, it is homeomorphic to $N$ as a pre\=/sutured manifold.
\item\label{itm:correctness of fibredness certificate:M' interval bundle} The pre\=/sutured manifold $M'$ is a product interval bundle over $F$, and the interval bundle structure agrees with that of $M_N$, as well as with the intrinsic one on $\posnbhd{F_v}$.
\item\label{itm:correctness of fibredness certificate:vertical} The properly embedded surface $\thin{X}\cap\thick{\posnbhd{F_v}}$ is vertical in $M'$.
\end{enumarabic}
\end{proposition}
\begin{proof}
Let $M''=\posnbhd{F_0}\cup\negnbhd{F_1}$, so that $M_N=X\cup M''$.
Endow $M''$ with its intrinsic interval bundle structure.
Let $A=M''\cap X$ and $B=N''\cap N'$; by \zcref{itm:fibredness certificate:fg of vertical boundary}, the image of $B$ under $f_g$ is precisely $A$.
By \zcref{itm:fibredness certificate:fg of horizontal boundary}, there is a homeomorphism $\map{f_0}{\hboundary N''}{\hboundary M''}$ whose restriction to $B\cap\hboundary N$ is a simplicial isomorphism $\umap{B\cap\hboundary N}{A\cap\hboundary M'}$ that agrees with $f_g$.
This implies that the restriction of $f_g$ to $B$ is actually a homeomorphism $\umap{B}{A}$.
Denote by $f_1$ the homeomorphism $\umap{\hboundary N''\cup B}{\hboundary M''\cup A}$ that agrees with $f_0$ on $\hboundary N''$ and with $f_g$ on $B$.
Note that each component of $B$ intersects two distinct components of $\hboundary N''$, as $\abstr{N''}$ is a product interval bundle by \zcref{itm:fibredness certificate:N'' is I-bundle}.
Therefore, we see that each component of $A$ intersects two distinct components of $\hboundary M''$.
It follows that $\abstr{M''}$ is also a product interval bundle, and that $A$ is a vertical subset of $\vboundary M''$.
We can then apply \zcref{thm:extending homeomorphisms of the boundary of interval bundles} to extend $f_1$ to a homeomorphism $\map{f_2}{N''}{M''}$ of pre\=/sutured subcomplexes that agrees with $f_1$ on $B$.
For a completely rigorous justification of this step, one should in fact apply \zcref{thm:extending homeomorphisms of the boundary of interval bundles} to $\abstr{N''}$ and $\abstr{M''}$ instead of $N''$ and $M''$, and exploit the fact that the singular locus of $N''$ is contained in $B$ to obtain the desired homeomorphism $f_2$.

Let $\map{f}{N}{M_N}$ be the map that coincides with $f_g$ on $N'$ and with $f_2$ on $N''$.
By construction of $f_2$ and \zcref{itm:fibredness certificate:fg of horizontal boundary}, we have that $f(\boundary_i N)=\boundary_i M_N$ for $i\in\{0,1\}$.
Moreover, since $f_2$ is a homeomorphism and $f_g$ restricts to a homeomorphism on the interior of $M'$, we see that $f$ restricts to a homeomorphism $\umap{\interior{N}}{\interior{M_N}}$.
It follows that $f$ lifts to a homeomorphism of pre\=/sutured manifolds $\umap{N}{\abstr{M_N}}$, thus proving that $\abstr{M_N}$ is a product interval bundle over a compact orientable surface, namely $\abstr{F_g\cup F_0}$.

Consider now the intrinsic interval bundle structure on $\posnbhd{F_v}$.
The intersection $\posnbhd{F_v}\cap M_N$ is a vertical subset of $\vboundary\posnbhd{F_v}$, since it is the intersection of the guts of $M'$ with the union of product parallelity components $\posnbhd{F_v}$.
It is not hard to see that $\posnbhd{F_v}\cap M_N$ is also a vertical subset $\vboundary M_N$; in fact, by \zcref{itm:fibredness certificate:N+(Fv) vertical}, this intersection is a union of annuli and squares in $\vboundary M_N$, each of which intersects both $\boundary_0 M_N$ and $\boundary_1 M_N$.
We deduce that the interval bundle structures on $\posnbhd{F_v}$ and $M_N$ glue together to give a product interval bundle structure on the pre\=/sutured manifold $M'$.
Finally, since $\thin{X}\cap\thick{\posnbhd{F_v}}$ is a union of boundary components of a regular neighbourhood of $\posnbhd{F_v}\cap M_N$, it is a vertical surface in $M'$.
\end{proof}

We remark that proving \zcref{itm:correctness of fibredness certificate:vertical} above is not strictly necessary for the purposes of \zcref{thm:fibredness detection is in NP}, as \zcref{itm:correctness of fibredness certificate:M' interval bundle} would suffice.
We include it because it completes the picture of the certificate given in \zcref{sec:overview of the certificate}, and because we will need it in our upcoming article \cite{baroni:certifying-hyperbolicity-fibred}.

\subsection{Verification of the certificate}

\begin{proposition}[Verification of the fibredness certificate]
\label{thm:verification of fibredness certificate}
There is an algorithm that takes as input
\begin{itemize}
\item a triangulation of a compact connected oriented $3$\=/manifold $M$,
\item a transversely oriented connected normal surface $F$ in $M$, and
\item a certificate $\Sigma\in\cert[fib](M,F)$,
\end{itemize}
and decides whether $\Sigma\in\cert*[fib](M,F)$.
The running time of the algorithm is polynomial in $\card{\Sigma}$.
\end{proposition}
\begin{proof}
Many of the required properties can be easily verified, either by direct inspection, or by a straightforward application of the algorithm of \citeauthor{agol-hass-thurston:computational-complexity-knot} or one of its corollaries listed in \zcref{sec:applications of agol-hass-thurston}.
For example, one can check that $\RRR$ is a pre\=/sutured triangulation of $N$ by direct verification, and that $F_v$ is a sub\=/$2$\=/complex of $F$ using the algorithm of \zcref{thm:finding the components of a sub-2-complex of a normal surface}.
We will not discuss other instances of this kind of verification, and instead focus on the steps that require a non\=/trivial algorithmic approach.

\paragraph{\zcref[S]{itm:fibredness certificate:Fs}.}
The sub\=/$2$\=/complexes $F_g$ and $F_g'$ of $F$ can be constructed explicitly, using \zcref{thm:transverse orientation of a normal surface} to compute transverse orientations.
\zcref[S]{thm:deciding containment of sub-2-complexes of normal surfaces} immediately tells us whether \zcref{itm:fibredness certificate:Fg F0 disjoint} holds.
We can apply \zcref{thm:finding the components of a sub-2-complex of a normal surface} to find the components of $\closure{F\setminus F_g}$ and of $F_0$, and check that each component of the latter is also a component of the former.
This allows us to compute $F_v$, and verify that \zcref{itm:fibredness certificate:F0 Fv disjoint} holds.
We perform an analogous test for $F_1$ and $F_g'$, making sure that \zcref{itm:fibredness certificate:F1 Fv' disjoint,itm:fibredness certificate:Fg' F1 disjoint} are satisfied.
Finally, to verify \zcref{itm:fibredness certificate:N+(Fv) vertical}, we pick a triangle $T$ in each component of $F_v$, and we check that the transverse orientations of $F$ at $T$ and $\transfermap(T)$ agree using \zcref{thm:transverse orientation of a normal surface} (recall the definition of the transfer map $\transfermap$ from \zcref{sec:triangulating the guts}); we also verify that $\transfermap(T)\subs F_v'$ using \zcref{thm:deciding containment of sub-2-complexes of normal surfaces}.
This guarantees that $\posnbhd{F_v}\subs\negnbhd{F_v'}$; the verification of the other containment is perfectly symmetric.

\paragraph{\zcref[S]{itm:fibredness certificate:N}.}
\zcref[S]{itm:fibredness certificate:N is I-bundle,itm:fibredness certificate:R suitable triangulation} can be verified with the aid of the certificate $\Sigma'$.
Similarly, \zcref{itm:fibredness certificate:N'' is I-bundle,itm:fibredness certificate:R'' suitable triangulation} can be verified with the aid of the certificate $\Sigma''$.
The other properties can be checked by direct inspection in polynomial time in the size of the triangulation $\RRR$.

\paragraph{\zcref[S]{itm:fibredness certificate:f}.}
The only property whose verification is non\=/trivial is \zcref{itm:fibredness certificate:fg extends to horizontal boundary}.
We can directly check that $f_g$ restricts to a simplicial isomorphism
\[
\umap{N'\cap\hboundary N''}{\pospushoff{(\boundary_F F_0)}\cup\negpushoff{(\boundary_F F_1)}}.
\]
We are then in the following situation: we have a homeomorphism from a subset of the boundary of $\hboundary N''$ to a subset of the boundary of $\pospushoff{F_0}\cup\negpushoff{F_1}$, and we ask whether it can be extended to a homeomorphism $\umap{\hboundary N''}{\pospushoff{F_0}\cup\negpushoff{F_1}}$.
We can answer this question by analysing each component $C$ of $\hboundary N''$ separately.
If $C$ is disjoint from $N'$, then there is no such extension, because every component of $F_0$ and $F_1$ intersects $X$.
Otherwise, the subset $f_g(C\cap N')$ of $\pospushoff{(\boundary_FF_0)}\cup\negpushoff{(\boundary_F F_1)}$ contains at least one edge.
If more than one component of $\pospushoff{F_0}\cup\negpushoff{F_1}$ intersects $f_g(C\cap N')$ in at least one edge, then there is no extension.
Otherwise, suppose without loss of generality that there is a unique component $C'$ of $F_0$ such that $f_g(C\cap N')\subs\pospushoff{(C')}$.
If $f_g(C\cap N')\neq \pospushoff{(\boundary_F C')}$, then there is no extension.
Otherwise, let $\map{h}{C\cap N'}{\boundary_F C'}$ be the simplicial isomorphism defined by $\pospushoff{h(x)}=f_g(x)$ for each $x\in C\cap N'$.
We only need to check that $h$ extends to an orientation\=/preserving homeomorphism $\umap{C}{C'}$.
A first necessary condition is that $\abstr{C}$ and $\abstr{C'}$ are homeomorphic; this can be verified thanks to \zcref{thm:finding the components of a sub-2-complex of a normal surface}.
Secondly, we need the $\hboundary N''$\=/boundary sequences of $C$ and the $F$\=/boundary sequences of $C'$ to be compatible, in the following sense.
There must exist a bijection between components of $\boundary\abstr{C}$ and components of $\boundary\abstr{C'}$ such that, if a component $b$ of $\boundary\abstr{C}$ corresponds to a component $b'$ of $\boundary\abstr{C'}$ under this bijection, and a $\hboundary N''$\=/boundary sequence of $C$ is $e_1,\ldots,e_k$, then $e_1',\ldots,e_k'$ is an $F$\=/boundary sequence of $C'$, where
\[
\begin{dcases*}
\pospushoff{(e_i')}=f_g(e_i)&if $e_i$ is an edge of $C\cap N'$,\\
e_i'=\boundary F&if $e_i=\boundary\hboundary N''$
\end{dcases*}
\qquad\text{for each $i\in\{1,\ldots,k\}$.}
\]
Boundary sequences can be computed thanks to \zcref{thm:finding the components of a sub-2-complex of a normal surface}; the existence of a bijection with the desired property is easy to verify, since each edge of $C'$ appears in at most one $F$\=/boundary sequence.
If these conditions hold for each component $C$ of $\hboundary N''$, then \zcref{itm:fibredness certificate:fg extends to horizontal boundary} is satisfied, otherwise it is not.

\paragraph{\zcref[S]{itm:fibredness certificate:D}.}
\zcref[S]{itm:fibredness certificate:F0 in D0,itm:fibredness certificate:D0 disjoint from Fv} can be verified using \zcref{thm:finding the components of a sub-2-complex of a normal surface}, and \zcref{itm:fibredness certificate:D1 disjoint from boundary,itm:fibredness certificate:F1 in D1} can be inspected directly.
As far as \zcref{itm:fibredness certificate:D0 union of discs} is concerned, we can use \zcref{thm:finding the components of a sub-2-complex of a normal surface} again to retrieve information about the components of $D_0$.
We can then manually combine components of $D_0$ into components of $\thick{D_0}$, since the number of singular points of $D_0$ is bounded above by $\length{\boundary_F D_0}$.
In particular, we can obtain the list of components of $\thick{D_0}$, together with their Euler characteristics and $F$\=/boundary sequences.
This is enough to check whether \zcref{itm:fibredness certificate:D0 union of discs} is satisfied.

\paragraph{\zcref[S]{itm:fibredness certificate:quantitative}.}
These conditions are just inequalities, which can be verified directly.

\paragraph{\zcref[S]{itm:fibredness certificate:aux}.}
These properties can be verified thanks to the algorithm of \zcref{thm:interval bundle recognition is in NP}.
\end{proof}

\section{Fibres of bounded weight}

\subsection{Fibres of negative Euler characteristic}

\bgroup

\RenewDocumentCommand{\P}{m}{\mathcal{P}_{#1}}
\NewDocumentCommand{\Pbar}{m}{\bar{\mathcal{P}}_{#1}}
\NewDocumentCommand{\vbar}{m}{\bar{\v{#1}}}
\NewDocumentCommand{\cone}{m}{\operatorname{cone}(#1)}

So far, we have developed the tools to certify that a given normal surface in a triangulated $3$\=/manifold is the fibre of a fibration over the circle.
Crucially, and unavoidably, the size of such a certificate depends on the ``size'' -- that is, the logarithm of the weight -- of the normal surface.
Therefore, the only remaining issue we need to address is whether every triangulated fibred $3$\=/manifold admits a normal fibre whose size is polynomial in the size of the triangulation.
This question is answered in the affirmative by \zcref{thm:fibre of small weight} below, for fibred $3$\=/manifolds admitting a fibre of negative Euler characteristic (the case of non\=/negative Euler characteristic is addressed in \zcref{sec:small toroidal fibres}).
We remark that this argument is not unknown to experts: in fact, it appears almost verbatim in \cite{schleimer:almost-normal-heegaard}, and it is anyway a straightforward application of the results in \cite{tollefson-wang:taut-normal-surfaces}.
However, we include it here for the sake of completeness, especially since the statement of \cite[Theorem~6.3.3]{schleimer:almost-normal-heegaard} requires the $3$\=/manifold to be atoroidal, and does not provide a bound on the Euler characteristic.

\begin{proposition}[A fibre of small weight and negative Euler characteristic]
\label{thm:fibre of small weight}
Let $\TTT$ be a triangulation of a compact connected orientable $3$\=/manifold $M$ with $t$ tetrahedra.
Suppose that $M$ fibres over the circle with fibre $F_0$ of negative Euler characteristic.
Then there is a normal fibre $F$ of $M$ with $\chi(F)\ge\chi(F_0)$ and
\[
w(F)\le t^3\cdot 2^{7t+8}.
\]
\end{proposition}

As remarked above, the proof of \zcref{thm:fibre of small weight} relies heavily on the tools developed in \cite{tollefson-wang:taut-normal-surfaces}, of which we give a very brief summary.
Let $\TTT$ be a triangulation of a compact connected oriented irreducible $3$\=/manifold $M$ with (possibly empty) toroidal boundary, and denote by $t$ the number of tetrahedra of $\TTT$.
The matching equations of $\TTT$ (see \zcref{fig:normal surfaces:matching equation}) define a system of linear equations in $\RR^{7t}$; the non\=/negative solutions of this system form a convex polyhedral cone $\P{\TTT}\subs\RR_{\ge 0}^{7t}$.
Denote by $\Pbar{\TTT}$ the set of points $\mathbf{x}\in\P{\TTT}$ such that $x_1+\ldots+x_{7t}=1$, which is a compact convex polyhedron.
In particular, we can talk about the \emph{faces} of $\Pbar{\TTT}$; more precisely, a face of $\Pbar{\TTT}$ is the intersection of $\Pbar{\TTT}$ with a hyperplane of the form $\{\mathbf{x}\in\RR^{7t}:\text{$x_i=0$ for $i\in\III$}\}$ for some subset $\III\subsetneq\{1,\ldots,7t\}$.
If $C$ is a face of $\Pbar{\TTT}$, we denote by $\cone{C}$ the set of positive scalar multiples of points in $C$.
More generally, we use the same notation for any subset $C$ of some Euclidean space $\RR^n$.

For every normal surface $F$ in $M$, its normal vector $\v{F}$ is a point in $\P{\TTT}\cap\ZZ^{7t}$.
Conversely, every point $\mathbf{w}\in\P{\TTT}\cap\ZZ^{7t}$ that also satisfies the consistency equations of $\TTT$ defines a normal surface $F_{\mathbf{w}}$ in $M$.
For a non\=/empty normal surface $F$, let $\vbar{F}$ be the projection of $\v{F}$ onto $\Pbar{\TTT}$ -- that is, the unique point in $\Pbar{\TTT}$ that is a scalar multiple of $\v{F}$; this will necessarily be a rational point.
We say that a face $C$ of $\Pbar{\TTT}$ \emph{carries} a normal surface $F$ if $\vbar{F}\in C$.
For every normal surface $F$, there is a unique minimal face $C$ of $\Pbar{\TTT}$ that carries $F$, which we denote by $C_F$.
Note that if a face $C$ of $\Pbar{\TTT}$ carries a normal surface, then every integral point in $\cone{C}$ will satisfy the consistency equations, and hence represent a normal surface in $M$.

For a compact orientable surface $F$, define
\[
\chi_-(F)=-\chi(F\setminus\{\text{sphere and disc components of $F$}\}).
\]
The \emph{Thurston seminorm} of a homology class $\alpha\in H_2(M,\boundary M;\RR)$ is defined as the minimum value of $\chi_-(F)$ over all oriented surfaces $F$ properly embedded in $M$ that represent $\alpha$.
This can be extended to a seminorm $\map{x}{H_2(M,\boundary M;\RR)}{\RR}$ by linearity, convexity, and continuity (see \cite{thurston:norm-homology-3} for details).
\Citeauthor{tollefson-wang:taut-normal-surfaces} call an oriented properly embedded incompressible boundary\=/incompressible surface $F\subs M$ \emph{taut} if its homology class $[F]\in H_2(M,\boundary M;\RR)$ is non\=/zero, no union of components of $F$ is homologically trivial, and $x([F])=\chi_-(F)$.
If, moreover, a taut surface is in general position with respect to $\TTT$ and has minimal weight amongst all taut general position surfaces in its homology class, then it is called \emph{lw\=/taut}; note that this notion of ``lw'' does not coincide with our notion of ``least\=/weight''.

The following is one of the main results of \cite{tollefson-wang:taut-normal-surfaces}.

\begin{theorem}[{\cite[Theorem~3.3 and Corollary 3.4]{tollefson-wang:taut-normal-surfaces}}]
\label{thm:lw-taut normal surfaces}
Let $F$ be an oriented lw\=/taut normal surface in $M$.
Then every normal surface carried by $C_F$ is lw\=/taut.
Moreover, there is a way to assign orientations to normal surfaces carried by $C_F$ such that, for every pair $G,H$ of normal surfaces carried by $C_F$, the following equalities hold:
\[
[G+H]=[G]+[H]\quad\text{and}\quad x([G+H])=x([G])+x([H]).
\]
\end{theorem}

We are now ready to prove \zcref{thm:fibre of small weight}.

\begin{proof}[Proof of \zcref{thm:fibre of small weight}]
Let $G$ be an oriented least\=/weight normal fibre of $M$ that is isotopic to $F_0$.
Note that, as discussed at the beginning of \cite[\S~3]{thurston:norm-homology-3}, every incompressible surface in $M$ that is homologous to $F_0$ is actually isotopic to it; in particular, this implies that $G$ is lw\=/taut.
From now on, all the surfaces we consider will be carried by $C_G$; we will implicitly endow them with the orientation given by \zcref{thm:lw-taut normal surfaces}.
The same theorem then implies that every normal surface carried by $C_G$ is lw\=/taut, and that the operation of taking the homology class is linear with respect to the normal sum on $C_G$; similarly, the Thurston seminorm $x$ is linear with respect to the normal sum on $C_G$.

Let $C'$ be the minimal face of the $x$\=/unit ball in $H_2(M,\boundary M;\RR)$ such that $[G]$ lies in $\cone{C'}$.
Crucially, by the linearity properties mentioned above, the homology classes of all the surfaces carried by $C_G$ will also lie in the closure of $\cone{C'}$.
Since $G$ is a fibre of $M$, Theorems 3 and 5 of \cite{thurston:norm-homology-3} imply that $C'$ is a top\=/dimensional face of the $x$\=/unit ball, and every integral homology class in the interior of $\cone{C'}$ is represented by a union of parallel fibres of $M$; in fact, every oriented incompressible surface with no sphere and disc components and whose homology class lies in the interior of $\cone{C'}$ is a union of parallel fibres of $M$.

Let now $F_1,\ldots,F_m$ be fundamental normal surfaces in $M$ such that $G=F_1+\ldots+F_m$; it is clear that these are all carried by $C_G$.
Up to reordering, we can assume that $\{[F_1],\ldots,[F_n]\}$ is a basis of the subspace of $H_2(M,\boundary M;\RR)$ spanned by $[F_1],\ldots,[F_m]$ for some $1\le n\le m$.
Note that the integer $n$ is bounded above by
\[
n\le\dim H_2(M,\boundary M;\RR)\le 2t,
\]
where the second inequality follows from the fact that there are at most $2t$ triangles of $\TTT$ that do not lie on $\boundary M$.
Let $F'=F_1+\ldots+F_n$.
We know that there is at least one linear combination of $[F_1],\ldots,[F_n]$ with positive coefficients that lies in the interior of $\cone{C'}$ -- namely, $[G]$.
Therefore, the class $[F']$ must also lie in the interior of $\cone{C'}$.
In particular, since $F'$ is lw\=/taut, this implies that each component of $F'$ must be a fibre of $M$.

Concerning the weight of $F'$, \zcref{thm:bound on weight of fundamental surfaces} gives the inequality
\[
w(F_i)\le t^2\cdot 2^{7t+7}\qquad\text{for every $1\le i\le n$},
\]
from which we get the bound
\[
w(F')=w(F_1)+\ldots+w(F_n)\le n\cdot t^2\cdot 2^{7t+7}\le t^3\cdot 2^{7t+8}.
\]
The Euler characteristic of $F'$ satisfies
\begin{alignaside*}
-\chi(F')&=x([F'])&&\aside{since $F'$ is lw\=/taut}\\
&=x([F_1]+\ldots+[F_n])\\
&\le x([F_1])+\ldots+x([F_m])&&\aside{by convexity of $x$}\\
&=x([G])&&\aside{by linearity of $x$ on $C_G$}\\
&=-\chi(G)&&\aside{since $G$ is lw\=/taut.}
\end{alignaside*}
We conclude the proof by taking $F$ to be a component of $F'$.
As previously observed, the surface $F$ is a fibre of $M$; its weight and Euler characteristic satisfy
\[
w(F)\le w(F')\le t^3\cdot 2^{7t+8}\qquad\text{and}\qquad\chi(F)\ge\chi(F')\ge\chi(G)=\chi(F_0),
\]
as desired.
\end{proof}

\egroup

\subsection{Toroidal fibres}
\label{sec:small toroidal fibres}

As we will see in the proof of \zcref{thm:fibredness detection is in NP}, recognising $3$\=/manifolds that fibre over the circle with fibre a sphere, disc, or annulus is already known to be in \NP{}.
Since \zcref{thm:fibre of small weight} covers $3$\=/manifolds with fibre of negative Euler characteristic, the only remaining case is when the fibre is a torus.
In this setting, the argument of \zcref{thm:fibre of small weight} does not work, since it heavily relies on the Thurston seminorm, which vanishes everywhere for these $3$\=/manifolds.
Fortunately, torus bundles are simple enough that we can completely classify their orientable non\=/separating incompressible surfaces.

Let $M$ be a compact oriented $3$\=/manifold that fibres over the circle with fibre a torus, and let $F_0\subs M$ a transversally oriented toroidal fibre of $M$.
Our goal is to show that all connected orientable incompressible surfaces in $M$ that are not homologically trivial are fibres of \emph{some} fibration of $M$ (that is, not necessarily the one defined by $F_0$).
Upon fixing a basis of $H_1(F_0;\ZZ)$, the monodromy $\varphi$ of $F_0$ is represented by an element of $\SL{2}{\ZZ}$.
It is well\=/known that if an element of $\SL{2}{\ZZ}$ has an integral eigenvector, then it is conjugate to
\[
\mathbf{L}_n^+=\begin{pmatrix}1&0\\n&1\end{pmatrix}\qquad\text{or}\qquad\mathbf{L}_n^-=\begin{pmatrix}-1&0\\n&-1\end{pmatrix}
\] 
for some $n\in\ZZ$.

\begin{proposition}[Incompressible surfaces in torus bundles]
\label{thm:incompressible surfaces in torus bundles}
In the setting above, let $F$ be a closed orientable incompressible non\=/separating surface embedded in $M$.
Then $F$ is a toroidal fibre of some fibration of $M$.
\end{proposition}
\begin{proof}
Since $F_0$ is a fibre of $M$, the $3$\=/manifold $M'=M\cut F_0$ is homeomorphic to $F_0\times[0,1]$.
In fact, it is useful to have a concrete model for $M'$ as $M'=\RR^2/\ZZ^2\times[0,1]$, and for $M$ as
\[
M=\quotient{\RR^2/\ZZ^2\times[0,1]}{\{(x,1)\sim(\varphi(x),0):x\in\RR^2/\ZZ^2\}},
\]
where $F_0=\RR^2/\ZZ^2\times\{0\}$.
Isotope $F$ so that it intersects $F_0$ transversely in the smallest possible number of curves.
Denote by $F'$ the intersection of $F$ with $M'$.
We now employ a very classical argument -- sketched, for instance, in the proof of \cite[Proposition~11.4.12]{martelli:introduction-geometric-topology} -- to show that $F'$ must be a collection of vertical annuli

First, we show that no component of $F\cap F_0$ can bound a disc in $F_0$.
For a contradiction, suppose that there is a disc $D$ embedded in $F_0$ such that $D\cap F=\boundary D$ (up to taking an innermost disc, this is not restrictive).
Since $F$ is incompressible, there must be a disc $D'$ embedded in $F$ such that $\boundary D'=\boundary D$.
By irreducibility of $M$, the sphere $D\cup D'$ must bound a $3$\=/ball $B$ in $M$; note that $F\cap B=D'$.
The $3$\=/ball $B$ can be used to isotope $D'$ to $D$, and then push it slightly past $D$, thus removing $\boundary D$ from the intersection $F\cap F_0$; this contradicts the minimality of the intersection.
By swapping the roles of $F$ and $F_0$ -- note that we only used the fact that $F$ is incompressible -- we conclude that no component of $F\cap F_0$ can bound a disc in $F_0$.
In particular, we have shown that $F\cap F_0$ is a collection of parallel essential curves in the torus $F_0$.

Next, we show that $F'$ is incompressible and boundary\=/incompressible in $M'$.
Let $D$ be a compressing disc for $F'$ in $M'$.
Since $F$ is incompressible in $M$, there is a disc $D'$ embedded in $F$ such that $\boundary D'=\boundary D$.
If the disc $D'$ is contained in $M'$, then we see that $D$ is trivial.
Otherwise, we find a component of $F\cap F_0$ that bounds a disc in $F$, the existence of which we have already ruled out; this proves that $F'$ is incompressible.
Let now $D$ be a boundary\=/compressing disc for $F'$ in $M'$.
Up to taking an innermost disc, we can assume that $D\cap F'\subs\boundary D$.
Let $a_1$ and $a_2$ the two components of $\boundary F'$ intersecting $D$; these are two parallel curves on a component of $\hboundary M'$ -- importantly, they are distinct because $F'$ is orientable.
Compress $F'$ along $D$ to obtain a new surface $F''$.
Denote by $a$ the boundary component of $F''$ that is obtained by surgery on $a_1\cup a_2$ along $D\cap\hboundary M'$.
Since $a_1$ and $a_2$ are parallel, we see that $a$ bounds a disc $D'$ in $\hboundary M'$.
It is not hard to show, using incompressibility of $F'$ and irreducibility of $M'$, that the component of $F''$ containing $a$ cobounds a $3$\=/ball in $M'$ with $D'$.
This means that the component of $F'$ containing $a_1$ and $a_2$ is an annulus cobounding a solid torus in $M'$ with an annulus in $\hboundary M'$.
This solid torus can be used to isotope said component of $F'$ onto $F_0$, and then slightly past it, thus removing $a_1$ and $a_2$ from the intersection $F\cap F_0$; this contradicts the minimality of the intersection.

In conclusion, we found that $F'$ is an orientable incompressible boundary\=/incompressible surface in $M'$.
It is well\=/known (see \cite[Proposition~9.3.18]{martelli:introduction-geometric-topology}) that any such surface is isotopic in $M'$ either to a union of \emph{horizontal} surfaces -- that is, surfaces of the form $F_0\times\{t\}$ for some $t\in[0,1]$ -- or to a union of \emph{vertical} surfaces -- that is, surfaces of the form $a\times[0,1]$ for some essential curve $a$ in $F_0$.
If $F'$ is horizontal, we deduce that $F$ is isotopic to $F_0$ in $M$, and we are done.
If $F'$ is vertical, then it is a union of annuli; in particular, this implies that there is a slope on $F_0$ that is preserved by $\varphi$.
In other words, the monodromy $\varphi$ has an integral eigenvector, and it is therefore conjugate to $\mathbf{L}_n^+$ or $\mathbf{L}_n^-$ for some integer $n$.
In fact, it is not restrictive to assume that $\varphi=\mathbf{L}_n^+$ or $\varphi=\mathbf{L}_n^-$.
In particular, this implies that each component of $F\cap F_0$ is a curve isotopic in $F_0$ to $\{0\}\times\RR/\ZZ\times\{0\}$.
By ``straightening'' $F'$, we can assume that each component of $F'$ is a flat annulus of the form
\[
\{(x,y,z)\in\RR^2\times[0,1]:\alpha\cdot x+\beta\cdot z=\gamma\}\subs\RR^2/\ZZ^2\times[0,1]
\]
for constants $\alpha,\beta,\gamma\in\QQ$.

\paragraph{When $\varphi=\mathbf{L}_n^+$.}
This case is depicted in \zcref{fig:incompressible surfaces in torus bundles:+}.
Let $T$ be the torus $\RR^2/\ZZ^2$, and consider the projection $\map{p}{M}{T}$ defined by
\[
p(x,y,z)=(x,z).
\]
It is easy to see that this projection defines a circle bundle structure over $T$ on $M$.
We see that $F_0=p^{-1}(\RR/\ZZ\times\{0\})$, while $F=p^{-1}(a)$ for some essential curve $a\subs T$.
Let $\map{\psi}{T}{T}$ be a homeomorphism that sends $a$ to $\RR/\ZZ\times\{0\}$.
This homeomorphism lifts to a homeomorphism of $M$ sending $F$ to $F_0$; since $F_0$ is a toroidal fibre of $M$, we conclude that the same holds for $F$.

\paragraph{When $\varphi=\mathbf{L}_n^-$.}
In this case, depicted in \zcref{fig:incompressible surfaces in torus bundles:-}, the $3$\=/manifold $M$ is a circle bundle over the Klein bottle $K$.
More precisely, if we model $K$ as the quotient of the Euclidean plane $\RR^2$ by the group of isometries generated by $(x,z)\mapsto(x+1,z)$ and $(x,z)\mapsto(-x,z+1)$, then the projection map $\map{p}{M}{K}$ is given by
\[
p(x,y,z)=(x,z).
\]
Like in the torus bundle case, we have that $F=p^{-1}(a)$ for some two\=/sided essential curve $a\subs K$.
However, unlike in the torus bundle case, there are only two such curves up to isotopy.
One is the meridian $m=\{(x,0)\in K\}$, and its preimage $p^{-1}(m)$ in $M$ is precisely $F_0$.
The other two\=/sided essential curve is
\[
b=\{(x,z)\in K:\text{$x=1/4$ or $x=3/4$}\},
\]
and it is separating in $K$; therefore, its preimage $p^{-1}(b)$ separates $M$.
Since we are assuming that $F$ is non\=/separating, we conclude that $F$ must be isotopic to $F_0$ in $M$; in particular, it is a fibre of $M$.
\end{proof}

\begin{myfigure}
\begin{subfigure}{.45\textwidth}
\centering
\myfiguresource[scale=.29]{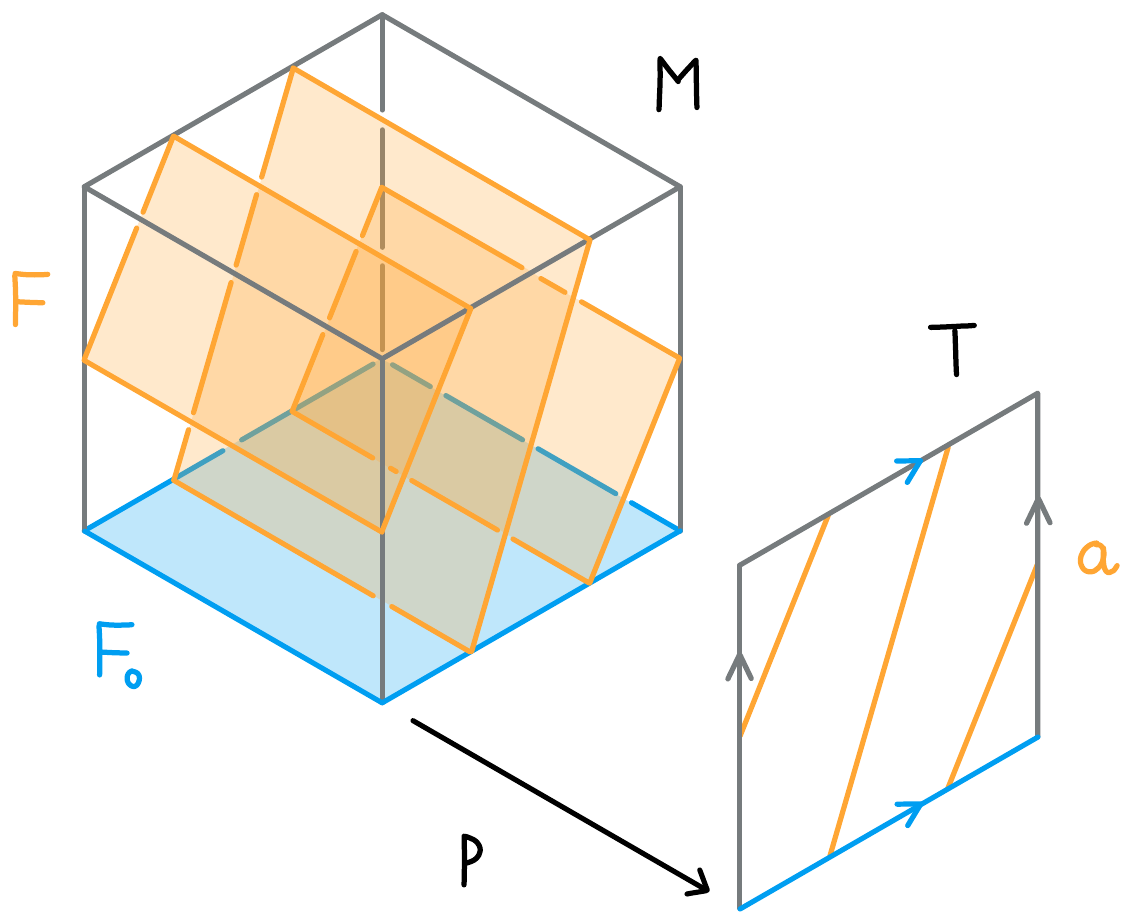}
\caption{}
\label{fig:incompressible surfaces in torus bundles:+}
\end{subfigure}
\hfill
\begin{subfigure}{.45\textwidth}
\centering
\myfiguresource[scale=.29]{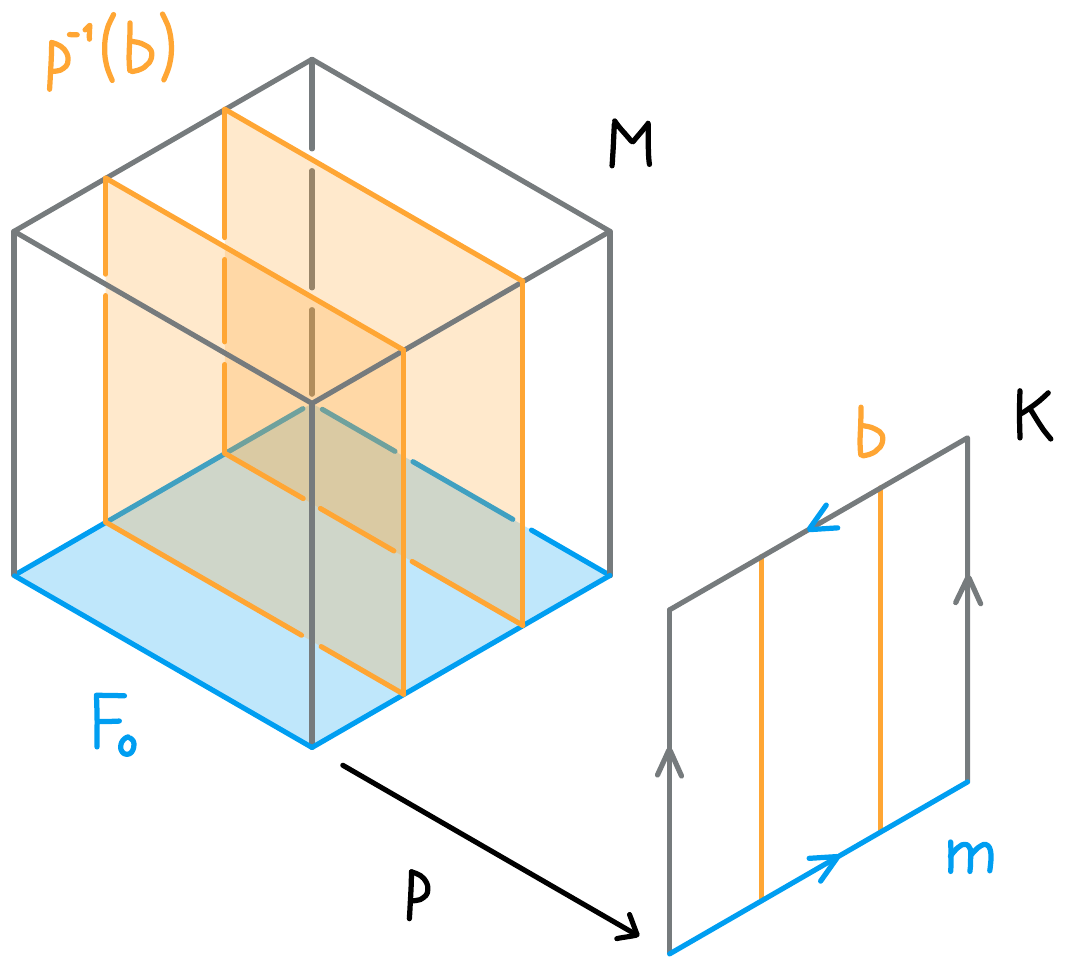}
\caption{}
\label{fig:incompressible surfaces in torus bundles:-}
\end{subfigure}
\caption{\subref{fig:incompressible surfaces in torus bundles:+} When the monodromy $\varphi$ is conjugate to $\mathbf{L}_n^+$ for some $n\in\ZZ$, the $3$\=/manifold $M$ is a circle bundle over the torus $T$; the surface $F$ is related to $F_0$ by a homeomorphism of $M$.
\subref{fig:incompressible surfaces in torus bundles:-} When the monodromy $\varphi$ is conjugate to $\mathbf{L}_n^-$ for some $n\in\ZZ$, the $3$\=/manifold $M$ is a circle bundle over the Klein bottle $K$; the surface $F$ is isotopic to $F_0$, since $p^{-1}(b)$ is separating.}
\end{myfigure}

As a consequence of the previous proposition, we find that triangulated torus bundles admit a normal fibre of bounded weight.

\begin{proposition}[A toroidal fibre of small weight]
\label{thm:toroidal fibre of small weight}
Let $\TTT$ be a triangulation of a compact connected orientable $3$\=/manifold $M$ with $t$ tetrahedra.
Suppose that $M$ fibres over the circle with fibre a torus.
Then there is a toroidal normal fibre $F$ of $M$ with
\[
w(F)\le t^2\cdot 2^{7t+7}.
\]
\end{proposition}
\begin{proof}
Let $F$ be a least\=/weight toroidal normal fibre of $M$.
We will prove that $F$ is fundamental; by \zcref{thm:bound on weight of fundamental surfaces}, this will imply the bound in the statement.
Suppose, for the sake of contradiction, that $F$ can be presented as a normal sum $F=G_1+G_2$, with $G_1$ and $G_2$ non\=/empty normal surfaces.
We can assume that $G_1$ and $G_2$ minimise the number of components of $G_1\cap G_2$; by \cite[Lemma~3.3.30]{matveev:algorithmic-topology-classification}, this guarantees that $G_1$ and $G_2$ are connected.

As we already observed at the start of the proof of \zcref{thm:fibre of small weight}, the fibre $F$ is lw\=/taut.
\citeauthor{tollefson-wang:taut-normal-surfaces}'s \zcref{thm:lw-taut normal surfaces} then implies that every normal surface carried by $C_F$ is lw\=/taut -- in particular, this holds for $G_1$ and $G_2$.
Therefore, the surfaces $G_1$ and $G_2$ are orientable, incompressible, and homologically non\=/trivial.
By \zcref{thm:incompressible surfaces in torus bundles}, we deduce that $G_1$ and $G_2$ are toroidal fibres of $M$; however, this contradicts the minimality of $F$.
\end{proof}

\subsection{Fibredness is in NP}

By combining \zcref{thm:fibre of small weight,thm:toroidal fibre of small weight}, we obtain that if a triangulated orientable $3$\=/manifold fibres over the circle, and the fibre is a torus or has negative Euler characteristic, then said $3$\=/manifold admits a fibre of exponential weight in the number of tetrahedra.
Since the size of our fibredness certificate depends logarithmically on the weight of the fibre, this is enough to certify fibredness in polynomial time.
Formally, we consider the following decision problem. 

\begin{algoproblem}[\problemname{Fibredness detection}]
\label{prb:fibredness detection}
\Input{a triangulation of a compact connected oriented $3$\=/manifold $M$.}
\Output{whether $M$ fibres over the circle.}
The size of the input is measured by the number of tetrahedra in the triangulation of $M$.
\end{algoproblem}

\begin{theorem}[{\nameref{prb:fibredness detection} is in \NP}]
\label{thm:fibredness detection is in NP}
The problem \nameref{prb:fibredness detection} is in \NP{}.
\end{theorem}
\begin{proof}
The problems of recognising compact connected orientable $3$\=/manifolds that fibre over the circle with fibre of non\=/negative Euler characteristic and not a torus have been already addressed in the literature.
Specifically, the following recognition problems are known to be in \NP{}:
\begin{itemize}
\item recognising $S^2\times S^1$, by \cite[Theorem~3]{ivanov:computational-complexity-basic};
\item recognising $D^2\times S^1$, by \cite[Theorem~3]{ivanov:computational-complexity-basic};
\item recognising $S^1\times[0,1]\times S^1$, by \zcref{thm:interval bundle recognition is in NP} (that is, \cite[Theorem~12.1]{lackenby:efficient-certification-knottedness}, with the caveat that, since we are trying to recognise an interval bundle over the closed surface $S^1\times S^1$, the discussion at the beginning of \cite[Section~11]{lackenby-schleimer:recognising-elliptic-manifolds} is necessary).
\end{itemize}
Since these are only a finite number of cases, they can each be handled by their respective verification algorithms, and we can therefore focus only on the case where $M$ fibres over the circle with fibre of negative Euler characteristic or a torus.

The verification algorithm takes as input a triangulation $\TTT$ of a compact connected oriented $3$\=/manifold $M$, and a certificate consisting of a transversely oriented normal surface $F$ in $M$ and a certificate $\Sigma\in\cert[fib](M,F)$.
The algorithm then verifies that $F$ is connected and orientable, using \zcref{thm:counting components of a normal surface,thm:transverse orientation of a normal surface}, and that $\Sigma\in\cert*[fib](M,F)$, using the algorithm of \zcref{thm:verification of fibredness certificate}.
These verifications can be performed in polynomial time in $\card{\TTT}$ and $\card{\Sigma}$.
If these checks are successful, then the third statement of \zcref{thm:correctness of fibredness certificate} implies that $M$ fibres over the circle with fibre $F$.
Conversely, \zcref{thm:fibre of small weight,thm:toroidal fibre of small weight,thm:existence of fibredness certificate} guarantee that, if $M$ fibres over the circle with fibre of negative Euler characteristic or a torus, then there exist a transversely oriented connected normal surface $F$ in $M$ and a certificate $\Sigma\in\cert*[fib](M,F)$ such that $\card{\Sigma}$ is bounded above by a polynomial in $\card{\TTT}$.
\end{proof}

\printbibliography[heading=bibintoc]

\end{document}